\documentclass{article}
\usepackage[utf8]{inputenc}
\usepackage{xcolor}
\usepackage{amssymb}
\usepackage{tikz-cd}
 \usepackage{amscd,amsmath}

\usepackage{amssymb}
\usepackage{amsthm}
\usepackage{wasysym}
\DeclareMathOperator{\End}{End}
\DeclareMathOperator{\Dom}{Dom}
\DeclareMathOperator{\kat}{kat}
\DeclareMathOperator{\Kat}{Kat}

\DeclareMathOperator{\Endog}{Endog}

\def\Ind#1#2{#1\setbox0=\hbox{$#1x$}\kern\wd0\hbox to 0pt{\hss$#1\mid$\hss}
\lower.9\ht0\hbox to 0pt{\hss$#1\smile$\hss}\kern\wd0}

\def\notind#1#2{#1\setbox0=\hbox{$#1x$}\kern\wd0
\hbox to 0pt{\mathchardef\nn=12854\hss$#1\nn$\kern1.4\wd0\hss}
\hbox to 0pt{\hss$#1\mid$\hss}\lower.9\ht0 \hbox to 0pt{\hss$#1\smile$\hss}\kern\wd0}

\usepackage{amsmath} 
\usepackage{amssymb}
\usepackage{tikz-cd}
 \usepackage{amscd,amsmath}
\usepackage{subfigure}
\usepackage{csquotes} 
\newtheorem{theorem}{Theorem}[section]
\newtheorem{question}{Question}[section]

\newtheorem{lemma}[theorem]{Lemma}

\newtheorem*{defn}{Definition}
\newtheorem*{thmA}{THEOREM A}
\newtheorem*{thmB}{THEOREM B}
\newtheorem*{thmB,2}{Theorem B, Second version}
\usepackage{booktabs}
\usepackage{caption}
\usepackage{amssymb}
\usepackage{hyperref}
\hypersetup{colorlinks,allcolors=black}
\usepackage{tikz-cd}
\usepackage{mathptmx} 
\usepackage{graphicx} 
\usepackage{amsthm}
\title{Endogenies and linearization in the non-virtually connected case}
\author{Moreno Invitti}
\begin{document}
\maketitle
\begin{abstract}
    We prove a linearization theorem for pre-rings of endogenies acting on a definable abelian group of finite dimension. Observe that no assumptions on the connectivity of $A$ are made. We also prove a similar result when one of the two pre-rings is of quasi-endomorphisms. A corollary of these results is a generalization of Zilber's Field Theorem for finite-dimensional theories.
\end{abstract}
\section{Introduction}
The present work characterizes, under the model-theoretic hypothesis of finite-dimensionality,
irreducible bi-modules, i.e., abelian groups together with two commuting subrings of endomorphisms, and which are "irreducible" for the bi-action.\\
This result is well-known for groups of finite Morley rank \cite{zilber1984some}. It is also known for $o$-minimal theories (see, for example \cite{peterzil2000simple}). The main problem is that the two proofs use techniques proper to the two families of theories involved (in the finite Morley rank case, the indecomposability theorem is necessary). A natural question arises: to which generality can we extend the linearization?\\
An interesting family of theories is finite-dimensional theories, in the sense of \cite{wagner2020dimensional}.
\begin{defn}
    A theory $T$ is said to be \emph{finite-dimensional} if there exists a function $\operatorname{dim}$ from the class of all interpretable subset in any model $\mathcal{M}$ of $T$ into $\omega\cup\{-\infty\}$ such that for any $\phi(x,y)$ formula, $X,Y$ interpretable sets in $T$ and $f$ interpretable function from $X$ to $Y$,
\begin{itemize}
    \item If $a,a'$ have same type over $\emptyset$, $\operatorname{dim}(\phi(x,a))=\operatorname{dim}(\phi(x,a'))$ 
    \item $\operatorname{dim}(\emptyset)=-\infty$ and $\operatorname{dim}(X)=0$ if and only if $X$ is finite;
    \item $\operatorname{dim}(X\cup Y)=\max\{\operatorname{dim}(X),\operatorname{dim}(Y)\}$;
    \item If $\operatorname{dim}(f^{-1}(y))\geq k$ if for any $y\in Y$, then $\operatorname{dim}(X)\geq \operatorname{dim}(Y)+k$;
    \item If $\operatorname{dim}(f^{-1}(y))\leq k$ for any $y\in Y$, then $\operatorname{dim}(X)\leq \operatorname{dim}(Y)+k$.
\end{itemize}
\end{defn}
Examples of finite-dimensional theories are $o$-minimal theories, supersimple theories of finite Lascar rank, and supperosy theories of finite $U^{\text{\thorn}}$-rank.\\
 A first proof of Zilber's Field Theorem in this context has been given by Deloro in \cite{deloro2024zilber}. An important hypothesis of the theorem is that the module on which the ring of endomorphisms acts is connected. This clearly holds in the finite Morley rank and $o$-minimal context, but in general is not true: it is sufficient to take a supersimple theory of finite Lascar rank. Therefore, it remains unclear what happens for non virtually connected modules. The aim of this article is exactly to analyse this case.\\
 In reality, we will work with two more general notions of endomorphisms: endogenies and quasi-endomorphisms. An endogeny $\phi$ of an abelian group $A$ is a subgroup of $A\times A$ such that $\pi_1(\phi)=A$ and $\{a\in A:\ (0,a)\in \phi\}$ is finite. The latter subgroup is called the \emph{katakernel} of $\phi$. A \emph{quasi-endomorphism} $\phi$ is a subgroup of $A$ such that $\pi_1(\phi)$, that is called the domain of $\phi$, is of finite index in $A$ and the katakernel is finite. Therefore, an endogeny is a quasi-endomorphism that is total.\\
Endogenies and quasi-endomorphisms arise naturally in the finite-dimensional context. Indeed, let $H\geq H_1$ be definable subgroups of the definable abelian group $G$ and $\Gamma\leq \operatorname{End}(G)$. Assume that $H$ and $H_1$ are \emph{$\Gamma$-almost invariant} \hbox{i.e.} $\gamma(H)\cap H$ is of finite index in $\gamma(H)$ for any $\gamma\in \Gamma $. Then, $\Gamma$ acts on $G/H$ by endogenies (with katakernel of $\gamma$ equal to $\gamma(H)+H/H$) and $\Gamma$ acts by quasi-endomorphisms on $H/H_1$ with domain $\{h\in H:\ \gamma(h)\in H\}$. In the finite Morley rank context and in the $o$-minimal one, we can avoid working with endogenies or quasi-endomorphisms since, if $H$ is $\Gamma$-almost invariant, the connected component $H^0$ is $\Gamma$-invariant and the action of $\Gamma$ on $G/H^0$ is again by endomorphisms.\\
The article follows the ideas of \cite{WagnerDeloro}, which proves the linearization theorem for endogenies in the connected case.\\
In the first section, we introduce all the definitions that we need in the article, and we verify some easy lemmas.\\
Then, we proceed to prove the following result.
\begin{thmA}
In a finite-dimensional theory, let $A$ be an abelian definable group and $\Gamma,\Delta$ two invariant pre-rings of definable endogenies such that:
\begin{itemize}
    \item $\Gamma,\Delta$ sharply commute;
    \item $A$ is absolutely $(\Gamma,\Delta)$-minimal;
    \item $\Gamma$ is essentially unbounded and $\Delta$ is essentially infinite or vice-versa.
\end{itemize}
    There is a finite $\Gamma$ and $\Delta$-invariant subgroup $F$ such that $\Gamma/{\sim}$ and $\Delta/{\sim}$ act by endomorphisms on $A/F$.
\end{thmA}
Moreover, $F$ will be given explicitly.\\
In the second section, we verify the easy case in which both $\Gamma$ and $\Delta$ act by endogenies with finite kernel or finite image.\\
The third section analyzes the case in which the unbounded pre-ring of endogenies has an endogeny with infinite kernel and infinite image. Finally, we complete the proof in the fourth section.\\
The main result of the fifth section is the following linearization theorem.
\begin{thmB}
    Let $A$ be an abelian definable group and $\Gamma,\Delta$ two invariant pre-rings of definable endogenies such that 
    \begin{itemize}
        \item $C^{\#}(\Gamma)=\Delta$ and $C^{\#}(\Delta)=\Gamma$;
        \item $A$ is absolutely $(\Gamma,\Delta)$-minimal;
        \item $\Gamma$ is essentially unbounded and $\Delta$ is essentially infinite or vice-versa;
        \item $A_0=\Kat(\Gamma,\Delta)$.
    \end{itemize}
    Then, there exists $K$ a definable infinite field such that $A/A_0$ is a finite-dimensional $K$-vector space contained into $\big(\Gamma/{\sim}\big)\cap \big(\Delta/{\sim}\big)$ and $\Gamma/{\sim}$ and $\Delta/{\sim}$ act $K$-linearly on $A/A_0$.
\end{thmB}
\textbf{Remark}:\\
These two theorems are not direct extensions of the Theorems in \cite{WagnerDeloro}. Indeed, we do not assume $A$ minimal but \emph{absolutely minimal} \hbox{i.e.} there exists no almost invariant infinite definable subgroups of infinite index. In reality, if we assume connectivity, our proofs can also be applied in the minimal case and so, in this sense, the results are an extension of the connected case.\\
In the sixth section, we introduce some basic results for quasi-endomorphisms. Then, in sections $7,8,9$, we prove an extension of Theorem $A$ in the case one of the two pre-rings $\Gamma,\Delta$ is of quasi-endomorphisms and the other of endogenies.\\
In section $10$, we prove a version of Theorem $A$ and $B$ in the case that $A$ has finite $n$-torsion for every $n<\omega$. Finally, in section $11$, we prove a generalization of Zilber's Field Theorem in finite-dimensional theories.
\subsection{Finite-dimensional groups}
In this subsection, we recall the main results on finite-dimensional groups, in particular Lemma \ref{boundedind}, which will be fundamental in the proof of Theorem A and Theorem B.
\begin{lemma}\label{boundedind}
    Let $G$ be a definable group of finite dimension and $\{H_i\}_{i\in I}$ a family of uniformly definable subgroups. Then, there exists $n,d<\omega$ such that there is no $J=\{j_1,...,j_n\}\subseteq I$ of cardinality $n$ such that $|\bigcap_{i=1}^kH_{j_i}:\bigcap_{i=1}^{k+1} H_{j_i}|$ has index greater than $d$ for any $k\leq n-1$.
\end{lemma}
\begin{proof}
Let $n=\dim(G)$. We will work in a sufficiently saturated structure $\mathfrak{M}$ in which $G$ is definable and assume that $\phi(x,y)$ is the formula defining the family. Assume, by contrary, that the conclusion is false then for $N=n+2$ and for any $k\in \omega$, there exists $g_1,...,g_N$ such that $|\bigcap_{j\leq k-1} \phi(\mathfrak{M},g_j)/\bigcap_{j\leq k} \phi(\mathfrak{M},g_j)|\geq k$. Therefore, the partial type given by the formulas 
$$\exists a^1_1,...,a^1_k,...,a^N_1,...,a_k^N: a^i_j\in \bigcap_{j=1}^i \phi(\mathfrak{M},x_j)\wedge a^i_j{a^i_k}^{-1}\not\in \phi(\mathfrak{M},x_j)\wedge \forall_{i\leq N}\ \phi(\mathfrak{M},x_i)\leq G(\mathfrak{M})$$
for all $k<\omega$, is finitely satisfable. By compactness and saturation, there exist subgroups $\{H_i\}_{i\leq N}$ such that $|\bigcap_{i<j}H_i/\bigcap_{i\leq j}H_i|$ is infinite for every $i$. This would imply that the dimension of $G$ is strictly greater than $n$, a contradiction.
\end{proof}
If we apply these results to the centralizers of elements in $G$, we obtain that a finite-dimensional definable group has the hereditary $\widetilde{\mathfrak{M}}_c$-property.
\begin{defn}
    A group $G$ is \emph{hereditarily-$\widetilde{\mathfrak{M}}_c$} if for any definable subgroups $H,N$, such that $N$ is normalised by $H$, there exist natural numbers $n_{HN}$ and $d_{HN}$ such that any sequence of centralisers 
    $$C_H(a_0/N)\geq C_H(a_0,a_1/N)\geq...\geq C_H(a_0,a_1,...,a_n/N)\geq...$$
    with each centraliser of index at least $d_{HN}$ in the previous one, has length at most $n_{HN}$.
\end{defn}
The notions of almost containment and commensurability are fundamental in the study of groups definable in finite-dimensional theories.
\begin{defn}
\begin{itemize}
    \item  Let $G$ be a group and $A, B\leq G$ subgroups. $A$ \emph{almost contains} $B$, denoted $A\apprge B$, if $|B:B\cap A|$ is finite.
    \item  If $A\apprle B$ and $B\apprle A$, $A$ and $B$ are \emph{commensurable}.
    \item Given a family of subgroups $\{G_i\}_{i\in I}$ in $G$, the family $\{G_i\}_{i\in I}$ is \emph{uniformly commensurable} if there exists $n<\omega$ such that $|G_i/G_i\cap G_j|\leq n$ for every $i,j\in I$.
\end{itemize}
  \end{defn}
\section{Endogenies}
\subsection{Definition}
We will define the notion of endogeny of a group $A$.
\begin{defn}
    Let $A$ be an abelian group. A subgroup $\phi$ of $A\times A$ is called an \emph{endogeny} if it is total ($\pi_1(\phi)=A$ with $\pi_1:A\times A\to A$ the projection on the first component) and the subgroup $\{a\in A:\ (0,a)\in \phi\}$, called the \emph{katakernel} of $\phi$, is finite. The katakernel of $\phi$ is denoted by $\kat(\phi)$.\\
    An endogeny is uniquely determined by the homorphism 
    $$\phi: A\mapsto A/\kat(\phi) \ \ \ \ \ \ \  a\mapsto \phi[a]+\kat(\phi)=\{b\in A:\ (a,b)\in \phi\}$$
    We use the notation with the square parentheses to underline that $\phi[a]$ is a finite subset in $A$ and not an element of $A$.\\
    The set of endogenies of $A$ is denoted by $\Endog(A)$.
\end{defn}
Endogenies have a predominant role in the analysis of $\omega$-categorical groups, as showed by \cite{evans2000supersimple},\cite{dobrowolski2020omega} and \cite{baur1979totally}.\\
We define two operations on $\Endog(A)$ called sum and product:
\begin{defn}
    Let $A$ be an abelian group and $\phi,\psi$ two endogenies of $A$. We define:
    \begin{itemize}
        \item $\phi+\psi$ as the endogeny that sends $a\in A$ into $\phi[a]+\psi[a]$. From a set-wise point of view, the sum is given by the subgroup 
        $$\{(a,b):\ \exists b_1,b_2\ (a,b_1)\in \phi\wedge(a,b_2)\in \psi\wedge (b_1+b_2)=b\}.$$
        The katakernel of the sum is $\kat(\phi)+\kat(\psi)$.
        \item $\phi\circ \psi$ as the composition of relations and so $\phi\circ \psi[a]=\phi[\psi[a]]$. From a set-wise point of view, the composition is given by the subgroup 
        $$\{(a,b):\ \exists  z\in A\ (a,z)\in \psi\wedge (z,b)\in \phi\}.$$
        The katakernel of $\phi\circ \psi$ is $\phi[\kat(\psi)]$. The $\circ$ is usually omitted.
    \end{itemize}
\end{defn}
The structure $(\Endog(A),+,\circ)$ is not a ring but only a pre-ring.
\begin{defn}
    A \emph{pre-ring} is a set $X$ with two binary operations $+,\cdot$ and a constant $0$ such that:
    \begin{itemize}
        \item $(X,+,0)$ is an abelian group;
        \item $(X,\cdot,1)$ is a monoid;
        \item The left distributivity property holds \hbox{i.e.} $z\cdot (x+y)=z\cdot x+z\cdot y$ for every $x,y,z\in X$.
        \item $0\cdot x=x\cdot 0=0$ for any $x\in X$.
    \end{itemize}
\end{defn}
\begin{lemma}\label{Lemma 1}
    The structure $(\Endog(A),+,\cdot,0,Id)$ is a pre-ring. The right distributivity holds up to a finite error set. Indeed,
    $$(\phi\delta+\psi\delta)[a]\leq (\phi+\psi)\delta[a]+\phi[\kat(\delta)]=(\psi+\phi)\delta[a]+\psi[\kat(\delta)].$$
\end{lemma}
\begin{proof}
    It is quite obvious that $(\Endog(A),+,0)$, with $0$ the $0$-endomorphism, is a commutative group and $(\Endog(A),\cdot,Id)$ is a monoid.\\
    We verify that the left distributivity holds: given $\phi,\psi,\delta\in \Endog(A)$, then 
    $$\delta(\phi+\psi)[a]=\{c\in A:\ \exists b_1,b_2\  (a,b_1)\in\psi\wedge(a,b_2)\in \phi\wedge (b_1+b_2,c)\in \delta\}.$$
    Therefore, every element $c\in \delta(\phi+\psi)[a]$ is given by $c_1+c_2$ such that $(b_i,c_i)\in \delta$ for $i=1,2$. Then, $c=c_1+c_2\in \delta\phi[a]+\delta\psi[a]$. The vice versa is similar.\\
    For right distributivity, observe that
    $$(\phi+\psi)(\delta)[a]=\{c\in A: \exists b,c_1,c_2\ (a,b)\in \delta\wedge(b,c_1)\in \phi\wedge(b,c_2)\in \psi\wedge c=c_1+c_2\}.$$
    On the other hand,
    $$(\phi\delta+\psi\delta)[a]=\{c\in A:\ \exists c_1,c_2,b_1,b_2\ (a,b_1),(a,b_2)\in \delta\wedge (b_1,c_1)\in \phi\wedge (b_2,c_2)\in \psi\wedge c_1+c_2=c\}.$$
    The second set clearly contains the first, but the vice versa is not always true. On the other hand
    $$\phi\delta+\psi\delta[a]\leq (\phi+\psi)\delta[a]+\phi[\kat(\delta)]=(\psi+\phi)\delta[a]+\psi[\kat(\delta)].$$
    Since $b_1-b_2\leq \kat(\delta)$, then $\psi\delta[b_2]\leq \psi[\delta[b_1]+\kat(\delta)]\leq \psi[\delta[b_1]]+\psi[\kat(\delta)]$.
\end{proof}
To overcome this problem, we introduce an equivalence relation on $\Endog(A)$. 
\begin{defn}
    Let $A$ be an abelian group and $\phi,\psi$ two endogenies of $A$. Then, $\phi$ is \emph{equivalent} to $\psi$, denoted by $\phi\sim \psi$, iff $(\phi-\psi)[A]$ is finite.
\end{defn}
We verify that $\sim$ is an equivalence relation:
\begin{itemize}
    \item $\phi\sim\phi$ since $\phi-\phi$ is equal to the function that sends any $a$ in $\kat(\phi)$ that is finite.
    \item If $\phi\sim\psi$, then $\psi\sim \phi$. Indeed, if $\operatorname{Im}(\phi-\psi)$ is finite, then also $\operatorname{Im}(-(\phi-\psi))=\operatorname{Im}(\psi-\phi)$ is finite.
    \item If $\phi\sim \psi$ and $\psi\sim \delta$, then 
    $$(\phi-\delta)[a]=\{c\in A:\ \exists c_1,c_2\ (a,c_1)\in \phi\wedge (a,c_2)\in \delta\wedge c_1-c_2=c\}.$$
    Given $d\in \psi[a]$, then $c_1-d\in \operatorname{Im}(\phi-\psi)$ and $c_2-d\in \operatorname{Im}(\psi-\delta)$. Therefore, $c=c_1-c_2=c_1-d+d-c_2\leq \operatorname{Im}(\phi-\psi)+\operatorname{Im}(\psi-\delta)$ that is finite. 
\end{itemize}
The proof that ($\Endog(A)/{\sim},+,\circ$) is a ring is immediate. 
\begin{lemma}\label{Lemma 2}
    Let $A$ be an abelian group. $(\Endog(A)/{\sim},+,\circ,0,1)$ is a ring.
\end{lemma}
\begin{proof}
    It is sufficient to verify that the operations pass to the quotient since the left distributivity holds: by Lemma \ref{Lemma 1}, $(\phi+\psi)\delta-(\phi\delta-\psi\delta)[A]\leq \psi[\kat(\delta)]$ that is finite. The other properties hold again by Lemma \ref{Lemma 1}.\\
    $+$ passes to the quotient since if $\phi\sim \phi'$ and $\psi\sim \psi'$ then $\phi+\psi\sim \phi'+\psi'$ indeed $(\phi+\psi-\phi'-\psi)[a]=(\phi-\phi')[a]+(\psi-\psi')[a]$ that is finite.\\
    $\circ$ passes to the product \hbox{i.e.} $\phi\psi\sim\phi'\psi'$ since $\phi\psi[a]-\phi'\psi'[a]$ is equal to $\phi\psi[a]+b-b+\phi'\psi'[a]$ for $b\in \phi\psi'[a]$. Up to equivalence, 
    $$\phi\psi[a]-\phi\psi'[a]=\phi[\psi[a]-\psi'[a]]\leq \phi[\operatorname{Im}(\psi-\psi')].$$
    The latter is a fixed finite subgroup. The same for the term $\phi\psi'[a]-\phi'\psi[a]$. 
\end{proof}
Other objects that will be necessary in the proof of the two theorems are \emph{quasi-endomorphisms}.
\begin{defn}
    Let $A$ be an abelian group. A \emph{quasi-endomorphism of $A$} is an additive subgroup $\phi$ of $A\times A$ such that $\Dom(\phi)=\{a\in A:\ \exists b\in A (a,b)\in \phi\}$ is a subgroup of finite index in $A$ and $\kat(\phi)=\{a\in A:\ (0,a)\in \phi\}$ is finite. The set of quasi-endomorphisms of $A$ is denoted by $\mathcal{F}-\End(A)$.
\end{defn}
\subsection{Sharp commutation}
We introduce the first ingredient of Theorem A: the sharp commutation between endogenies.
\begin{defn}
    Let $A$ be an abelian group and $\phi,\psi$ two endogenies. $\phi$ and $\psi$ \emph{sharply commute} if $(\phi\psi-\psi\phi)[A]\leq \kat(\phi)+\kat(\psi)$.\\
    Two subsets $X,Y$ of $\Endog(A)$ \emph{sharply commute} if every $\phi\in X$ and $\psi\in Y$ sharply commute.\\
    Given $X$ a subset of $\Endog(A)$, the subset $C^{\#}(X)$ of $\Endog(A)$ is the set 
    $$\{\phi\in \Endog(A):\ \phi \text{ and $\psi$ sharply commute}\}.$$
\end{defn}
We verify that $C^{\#}(X)$ is a pre-subring of $\Endog(A)$ for every $X\subseteq \Endog(A)$.
\begin{lemma}\label{Lemma 3}
    Let $A$ be an abelian group and $X$ a subset of $\Endog(A)$. Then,
    \begin{itemize}
        \item If $\phi$ and $\psi$ sharply commute, $\phi[\kat(\psi)]\leq \kat(\phi)+\kat(\psi)$;
        \item $C^{\#}(X)$ is a pre-subring of $\Endog(A)$.
    \end{itemize}
\end{lemma}
\begin{proof}
   We prove the first point: $0\in \psi[\kat(\phi)]$ and so 
   $$\psi[\kat(\phi)]=\psi[\kat(\phi)]-0\leq \operatorname{Im}(\psi\phi-\phi\psi)\leq \kat(\phi)+\kat(\psi).$$
   For the second point, since $\Endog(A)$ is a pre-ring by Lemma \ref{Lemma 1}, it is sufficient to verify that $C^{\#}(X)$ is closed for sum, opposite, and composition.\\
   Let $\phi,\psi$ in $C^{\#}(X)$. Then,  for any $x\in X$, $((\phi-\psi)x-x(\phi-\psi))[a]$ is equal to $(\phi x-\psi x-x\psi+x\phi)[a]$ up to the finite term $x[\kat(\phi)]\leq \kat(x)+\kat(\phi)$ by point $1$. Applying commutativity, we obtain $(\phi x-x\phi)[a]+(x\psi-\psi x)[a]$, that is contained, by sharp commutation, into $\kat(x)+\kat(\phi)+\kat(\psi)=\kat(x)+\kat(\phi+\psi)$. For the composition, given $(\phi\psi x-x\phi\psi)[a]$ and $b\in \phi x\psi[a]$, we proceed as before.
\end{proof}
\subsection{Invariance}
We introduce three different notions of invariance, and we analyze their relations with sharp commutation. 
\begin{defn}
  Let $\Gamma$ be a subset of $\Endog(A)$. Then,
  \begin{itemize}
      \item $B\leq A$ is \emph{$\Gamma$-invariant} if $\gamma[A]\leq A$ for any $\gamma\in \Gamma$;
      \item $B\leq A$ is \emph{ weakly $\Gamma$-invariant} if $\gamma[A]\leq A+\kat(\gamma)$ for any $\gamma\in \Gamma$;
      \item $B\leq A$ is \emph{almost $\Gamma$-invariant} if $\gamma[A]$ is almost contained in $A$ for every $\gamma\in \Gamma$ \hbox{i.e.} $\gamma[A]\cap A$ is of finite index in $\gamma[A]$.
  \end{itemize} 
\end{defn}
We verify that almost $\Gamma$-invariance and weak $\Gamma$-invariance behave well under the action of an endogeny that sharply commutes with $\Gamma$.
\begin{lemma}\label{Lemma 4}
    Let $A$ be an abelian group, $\Gamma$ a pre-ring of endogenies and $\delta\in C^{\#}(\Gamma)$. If $B\leq A$ is weakly(almost) $\Gamma$-invariant, then $\delta[B]$ is weakly(almost) $\Gamma$-invariant.
\end{lemma}
\begin{proof}
    We start from weak invariance. It is sufficient to observe that $\gamma(\delta[B])\leq \delta\gamma[B]+\kat(\gamma)$ by sharp commutation. By weak invariance and Lemma \ref{Lemma 3}, this is contained in
    $$\delta[B]+\delta[\kat(\gamma)]+\kat(\gamma)\leq \delta[B]+\kat(\gamma).$$
    For the almost invariance, $\gamma[\delta[B]]$ is contained in $\delta[\gamma[B]]+\kat(\delta)$ by sharp commutation. Since $B\cap \gamma[B]$ is of finite index in $\gamma[B]$, also $\delta[B\cap \gamma[B]]$ is of finite index in $\delta[\gamma[B]]$. $\delta[B\cap \gamma[B]]$ is contained in $\delta[B]\cap \gamma[\delta[B]]$. Therefore, $\delta[B]\cap \gamma[\delta[B]]$ is of finite index in $\gamma[\delta[B]]$. This verifies that $\delta[B]$ is almost $\Gamma$-invariant.
\end{proof}
We verify that we can define a new ring of endogenies on weakly $\Gamma$-invariant subgroups through restriction-corestriction.
\begin{defn}
    Let $A$ be an abelian group, $B\leq A$ a subgroup, and $\phi$ an endogeny. Then, the \emph{restriction-corestriction of $\phi$ to $B$}, denoted by $\phi_B$, is the subgroup $\phi\cap B\times B$ of $B\times B$.
\end{defn}
In general, the restriction-corestriction is not an endogeny of $B$ (since it is not total). We verify that the restriction-corestriction is well-defined when $B$ is a weakly invariant subgroup.
\begin{lemma}\label{Lemma 5}
    Let $A$ be an abelian group, $\phi\in \Endog(A)$ and $B\leq A$ a weakly $\{\phi\}$-invariant subgroup. Then, $\phi_B$ is an endogeny of $B$.
\end{lemma}
\begin{proof}
    It is sufficient to verify that the endogeny $\phi_B$ is total since it is obviously an additive subgroup of $B\times B$ whose katakernel is $\kat(\phi)\cap B$ that is finite.\\
    It is total if, for every $b\in B$, there exists $b'\in B$ such that $b'\in \phi[b]$. This happens iff $\phi[B]\leq B+\kat(\phi)$.
\end{proof}
We also introduce three different notions of minimality.
\begin{defn}
Let $A$ be a definable abelian group and $\Gamma,\Delta$ two pre-rings of definable endogenies.
\begin{itemize}
    \item  $A$ is \emph{$\Gamma$-minimal} if there is no infinite definable $\Gamma$-invariant subgroup of infinite index in $A$;
    \item $A$ is \emph{weakly $\Gamma$-minimal} if there exists no infinite definable weakly $\Gamma$-invariant subgroup of infinite index in $A$;
    \item $A$ is \emph{absolutely $\Gamma$-minimal} is there is no infinite definable almost $\Gamma$-invariant subgroups of infinite index in $A$.
\end{itemize}
    $A$ is \emph{(absolutely/weakly) $(\Gamma,\Delta)$-minimal} if it is (absolutely/weakly) $\Gamma\cup\Delta$-minimal. 
\end{defn}
\subsection{Global katakernel}
We introduce the notion of global katakernel for a pre-subring of endogenies.
\begin{defn}
    Let $A$ be an abelian group and $\Gamma$ a pre-ring of endogenies. The \emph{global katakernel of $\Gamma$}, denoted by $\Kat(\Gamma)$, is $\sum_{\gamma \in \Gamma} \kat(\gamma)$.\\
    If $\Gamma,\Delta$ are two pre-rings of endogenies acting on $A$, the \emph{global bikatakernel $\Kat(\Gamma,\Delta)$} is simply the sum $\Kat(\Delta)+\Kat(\Gamma)$.
\end{defn}
The global katakernel $\Kat(\Gamma)$ is clearly the smallest subgroup such that $\Gamma$ acts by endomorphisms on the quotient of $A$ with the subgroup. The properties of global katakernels are resumed in the following Lemma.
\begin{lemma}\label{Lemma 6}
    Let $A$ be an abelian group and $\Delta,\Gamma$ two pre-rings of endogenies sharply commuting. Then,
    \begin{itemize}
        \item The global katakernel $\Kat(\Gamma)$ is weakly $\Delta$-invariant and $\Gamma$-invariant;
        \item the global bikatakernel $\Kat(\Gamma,\Delta)$ is $(\Delta,\Gamma)$-invariant;
        \item if $A_0\supseteq \Kat(\Gamma,\Delta)$ is weakly $\Delta$-invariant, then $\gamma^{-1}(A_0)$ is weakly $\Delta$-invariant for any $\gamma\in \Gamma$.
    \end{itemize}
\end{lemma}
\begin{proof}
For the first observe that given $\gamma\in \Gamma$ and $a\in \Kat(\Gamma)$ then $a\in \kat(\gamma_1)+...+\kat(\gamma_n)$ for $\gamma_i\in \Gamma$. Then $\gamma[a]\in \kat(\gamma(\gamma_1+...+\gamma_n))\leq \kat(\Gamma)$ since $\gamma(\gamma_1+...+\gamma_n)\in \Gamma$. Let $a\in \kat(\gamma_1)+...+\kat(\gamma_n)$ for $\gamma_i\in \Gamma$. Then, by Lemma \ref{Lemma 3},
$$\delta[a]\leq \delta[\kat(\gamma_1+...+\gamma_n)]\leq \kat(\gamma_1+...+\gamma_n)+\kat(\delta).$$
In conclusion, $\delta[\Kat(\Gamma)]\leq \Kat(\Gamma)+\kat(\delta)$.\\
The second point is an obvious consequence of the first one.\\
For the third point, by Lemma \ref{Lemma 3},
$$\gamma[\delta[A]]\leq \delta[\gamma[A]]+\kat(\gamma)\leq \delta[A]+\delta[\kat(\gamma)]+\kat(\gamma)\leq \delta[A]+\kat(\gamma).$$
The set $\gamma^{-1}[A_0]$ is equal to $\{a\in A:\ \gamma[a]\leq A_0\}$ since $A_0\subseteq \Kat(\Gamma)$. Applying $\delta$ to $a$ in this subset, 
$$\gamma\delta[a]=\delta\gamma[a]+\kat(\gamma)\leq \delta[A_0]+\kat(\gamma)\leq A_0+\kat(\delta)+\kat(\gamma)=A_0.$$
This completes the proof of the Lemma.
\end{proof}
\subsection{Dimensionality}
We will work with endogenies definable in a finite-dimensional theory, in the sense of \cite{wagner2020dimensional}. The first immediate consequence of this setting is the following Lemma.
\begin{lemma}\label{Lemma 7}
    Let $A$ be a definable abelian group and $\phi$ a definable endogeny of $A$. Then, $\dim(A)=\dim(\operatorname{Im}(\phi))+\dim(\ker(\phi))$ with $\ker(\phi)=\{a\in A:\ (a,0)\in \phi\}$ or equivalently $\{a\in A:\ \phi[a]=\kat(\phi)\}$.
\end{lemma}
\begin{proof}
    The proof follows simply by fibration. Let 
    $$\phi: a\in A\mapsto \phi[a] \operatorname{Im}(\phi)/\kat(\phi).$$
    Then, this is a definable surjective homorphism of groups and so $\dim(A)=\dim(\operatorname{Im}(\phi))+\dim(\ker(\phi))$.
\end{proof}
In both theorems $A$ and $B$, we work with invariant pre-subrings of endogenies, one essentially unbounded and one essentially infinite, in the following sense
\begin{defn}
Let $T$ be a first order theory and $\mathfrak{M}$ a model of $T$.
\begin{itemize}
    \item A subset $X$ in $\mathcal{M}$ is \emph{invariant} if it is union of type-definable sets in $\mathcal{M}$.
    \item An invariant subset $X=\bigcup_{i\in I} A_i(\mathcal{M})$ of $\mathcal{M}$ is \emph{unbounded} if, for every cardinal $\kappa$, there exists an elementary extension $\mathcal{M}'$ of $\mathcal{M}$ such that the cardinality of $X(\mathcal{M}')=\bigcup_{i\in I} A_i(\mathcal{M}')$ is greater than $\kappa$.\
    \item An invariant pre-ring of endogenies $\Gamma$ is \emph{essentially unbounded} if $\Gamma/{\sim}$ is unbounded, it is \emph{essentially infinite} if $\Gamma/{\sim}$ is infinite.
\end{itemize}
\end{defn}
\section{Theorem A: base case}
The proof of the theorem $A$ will be divided into three parts:
\begin{itemize}
    \item A \emph{base case} in which all the endogenies of $\Gamma\cup\Delta$ have finite kernel or finite image.
    \item A \emph{second case} in which the essentially unbounded ring of endogenies has an element with infinite image and infinite kernel.
    \item A \emph{third case} in which only the essentially infinite ring has an element with infinite kernel and infinite image.
\end{itemize}
We start proving the base case, in which every $\phi\in \Gamma\cup\Delta$ has finite kernel or finite image. The statement of the Theorem in this case is the following.
\begin{theorem}\label{Theorem Ab}
    Let $\Gamma,\Delta$ be two sharply commuting invariant pre-rings of definable endogenies of a definable abelian group $A$ of finite dimension such that:
    \begin{itemize}
        \item for any $\phi\in \Gamma\cup\Delta$, either $\phi[A]$ is finite or $\ker(\phi)$ is finite;
        \item $\Gamma$ is essentially unbounded and $\Delta$ is essentially infinite or vice-versa.
    \end{itemize}
    Then, $\Kat(\Gamma,\Delta)$ is finite. Moreover, there exists a finite subgroup $(\Gamma,\Delta)$-invariant that contains all the finite $\Gamma$ or $\Delta$-weakly invariant subgroups.
\end{theorem}
\begin{proof}
    Suppose that $\Gamma$ is essentially unbounded (in case $\Delta$ is essentially unbounded, the proof follows by symmetry).\\
    We prove that $A$ admits a maximal finite $\Gamma$-weakly invariant subgroup $A_0$. Suppose, by way of contradiction, that there exists $\{S_i\}_{i<\omega}$ an infinite countable strictly increasing sequence of finite weakly $\Gamma$-invariant subgroups. We verify that, for any weakly $\Gamma$-invariant subgroups $A$ and $B$, the sum $A+B$ is again $\Gamma$-weakly invariant. Indeed, 
    $$\gamma[A+B]=\gamma[A]+\gamma[B]\leq A+\kat(\gamma)+B+\kat(\gamma)\leq A+B+\kat(\gamma).$$ 
    This implies that the sum $S=\sum_{i<\omega} S_i$ is a countably infinite weakly $\Gamma$-invariant subgroup. Indeed, given $a\in \sum_{i<\omega} S_i$, then $a\in S_{i_1}+...+S_{i_n}$ and $\gamma[a]\leq S_{i_1}+...+S_{i_n}+\kat(\gamma)\leq S+\kat(\gamma)$ by previous proof. Let  $\Gamma_S$ denote the set $\{\gamma_S:\ \gamma\in \Gamma\}$ with $\gamma_S$ the restriction-corestriction of $\gamma$ to $S$ (that is well-defined by Lemma \ref{Lemma 5}). $\Gamma_S$ is essentially unbounded since, given two endogenies $\gamma,\gamma'$ whose difference has finite image on $S$, $\gamma-\gamma'$ has kernel of finite index in $A$. Therefore, it must belong to $\sim$ by hypothesis. Consequently, the set $\Gamma_S/{\sim}$ has the same cardinality as $\Gamma/\sim$, that is essentially unbounded by hypothesis. This is a contradiction since $\Endog(S)$ has at most cardinality $2^{\aleph_0}$ in each elementary extension of $A$.\\
    Since $\Kat(\Delta)$ is given by the sum of finite $\Gamma$ weakly-invariant subgroups, $\Kat(\Delta)$ is contained in $A_0$. Therefore, it is finite.\\
    Moreover, the finite images of elements in $\Delta$ are contained in $A_0$. Indeed, $\delta[A]$ is weakly $\Gamma$-invariant for each $\delta\in \Delta$: $\gamma[\delta[a]]\leq \delta\gamma[a]+\kat(\gamma)\leq \delta[A]+\kat(\gamma)$ by sharp commutation. In addiction, $A_0$ is $\Delta$-invariant indeed $\delta[A_0]$ is a weakly $\Gamma$-invariant finite subgroup, by Lemma \ref{Lemma 6}, and so it must be contained into $A_0$.\\
    We verify that $\Delta/\sim$ can be embedded in $\operatorname{End}(A/A_0)$ by the homorphism $[\_]$ defined as $[\delta](a+A_0)=\delta[a]+A_0$ for each $\delta\in \Delta$ and $a+A_0\in A/A_0$:
    \begin{itemize}
        \item $[\delta]$ is well-defined as function: given $a+A_0=b+A_0$ then $a-b\in A_0$ so $\delta[a-b]\in A_0$ by $\Delta$-invariance. In conclusion, $\delta[a]+A_0=\delta[b]+A_0$ that is an element in $A/A_0$ since $\kat(\delta)\subseteq A_0$. Moreover, it is an endomorphism since $\delta$ is additive.
        \item $[\_]:\Delta/\sim\to \End(A/A_0)$ is well-defined as function \hbox{i.e.} if $\delta\sim \delta'$ then $[\delta]=[\delta']$. Indeed, given $a\in A$, then $(\delta-\delta')[a]\leq A_0$ since $\delta-\delta'$ is an endogeny in $\Delta$ with finite image and so $[\delta](a+A_0)=[\delta'](a+A_0)$.
        \item The function $[\_]$ is injective: if $[\delta]=[\delta']$, then $\delta-\delta'$ has image contained in $A_0$ and so finite \hbox{i.e.} $\delta\sim \delta'$.
    \end{itemize}
    We prove that also $\Kat(\Gamma)$ is contained in $A_0$. It is sufficient to verify that it is finite. Assume, for a contradiction, that $\Kat(\Gamma)$ is infinite.\\
    If $\Delta$ is essentially unbounded, then $\Kat(\Gamma)$ is finite by the previous proof. Therefore, we can assume $\Delta$ essentially infinite and essentially bounded. We verify that $\Kat(\Gamma)$ is contained in a bounded subgroup $B$. This contradicts the essentially unboundedness of $\Gamma$. Indeed, since $\Kat(\Gamma)$ is $\Gamma$-invariant and every $\gamma\in \Gamma$ has finite kernel, $\Gamma$ can be embedded in $\Endog(\Kat(\Gamma))$ that is bounded still bounded. For every $[\delta]\in \Delta/{\sim}$, we denote $\phi[\delta]\in \Endog(A)$ the endogeny $\phi[\delta][a]=\delta[a]+A_0$. This is independent from the choice of an element in $[\delta]$. Indeed, given $\delta,\delta'\in [\delta]$, the difference $\phi[\delta]-\phi[\delta'][a]=\delta-\delta'[a]+A_0\leq A_0$ since $\operatorname{Im}(\delta'-\delta)$ is a weakly $\Gamma$-invariant subgroup and so contained in $A_0$ if $\delta\sim\delta'$. Therefore, $\phi[\delta][a]=\phi[\delta'][a]$. $\phi[\delta]$ commutes sharply with $\Gamma$. Indeed, for every $\gamma \in \Gamma$,
    $$(\gamma\phi[\delta]-\phi[\delta]\gamma)[a]=\gamma[A_0]+A_0+(\gamma\delta-\delta\gamma)[a]\leq A_0+\kat(\gamma).$$ 
    Therefore, the function $\phi:\Delta/\sim\to C^{\#}(\Gamma)$ sending $[\delta]$ in $\phi[\delta]$ is well-defined. Moreover, for every $\delta'\sim \delta$, $\ker(\delta')\subseteq \ker(\phi[\delta])$. Let $a\in \ker(\delta')$. By definition, $\delta'[a]\leq \kat(\delta')\leq A_0$, then $\phi[\delta][a]=\delta[a]+A_0=\delta'[a]+(\delta-\delta')[a]+A_0\leq A_0+A_0+A_0\leq A_0$ and so $a\in \ker(\phi[\delta])$. In addiction, the kernel of $\phi[\delta]$ is finite when $[\delta]\not=0$ since $\ker(\phi[\delta])=\delta^{-1}[A_0]$ that is finite since $A_0$ is finite and $\delta$ has finite kernel by hypothesis.\\
    Given $\gamma\in \Gamma$, $\kat(\gamma)$ is a finite weakly $\Delta$-invariant subgroup and so the restriction-corestriction of $\Delta$ to $\kat(\gamma)$ is well-defined by Lemma \ref{Lemma 5}. Since $\kat(\gamma)$ is finite, also $\Endog(\kat(\gamma))$ is finite and so, by essentially infiniteness of $\Delta$, there exists $\delta\in \Delta$ such that $\delta\not\sim 0$ and $\ker(\delta)\supseteq \kat(\gamma)$. By previous proof, $\ker(\phi[\delta])\supseteq \kat(\gamma)$ and, by arbitrariety of $\gamma\in \Gamma$, 
    $$\Kat(\Gamma)=\sum_{\gamma\in \Gamma} \kat(\gamma)\leq \sum_{[\delta]\in \Delta/\sim-[0]}\ker(\phi[\delta]).$$
    The latter is bounded by essentially boundedness, and since $\ker(\phi[\delta])$ is finite. Therefore, we have reached a contradiction.\\
    We verify that $\Delta/\sim$ acts by injective endomorphisms on $A/A_0$ \hbox{i.e.} $\ker([\delta])=0+A_0$ for $\delta\not\sim 0$. The kernel of $[\delta]$ is the set 
    $$\{a\in A:\ \delta[a]\leq A_0\}=\delta^{-1}[A_0].$$ 
    By Lemma \ref{Lemma 6}, this subgroup is a finite $\Gamma$-invariant subgroup and so it is contained in $A_0$. Since $\Delta/\sim$ is infinite and acts by monomorphism on $A/A_0$, there is no $\Delta/\sim$-invariant finite non-trivial subgroup in $A/A_0$. Suppose, for a contradiction, that there exists a finite $\Delta/\sim$-invariant subgroup $B/A_0$. Then, $\Delta/\sim$ embeds in $\End(B/A_0)$ since, given $b\not=0\in B/A_0$, $[\delta](b)=[\delta'](b)$ iff $[\delta]=[\delta']$. This implies that $\End(B/A_0)$ is finite, and so there exists $[\delta]\in \Delta/\sim-[0]$, whose kernel contains $B$, a contradiction.\\
    For any $\gamma\sim 0$ in $\Gamma$, the finite subgroup $\operatorname{Im}(\gamma)+A_0$ is $\Delta$-invariant, by Lemma \ref{Lemma 6}, and so contained in $A_0$. Therefore, $\Gamma/\sim$ acts by endomorphisms on $A/A_0$ as in the previous case. The kernel of $[\gamma]$ is $\gamma^{-1}[A_0]$ that is a $\Delta/\sim$-invariant subgroup, by Lemma \ref{Lemma 6}, and so again contained in $A_0$.\\
    Therefore, $\Gamma/\sim$ acts by monomorphisms on $A/A_0$. Proceeding as before, we may conclude that every finite weakly $\Gamma$-invariant or $\Delta$-invariant subgroup is contained in $A_0$. This proves the base case of the theorem.
\end{proof}
\section{Theorem A: second case}
In this section, we prove the second case of our theorem: $\Gamma$ is essentially unbounded and not all the elements of $\Gamma$ have finite kernel or finite image.\\
We need the following easy lemma.
\begin{lemma}\label{easylemma}
    Assume that $\Gamma$ and $\Delta$ are two pre-rings of endogenies such that $\Gamma=C^{\#}(\Delta)$ and $\Delta=C^{\#}(\Gamma)$. Then, $\Kat(\Gamma,\Delta)$ contains all the finite weakly $\Gamma$-invariant or $\Delta$-invariant subgroups.
\end{lemma}
\begin{proof}
    Let $B$ be a weakly $\Gamma$-invariant finite subgroup. We construct an endogeny with katakernel equal to $B$ contained in $C^{\#}(\Gamma)$. Define $\delta[a]=B$ for every $a\in A$. This is a definable endogeny with katakernel $B$. $\delta\in \Delta= C^{\#}(\Gamma)$ iff $\gamma\delta[a]-\delta\gamma[a]\leq B+\kat(\gamma)$ for any $\gamma\in \Gamma$. This is equal to $\gamma[B]+B\leq B+\kat(\gamma)$ that is exactly the definition of weakly $\Gamma$-invariance. Therefore, $B\leq \Kat(\Gamma,\Delta)$. The proof for weakly $\Delta$-invariant subgroups follows by symmetry.
\end{proof}
\subsection{Lines}
To prove the second version of Theorem A, we need to introduce and analyze lines.
\begin{defn}
    Let $\Gamma$ be a pre-ring of definable endogenies acting on $A$. A \emph{$\Gamma$-line} or simply a line is any $\gamma[A]$ with $\gamma\in \Gamma$ that does not contain any infinite $\Gamma$ image of strictly smaller dimension.
\end{defn}
In the hypothesis of Theorem A, lines respect the following properties.
\begin{lemma}\label{lemma 8}
    Let $\Gamma,\Delta$ be two sharply commuting pre-rings of endogenies of $A$, an absolutely $(\Gamma,\Delta)$-minimal group. For any $\Gamma$-line $L=\gamma[A]$, the following properties hold:
    \begin{itemize}
        \item $L$ is weakly $\Delta$-invariant;
        \item Given $\gamma\in \Gamma$ such that $\gamma[A]$ is infinite, then $\gamma[A]$ contains a line;
        \item There exist $\gamma_1,..,\gamma_n\in \Gamma$ such that $\sum_{i=1}^n \gamma_i[L]$ is of finite index in $A$;
        \item The lines have all the same dimension;
        \item Given $\gamma\in \Gamma$, then $\gamma[L]$ is finite or $\ker(\gamma)\cap L$ is finite.
        \item For every couple of lines $L,L'$, there exists $\gamma\in \Gamma$ such that $\gamma[L]\leq L'$ and $\gamma[L]$ is of finite index in $L'$. In particular, for every line $L$, there exists $\gamma\in \Gamma$ such that $\gamma[A]\leq L$ and $\gamma[L]$ is of finite index in $L$.
    \end{itemize}
\end{lemma}
\begin{proof}
    The first proposition follows easily from sharp commutation since 
    $$\delta\gamma[A]\leq\gamma[\delta[A]]+\kat(\delta)\leq\gamma[A]+\kat(\delta).$$
    The second point follows from the definition of lines.\\
    For the third, let $L$ be a line and $\sum_{i=1}^n \gamma_i[L]=:A'$ a finite sum of $\Gamma$-images of $L$ of maximal dimension. This definable subgroup is weakly $\Delta$-invariant, since a finite sum of weakly $\Delta$-invariant subgroups is weakly $\Delta$-invariant. Moreover, it is almost $\Gamma$-invariant. Indeed, $\gamma[A']+A'$ is again a finite sum of $\Gamma$-images of $L$ and, by maximality of the dimension, $\gamma[A']$ is almost contained in $A'$. This implies that $A'$ must be of finite index in $A$, by the hypothesis of absolute minimality.\\
    For the fourth, it is sufficient to verify that every $\Gamma$-image has dimension greater than or equal to $L$. Then, given two lines $L=\gamma[A],L'=\gamma'[A]$, it follows that $\dim(L)\leq \dim(L')\leq \dim(L)$. Let $\gamma\in \Gamma$ and $L'=\gamma'[A]$ a $\Gamma$-line. By the previous point, there exist $\gamma_1,...,\gamma_n\in \Gamma$ such that $\sum_{i=1}^n \gamma_i[\gamma[A]]$ is of finite index in $A$. Consequently, $\gamma'[\sum_{i=1}^n \gamma_i\gamma[A]]$ is of finite index in $L'$ and so there is a non-finite term $\gamma'\gamma_i\gamma[A]$. Since a line does not contain any $\Gamma$-image infinite not of finite index, $\gamma'\gamma\gamma[A]$ is of finite index in $L'$. This implies, by dimensionality, that $\dim(L)\geq \dim(L')$.\\
    For the fifth point, let $L$ be a line and $\gamma'\in \Gamma$. Then, if $\gamma'[L]$ is not finite, it contains a line $L'$ that must have same dimension as $L$. By dimensionality, $\dim(\gamma'[L])=\dim(L)$ and the kernel is finite.\\
    We prove the last point: let $L,L'=\gamma'[A]$ be lines. By third point, there exist $\gamma_1,...,\gamma_n\in \Gamma$ such that $\sum_{i=1}^n \gamma_i[L]$ is of finite index in $A$. This implies that also $\sum_{i=1}^n \gamma'\gamma_i[L]$ is of finite index in $L'$. Since $L'$ is a line, every $\gamma'\gamma_i[L]$ is either finite or of finite index in $L$. Not all of them can be finite, since the sum is of finite index in $L'$. Consequently, there exists $\gamma'\gamma_i$ such that $\gamma'\gamma_i[A]\leq L'$ and $\gamma'\gamma_i[L]$ is of finite index in $L'$.
\end{proof}
For the proof of the second case of Theorem A, we equip each $\Gamma$-line with two invariant pre-rings of definable endogenies $\Gamma_L$ and $\Delta_L$, defined in the following way.
\begin{defn}
    Let $A$ be an abelian definable group equipped with two invariant sharply commuting pre-rings of endogenies $\Gamma$ and $\Delta$, and $L$ a $\Gamma$-line. We define:
\begin{itemize}
    \item $\Delta_L=\langle\delta_L:\ \delta\in \Delta\rangle$ with $\delta_L$ the restriction-corestriction to $L$ of $\delta$ (that is well-defined by Lemma \ref{Lemma 5} since $L$ is weakly $\Delta$-invariant by Lemma \ref{lemma 8}).
    \item $\Gamma_L=\langle\gamma_L:\ \gamma\in \Gamma\wedge\gamma[A]\leq L\rangle$ with $\gamma_L$ defined as $\gamma_L[l]=\gamma[l]$ for any $l \in L$.
\end{itemize}
Each of the elements of $\Gamma_L,\Delta_L$ is an endogeny of $L$ by definition.
\end{defn}
We verify that these two pre-rings of definable endogenies inherit the properties of $\Gamma$ and $\Delta$. We say that $\Delta$ is \emph{as essentially large as} $\Delta'$ if $\Delta,\Delta'$ are both essentially infinite or both essentially unbounded.
\begin{lemma}\label{Lemma 9}
    Let $A$ be a definable abelian group, $\Gamma,\Delta$ two sharply commuting invariant pre-rings of definable endogenies, and $L$ a $\Gamma$-line. The following properties hold:
\begin{itemize}
        \item Any element in $\Delta_L$ is equivalent, in $\Endog(L)$, to an endogeny $\delta'_L$ with $\delta'\in \Delta$. Any element of $\Gamma_L$ is equivalent to an element $\gamma_L$ for $\gamma\in \Gamma$ with image in $L$;
        \item If $\Gamma$ and $\Delta$ sharply commute, then $\Gamma_L$ and $\Delta_L$ sharply commute;
        \item $L$ is absolutely $(\Gamma_L,\Delta_L)$-minimal.
        \item $L$ has no $\Gamma_L$-lines.
        \item $\Gamma_L$ and $\Delta_L$ are as essentially large as $\Delta$ and $\Gamma$ respectively. 
    \end{itemize}
\end{lemma}
\begin{proof}
    \textbf{Point 1:}\\
    It is sufficient to prove that the set $\{\delta_L:\ \delta\in \Delta\}$ is closed for sum and product, up to equivalence. For the sum, let $\delta$ and $\delta'$ be two endogenies in $\Delta$. Then, $\delta_L+\delta'_L$ is equivalent to $(\delta+\delta')_L$. Indeed, given $l\in L$, $\delta_L+\delta'_L[l]=\delta[l]\cap L+\delta'[l]\cap L$ while $(\delta+\delta')_L[l]=(\delta[l]+\delta'[l])\cap L$. $\delta[l]=l'+\kat(\delta)$ and $\delta'[l]=l''+\kat(\delta')$ with $l',l''$ contained in $L$. Therefore, the first is equal to $l'+l''+\kat(\delta)\cap L+\kat(\delta')\cap L$ and the latter is $l'+l''+\Kat(\delta+\delta')\cap L$. Therefore, the image of $\delta_L+\delta'_L-(\delta+\delta')_L$ is contained in the finite set $\kat(\delta)\cap L+\kat(\delta')\cap L+(\kat(\delta)+\kat(\delta'))\cap L$ and the two endogenies are equivalent as elements of $\Endog(L)$.  For the product, the proof is similar. Let $\gamma,\gamma'$ be elements in $\Gamma$ with image contained in $L$. Then, $$(\gamma\gamma')_L[l]=\gamma[\gamma'[l]]=\gamma_L[\gamma'_L[l]]$$
    since $\gamma'[L]$ is contained into $L$. The same for sum.\\
    \textbf{Point 2:}\\
    Since, given a ring of endogenies $\Gamma$, the endogenies that sharply commute with $\Gamma$ form a ring, by Lemma \ref{Lemma 3}, it is sufficient to verify that $\gamma_L$ and $\delta_L$ sharply commute for any $\gamma\in \Gamma$ and $\delta\in \Delta$. The image of $\gamma_L\delta_L-\gamma_L\delta_L$ is contained in $L$ (by construction). On the other hand $\gamma_L[l]=\gamma[l]$ for every element $l\in L$. Therefore,  
    $$\delta_L\gamma_L[l]=\delta_L\gamma[l]=\delta\gamma[l]\cap L.$$
    Clearly, $\gamma\delta_L[l]$ for $l$ in $L$ is contained in $\gamma\delta[l]$. Therefore,
    $$(\gamma_L\delta_L-\delta_L\gamma_L)[l]\leq \delta\gamma[l]\cap L+\gamma\delta[l]\leq (\delta\gamma-\gamma\delta[l])\cap L\leq (\kat(\gamma)+\kat(\delta))\cap L.$$
    Since $\kat(\gamma)\leq L$, this is equal to $\kat(\gamma_L)+\kat(\delta_L)$.\\
    \textbf{Point 3:}\\
    Suppose, for a contradiction, that there exists a definable infinite almost $(\Gamma_L,\Delta_L)$-invariant subgroup $B\leq L$ not of finite index in $L$. By finite-dimensionality, there exists a finite sum of the form $S=\sum_{i=1}^n \gamma_i\delta_i[B]$, for an endogeny $\gamma_i$ in $\Gamma$ and $\delta_i$ in $\Delta$, of maximal dimension between all the finite sums in this form. We verify that $S$ is a definable infinite almost $\Gamma,\Delta$-invariant subgroup. $\gamma[S]$ is almost contained in $S$ for every $\gamma\in \Gamma$: given $\gamma[S]$, it is equal to $\gamma[(\sum_{i=1}^n \gamma_i\delta_i[B])]$ that coincides with $\sum_{i=1}^n (\gamma\gamma_i\delta_i)[B]$. By maximality of the dimension $\gamma[S]+S$ is almost contained in $S$. Let $\delta\in \Delta$, then $\delta[S]=\delta[(\sum_{i=1}^n \gamma_i\delta_i[B])]=\sum_{i=1}^n \delta\gamma_i\delta_i[B]$. By sharp commutation, this subgroup is contained in $\sum_{i=1}^n \gamma_i\delta\delta_i[B]+\kat(\delta)=S'+C$ with $C$ a finite subgroup. Since, by maximality of $S$, $S'$ is almost contained in $S$, $S$ is almost $(\Gamma,\Delta)$-invariant. By absolute minimality of $A$, $S$ has dimension equal to the dimension of $A$. Let $\gamma\in \Gamma$ such that $\gamma[A]=L$. Then, $\gamma[S]$ has finite index in $L$ and so $\sum_{i=1}^n \gamma\gamma_i\delta_i[B]$ has finite index in $L$. $\delta[B]=\delta_L[B]+\kat(\delta)$ since $L$ is weakly $\Delta$-invariant and so 
    $$\gamma\gamma_i\delta_i[B]\leq \gamma\gamma_i[{\delta_i}_L[B]]+\gamma\gamma_i[\kat(\delta_i)].$$ 
    By Lemma \ref{Lemma 3} the latter is contained in $\kat(\delta_i)+\kat(\gamma\gamma_i)$. Therefore, $\gamma[S]\leq \sum_{i=1}^n (\gamma\gamma_i)_L\delta_L[B]+\sum_{i=1}^n \kat(\delta_i)+\kat(\gamma\gamma_i)$. Since the last term is finite, $\sum_{i=1}^n (\gamma\gamma_i)_L\delta_L[B]$ must be of finite index in $L$. As $B$ is almost $(\Gamma,\Delta)$-invariant, $B$ almost contains each $(\gamma\gamma_i)_L\delta_L[B]$. Conseqeuntly, $B$ almost contains also the sum of them, and so $B$ is of finite index in $L$, clearly a contradiction.\\
    \textbf{Point $4$:}\\
    Let $\gamma\in \Gamma$ such that $\gamma[A]\leq L$. Then, $\gamma[L]$ is a $\Gamma$-image in $L$ that is a line by Lemma \ref{lemma 8}. Therefore, the image must be finite or of finite index in $L$. By point one, every element of $\Gamma_L$ is equivalent to the restriction to $L$ of an element of $\Gamma$ and so $L$ has no $\Gamma_L$-lines.\\
    \textbf{Point $5$:}\\
    Denote the restriction-corestriction map 
    $$P_L: \Delta/{\sim}\to \Delta_L/{\sim}.$$
    We verify that this map is a well-defined, surjective, and injective homomorphism. It is well defined since, if $\delta\sim \delta'$, then the image of $\delta-\delta'$ is finite and so also the one of the restriction-corestriction. It is surjective since, by point $1$, each element of $\Delta_L$ is equivalent to an element of $P_L(\Delta)$. It is a homorphism since, as we have already proved, $\delta_L+\delta'_L\sim (\delta+\delta')_L$, $(-\delta)_L\sim -\delta_L$, $(\delta\delta')_L\sim \delta_L\delta'_L$. To prove that it is injective, being a homomorphism of rings, it is sufficient to verify that, if $\delta_L\sim 0$, also $\delta\sim 0$. By Lemma \ref{lemma 8}, there exists a finite sum $S=\sum_{i=1}^n \gamma_i[L]$ with $\gamma_i\in \Gamma$ of finite index in $A$. $\delta[\sum_{i=1}^n (\gamma_i[L])]$ is of finite index in $\delta[A]$ so it is sufficient to verify that $\delta[\sum_{i=1}^n (\gamma_i[L])]$ is finite. This is equal to $\sum_{i=1}^n \delta\gamma_i[L]$ and each $\delta\gamma_i[L]$ is finite: by sharp commutation, this sum is equal to $\gamma_i\delta[L]+\kat(\delta)$ that is finite by hypothesis. In conclusion, $\Delta/{\sim}$ and $\Delta_L/{\sim}$ are isomorphic and so $\Delta/\sim$ is as essentially large as $\Delta$ .\\
    We prove that $\Gamma_L$ is as essentially large as $\Gamma$. Take $\{\gamma_1,...,\gamma_n\}\in \Gamma$ such that $K=\bigcap_{i=1}^n \ker(\gamma_i)$ has minimal dimension possible (also $0$). Suppose, for a contradiction, that none of the $\gamma_i\Gamma$ is as essentially large as $\Gamma$. This implies that $\times_{i=1}^n(\gamma_i\Gamma/{\sim})$ is not large. The function 
    $$\phi:(\Gamma/\sim,+)\to \times_{i=1}^n (\gamma_i\Gamma/\sim,+)$$
    that sends $\gamma$ in $(\gamma_1\gamma,...,\gamma_n\gamma)$ is obviously a homorphism of groups, and it is well-defined since if $\gamma\sim 0$, then also $\gamma_i\gamma\sim 0$ for any $i\leq n$. Assume, for a contradiction, that the kernel of this map is trivial. Then, $|\Gamma/\sim|\leq |\times_{i=1}^n \gamma_i\Gamma/\sim|$, contradicting largeness. Consequently, there exists $\gamma\in \Gamma$ such that $\gamma_i\gamma\sim 0$ for each $i$ and $\gamma\not\sim 0$.  By definition, each $\gamma_i\gamma$ has finite image. This implies that $\operatorname{Im}(\gamma)\cap \ker(\gamma_i)$ is of finite index in $\operatorname{Im}(\gamma)$. The same holds for the intersection of all $\ker(\gamma_i)$ \hbox{i.e.} $K\cap \operatorname{Im}(\gamma)$ is of finite index in $\operatorname{Im}(\gamma)$. Since $\gamma[A]$ is infinite, being not equivalent to $0$, there exists a line $L$ contained in $\gamma[A]$ and $\gamma'\in \Gamma$ such that $\gamma'[L]\subseteq L$ is of finite index in $L$ by Lemma \ref{lemma 8}. Since the dimension of $K\cap \ker(\gamma')$ is the same as $K$ (by minimality of the dimension) and $K$ has finite index in $L$ (by previous construction), we conclude that $D=K\cap \ker(\gamma')\cap L$ is of finite index in $L$. Therefore, $\gamma'[L]$ is finite, a contradiction since $L$ is infinite and $\gamma'[L]$ is of finite index in $L$.  Consequently, there exists $i\leq n$ such that $\gamma_i\Gamma$ is as essentially large as $\Gamma$. Let $\gamma_0$ be such that $\gamma_0[A]=L$ and $L'=\gamma'[A]$ with $\gamma'=\gamma_i$. By Lemma \ref{lemma 8}, there exists $\alpha\in \Gamma$ such that $\gamma_0\alpha[L']$ is of finite index in $L$. Therefore, $L'\cap \ker(\gamma_0\alpha)$ is finite (having $L,L'$ same dimension by Lemma \ref{lemma 8}). Suppose that, for an endogeny $\beta$, $\gamma_0\alpha\gamma'\beta\sim 0$. Then, $\gamma_0\alpha\gamma'\beta[A]$ is finite and $\ker(\gamma_0\alpha)$ almost contains $\gamma'\beta[A]$. Since the intersection of these two subgroups is finite, $\gamma'\beta\sim 0$. This implies that the function $\gamma'\Gamma/{\sim} \to \gamma_0\alpha\gamma'\Gamma/{\sim} \subseteq \gamma_0\Gamma/{\sim}$ is injective and so $\gamma_0\Gamma$ is as essentially large as $\Gamma$. We conclude that $\Gamma_L$ is as essentially large as $\Gamma$. By Lemma \ref{lemma 8}, there exist $\gamma'_1,...,\gamma'_n\in \Gamma$ such that $\sum_{i=1}^n \gamma'_i[L]$ is of finite index in $A$. Let $J$ be a large set of indices such that $\{\gamma_j:\ j\in J\}\subseteq \gamma_0\Gamma$ is a set of pairwise inequivalent endogenies (this set exists by essentially largeness of $\gamma_0\Gamma$). $\gamma_j\gamma_i[A]\subseteq L$ and so $(\gamma_j\gamma_i)_{L}$ is in $\Gamma_L$. Since $\sum_{i=1}^n\gamma'_i[A]$ is of finite index in $A$, then, for any couple $\gamma_j,\gamma_{j'}$ with $j,j'\in J$, there exists $i\leq n$ such that $(\gamma_j-\gamma_{j'})\gamma'_i[L]$ is infinite (if not the sum will be finite). By Ramsey's theorem, for a certain $i$, there exists a large subset such that $\{(\gamma_j\gamma_i)_L:\ j\in J'\}$ is a set of pairwise inequivalent endogenies in $\Gamma_L$. This completes the proof.
\end{proof}
\subsection{Proof of the theorem A: second case}
We have all the instruments to prove the second case of the theorem. We may assume that $\Gamma=C^{\#}(\Delta)$ and $\Delta=C^{\#}(\Gamma)$. Indeed, let $\Gamma,\Delta$ be two invariant pre-rings of definable endogenies on $A$, absolutely $(\Gamma,\Delta)$-minimal such that $\Gamma$ is essentially unbounded, and $\Delta$ is essentially infinite (or vice versa). Then, defined $\Delta'=C^{\#}(\Gamma)$ and $\Gamma'=C^{\#}(\Delta')$, the group $A$ and the two pre-rings $\Gamma'$ and $\Delta'$ respect the hypothesis of Theorem A:
\begin{itemize}
    \item $\Gamma'=C^{\#}(\Delta')$ by hypothesis and $\Delta'=C^{\#}(\Gamma')$ since if $\delta\in C^{\#}(\Gamma')$ then $\delta\in C^{\#}(\Gamma)=\Delta'$;
    \item $C^{\#}(\Gamma)$ is an invariant pre-ring of endogenies over the same parameter set of $\Gamma$. Indeed, given an automorphism $\sigma$ of the monster model fixing $X$, then $\Gamma$ is fixed, and given a definable endogeny $\phi$ that commutes with each element of $\Gamma$, also $\sigma(\phi)$ commutes with each element of $\Gamma$.
    \item $A$ is absolutely $(\Gamma',\Delta')$-minimal as it is absolutely $(\Gamma,\Delta)$-minimal and $\Delta'\supseteq \Delta$ and $\Gamma'\supseteq \Gamma$;
    \item $\Gamma'$ is essentially unbounded and $\Delta'$ is essentially infinite (or vice-versa) since the same holds for $\Delta$ and $\Gamma$.
\end{itemize}
Therefore, we can apply Theorem $A$ and, if $\Kat(\Gamma',\Delta')$ is finite, the same holds for $\Kat(\Gamma,\Delta)$ since $\Kat(\Gamma,\Delta)\leq \Kat(\Gamma',\Delta')$.
\begin{theorem}\label{Theorem 2A}
    Let $A$ be an abelian definable group of finite dimension and $\Gamma,\Delta$ two invariant pre-rings of endogenies on $A$ such that:
    \begin{itemize}
        \item $A$ is absolutely $(\Gamma,\Delta)$-minimal;
        \item $\Gamma=C^{\#}(\Delta)$ and $\Delta=C^{\#}(\Gamma)$;
        \item $\Gamma$ is essentially unbounded and $\Delta$ is essentially infinite;
        \item $\Gamma$ has a line;
        \item For any $\Gamma$-line $L$, the bikatakernel $\Kat(C^{\#}C^{\#}(\Gamma_L),C^{\#}(\Gamma_L))$ is finite.
    \end{itemize}
    Then, the bikatakernel $\Kat(\Gamma,\Delta)$ is finite.
\end{theorem}
\begin{proof}
By hypothesis, for any line $L$, the bikatakernel $L_0=\Kat(C^{\#}(\Gamma_L),C^{\#}C^{\#}(\Gamma_L))$ is finite and, by Lemma \ref{easylemma}, it contains every weakly $C^{\#}(\Gamma_L)$-invariant or $C^{\#}C^{\#}(\Gamma_L)$-invariant finite subgroup. Every finite image of elements in $\Gamma_L$ and every $\gamma_L^{-1}[L_0]$ for $\gamma_L\in \Gamma_L-[0]$ is weakly $C^{\#}(\Gamma_L)$-invariant and finite. The weakly $C^{\#}(\Gamma)$-invariance of the images follows as in Theorem \ref{Theorem Ab} while for $\gamma_L^{-1}[L_0]$ by Lemma \ref{Lemma 6}, since $L_0$ contains $\Kat(C^{\#}(\Gamma_L),\Gamma_L)$. Finally, $\gamma_L^{-1}[L_0]$ is finite since any element of $\Gamma_L$ is equivalent to an element $\gamma_L$ for $\gamma \in \Gamma$ by Lemma \ref{Lemma 9} and any element $\gamma_L$ has finite kernel or finite image by Lemma \ref{lemma 8}. Therefore, $\Gamma_L/\sim$ is an invariant ring of monomorphisms acting on $L/L_0$.\\
We verify that every endomorphism in $\Gamma_L/\sim$ is also surjective. The invariant ring of monomorphisms $\Gamma_L/\sim$ is, clearly, an integral domain. Moreover, the dimension of every type-definable subset $X$ of $\Gamma_L/(\Gamma_L)_0$ is bounded by $\dim(L)$: given $x,y\in X$ and $a\not=0\in L/L_0$ then $x(a)=y(a)$ iff $(x-y)(a)=0$ and so iff $x-y=0$, since $x-y\in \Gamma_L$, and so it is injective or $0$. Therefore, the definable homomorphism $X\to L/L_0$ that sends $\gamma$ to $\gamma[a]+L_0$ is injective and the dimension of $X$ is bounded by $\dim(L)$. In addition, there exists a type-definable subset $X$ of $\Gamma_L/\sim$ of dimension strictly greater than $0$. Indeed, $\Gamma_L/(\Gamma_L)_0$ is unbounded by Lemma \ref{Lemma 9} and, being a bounded union of type-definable subsets, not all type-definable subsets can be finite. Since we are in a finite-dimensional context, a type-definable subset is finite iff it is of dimension $0$. Therefore, there exists a type-definable subset of dimension $>0$ in $\Gamma_L/\sim$.\\
This implies, by Proposition 3.6 of \cite{wagner2020dimensional}, that the skew-field of fractions $K$ of $\Gamma_L/(\Gamma_L)_0$ is definable. We interpret every element of $K$ as an equivalence class of quasi-endomorphisms on $L/L_0$ through the function $\phi:K\to \mathcal{F}-\End(L/L_0)/{\sim}$ that sends $[m]/[n]$, for $m,n\in \Gamma_L/\sim$ and $n\not=[0]$, to $[m\circ n^{-1}]$ with $n^{-1}$ the quasi-endomorphism given by $n^{-1}:\operatorname{Im}(n)\to L/L_0$ sending $b=n(a)$ to $a+L_0$. We verify that this map is a definable function.
\begin{itemize}
        \item $mn^{-1}$ is well defined and in $\mathcal{F}-\End(L/L_0)$. First, the image of $n$, being $n$ different from $0$, is of finite index in $L/L_0$. It is a quasi-endomorphism since both $n$ and $m$ are endomorphisms and, by injectivity, there exists a unique $b$ such that $n(b)=a$.
        \item The function $\phi$ is well defined. We have to check that if $m/n=m'/n'$ as fractions of an ore domain, then the quasi-endomorphisms associated are equivalent. We recall that the equivalence relation in an ore domain $R$ is $m/n=m'/n'$ iff there exists $d,d'\in R$ such that $nd=n'd'\not=0$ and $md=m'd'$. We verify that the quasi-endomorphisms $mn^{-1}$ and $m'n'^{-1}$ are equivalent in $\mathcal{F}-\End(L/L_0)$. The domain of $\operatorname{Im}(d)$ and $\operatorname{Im}(d')$ are of finite index in $L/L_0$, therefore also $\operatorname{Im}(n'd)\cap \operatorname{Im}(nd')$ is of finite index. Let $b\in \operatorname{Im}(n'd)\cap \operatorname{Im}(nd')=D$ and we study $mn^{-1}-m'n'^{-1}(b)$.\\
        $n^{-1}(b)$ is in the domain of $d^{-1}$ and so $mn^{-1}(b)=md(d^{-1}n^{-1})(b)$ that is equal to $m'd'd^{-1}n^{-1}(b)$. Since $m'$ is injective, $m'd'd^{-1}n^{-1}(b)$ is equal to $m'(n'^{-1}(b))$ iff $d'd^{-1}n^{-1}(b)=n'^{-1}(b)$ \hbox{i.e.} $n'd'(d^{-1}n^{-1}(b))=b$. $d^{-1}n^{-1}(b)=c$ with $nd(c)=b$ but $nd=n'd'$ and so $n'd'(c)=b$.
\end{itemize}
Given $X$ a type-definable subset in $\Gamma_L/\sim$ of maximal dimension then, by the proof of Proposition $3.6$ in \cite{wagner2020dimensional}, $X-X/(X-X)^*$ is equal to $K$. The set of the quasi-endomorphisms in the form $(x-x')(y-y')^{-1}$ with $x,x',y,y'\in X$ is a type-definable subset of $\mathcal{F}-\End(L/L_0)$ and the images of $(x-x')(y-y')^{-1}\not=0$ is a uniformly definable family of subgroups of finite index in $L/L_0$. By Lemma \ref{boundedind}, there is a bound $M<\omega$ on these indices. Any element $\gamma$ of $\Gamma_L/\sim$ is equal, in $K$, to an element $(x-x')/(y-y')$ of $X-X/(X-X)^{\ast}$ and, by the preceding proof, they are equivalent as quasi-endomorphisms on $L/L_0$. Therefore, the image of the difference between $\gamma$ and $(x-x')(y-y')^{-1}$ is finite. This image is also a $C^{\#}(\Gamma_L)$-invariant subgroup: given $m,n,m',n'\in \Gamma_L/(\Gamma_L)_0$, the subgroup $mn^{-1}-m'n'^{-1}(L/L_0)$ is $C^{\#}(\Gamma_L)$-invariant. Let $b\in L/L_0$ and $\delta\in C^{\#}(\Gamma_L)$ then, if $a\in L/L_0$ is such that $n(a)=b$,
 $$n(\delta(a))+L_0=\delta(n(a))+L_0=\delta(b)+L_0.$$
 Since also $m$ commutes with $\delta$, 
 $$\delta((mn^{-1}-m'{n'}^{-1}(a)))+L_0=(mn^{-1}-m'{n'}^{-1})(\delta(a))+L_0\leq \operatorname{Im}(mn{-1}-m'{n'}^{-1}).$$
 In conclusion, $\operatorname{Im}(mn^{-1}-m'n'^{-1})$ is a finite $C^{\#}(\Gamma_L)$-invariant subgroup in $L$ and so contained in $L_0$. In other words, $mn^{-1}$ coincides with $\gamma$ on the image of $n$. Consequently, $\operatorname{Im}(\gamma)\geq \operatorname{Im}(mn^{-1})$ and the index of the image of any element in $\Gamma_L/(\Gamma_L)_0$ in $L/L_0$ must be smaller than $M$. Suppose, for a contradiction, that there exists $\gamma\in \Gamma_L/(\Gamma_L)_0$ not surjective. We prove by induction that $|L/L_0:\operatorname{Im}(\gamma^n)|=|L/L_0:\operatorname{Im}(\gamma)|^n$. The case $n=1$ is obvious, so assume that the assumption is true for $n$. For $n+1$, take the function
$$\gamma^n:L/L_0\to \gamma^n(L/L_0)/\gamma^{n+1}(L/L_0).$$
Then, $L/L_0/\gamma(L/L_0)$ is isomorphic to $\gamma^n(L/L_0)/\gamma^{n+1}(L/L_0)$ and so 
$$|L/L_0:\gamma^{n+1}(L/L_0)|=|L/L_0:\gamma^n(L/L_0)||\gamma^n(L/L_0):\gamma^{n+1}(L/L_0)|=|L/L_0:\gamma(L/L_0)|^{n+1}.$$ 
This is a contradiction, since the index of the image of $\gamma^n$ goes to infinity when $n$ goes to infinity. Therefore, any $\gamma$ is surjective on $L/L_0$.\\
Taken $\gamma\in \Gamma$ such that $\gamma[A]\leq L$ and $\gamma[L]$ is of finite index in $L$, which exists by Lemma \ref{lemma 8}, $\gamma[L/L_0]=L/L_0$ and so $\gamma[L]+L_0$ is equal to $L$. We define $\gamma'[a]=\gamma[a]+L_0$. This is an endogeny such that $\gamma'[A]\leq L$ and $\gamma'[L]=L$. It remains to prove that it belongs to $\Gamma=C^{\#}(\Delta)$. Given $a\in A$, 
$$\gamma'\delta[a]-\delta\gamma'[a]= (\gamma\delta-\delta\gamma)[a]+L_0+\delta[L_0].$$ 
Being $L_0$ a $\Delta_L$-invariant subgroup of $L$, 
$$\delta[L_0]=\delta_L[L_0]+\kat(\delta)\leq L_0+\kat(\delta).$$
Finally, by sharp commutation between $\delta$ and $\gamma$, 
$$(\delta\gamma-\gamma\delta)[a]\leq \kat(\delta)+\kat(\gamma)+L_0=L_0+\kat(\delta).$$
The proof of the Theorem follows from the next lemma. We recall that a finite sum $\sum_{i=1}^n A_i$ of subgroups $\{A_i\}_{i<n}$ of $A$ is \emph{almost direct} if, for any $i\leq n$, $A_i\cap \sum_{j\not=i} A_j$ is finite.
\begin{lemma}\label{lemma 10}
   Let $A$ be an abelian definable infinite group of finite dimension and $\Gamma,\Delta$ two invariant pre-rings of definable endogenies such that:
    \begin{enumerate}
        \item $A$ is $(\Gamma,\Delta)$-almost minimal;
        \item $\Gamma=C^{\#}(\Delta)$;
        \item $\Gamma$ is essentially infinite and $\Delta$ is essentially unbounded or vice-versa;
        \item any $\Gamma$-line $L$ has a finite $(\Gamma_L,\Delta_L)$-invariant subgroup $L_0$ that contains every $C^{\#}(\Gamma_L)$ or $C^{\#}(C^{\#}(\Gamma_L))$-weakly invariant subgroup;
        \item for any line $L$, there exists $\gamma\in \Gamma$ such that $\gamma[A]\leq L$ and $\gamma[L]=L$.
    \end{enumerate}
    Then, the following properties hold:
    \begin{itemize}
        \item There exists a quasi-projection $\pi:A\to L$ such that $\pi\in \Gamma$, $\pi[A]\leq L$ and $\pi[l]=l+\Kat(\pi)$ for any $l\in L$;
        \item There exist lines $\{L_i\}_{i\leq n}$ such that $\sum_{i=1}^n L_i$ is an almost direct sum of finite index in $A$;
        \item The bikatakernel $\Kat(\Gamma,\Delta)$ is finite.
    \end{itemize}
\end{lemma}
\begin{proof}
\textbf{Point 1:}\\
Let $L$ be a line and $\gamma\in \Gamma$ given by hypothesis $5$ of the statement.\\
We define $\pi:A\to L$ simply as $\pi[a]=(\gamma_L)^{-1}\gamma[a]+L_0$. $\gamma_L^{-1}[L_0]=L_0$ indeed $\gamma_L^{-1}[L_0]$ is a finite $C^{\#}(\Gamma_L)$-invariant subgroup that is contained in $L_0$, by definition. Moreover, $\gamma[L_0]=\gamma_L[L_0]=L_0$. We prove that $\pi$ is a quasi-projection on $L$. Clearly, $\pi$ is an endogeny such that $\pi[A]\leq L$ and $\pi[L]=L$. In addition, 
$$\pi[l]=\lambda_L^{-1}\lambda[l]+L_0=\lambda_L^{-1}\lambda_L[l]+L_0=l+\kat(\lambda_L)+L_0=l+L_0.$$
We verify that $\pi\in \Gamma$ \hbox{i.e.} it sharply commutes with every $\delta\in \Delta$. It is sufficient to prove that
$$\delta\pi-\pi\delta[a]\leq \delta_L\pi[a]-\pi\delta[a]+\kat(\delta)$$
is contained into $L_0+\kat(\delta)$ for any $\delta\in \Delta$ and $a\in A$. If we verify that $\gamma_L[\delta_L\pi-\pi\delta[a]]\leq L_0$, then the proof is completed. Indeed, this implies 
$$\delta_L\pi-\pi\delta[a]\leq \gamma_L^{-1}[L_0]=L_0.$$
We analyse $\gamma_L[\delta_L\pi[a]]$. By sharp commutation, it is contained in $$\delta_L\gamma_L\pi[a]+L_0=\delta_L\gamma_L[\gamma_L^{-1}\gamma[a]+L_0]+L_0=\delta_L{\gamma_L}{\gamma_L}^{-1}\gamma[a]+\delta_L[L_0]+L_0.$$
$\delta_L[L_0]=L_0$ since $L_0$ is $\Delta_L$-invariant and $\gamma_L\gamma_L^{-1}[l]=l+L_0$. Consequently, the last term is contained in
$$\delta_L[\gamma[a]+L_0]+L_0=\delta_L[\gamma[a]]+L_0.$$
Moreover, by surjectivity of $\gamma_L$, there exists $l\in L$ such that $\gamma[a]=\gamma_L[l]$. So, we obtain that 
$$\delta_L[\gamma[a]+L_0]+L_0\leq \delta_L[\gamma_L[l]]+L_0.$$
On the other hand,
$\gamma_L\pi\delta[a]=\gamma_L[\gamma^{-1}_L\gamma\delta[a]]+L_0$ is equal to $\gamma\delta[a]+L_0=\gamma\delta[a]\cap L+L_0$. As $\gamma$ and $\delta$ sharply commute, the latter is a subset of
$$\kat(\gamma)+\delta\gamma[a]\cap L+L_0=L_0+\delta_L\gamma_L[l].$$
Taking the difference, we conclude that $\gamma_L(\delta_L\pi-\pi\delta)[a]\leq L_0$.\\
\textbf{Point 2:}\\
We construct quasi-projections $\pi_i:A\to L_i$ such that $\sum_{i=1} \pi_i\sim 1$ and $\sum_{i=1}^n L_i$ is an almost direct sum of finite index in $A$.\\
Let $L$ be a line and $\pi$ a quasi-projection on $L$. $\operatorname{Im}(\pi)\cap \operatorname{Im}(1-\pi)\leq {L_i}_0$. Indeed, given $a\in \operatorname{Im}(\pi)\cap \operatorname{Im}(1-\pi)$, then there exist $b,c\in A$ such that $a\in \pi[b]$ and $a\in 1-\pi[c]$. Nevertheless, $\pi\pi[c]=\pi[c]$ and $\pi\pi[b]=\pi[b]$ by construction of $\pi$. Therefore,
$$a\in \pi[b]=\pi\pi[b]=\pi(1-\pi)[c]=\pi[c]-\pi\pi[c]=\pi[c]-\pi[c]=\kat(\pi).$$
By Lemma \ref{lemma 8}, $\operatorname{Im}(1-\pi)$ is either finite or contains a line $L_2$ with a quasi-projection $\pi'_2$. We define $\pi_2=\pi'_2(1-\pi_1)$: this is again a quasi-projection to $L_2$ with katakernel $\pi'_2[\kat(\pi_1)]$. In fact, $\pi'_2(1-\pi_1)[a]$ for $a\in L_2\leq \operatorname{Im}(1-\pi_1)$ is equal to 
$$\pi'_2[a-\pi_1[a]]=\pi'_2[a+\kat(\pi_1)]=\pi'_2[a]+\pi'_2[\kat(\pi_1)]=a+\pi'_2[\kat(\pi_1)].$$ 
Furthermore, $\pi_2[L_1]=\pi'_2(1-\pi_1)[L_1]=\kat(\pi_2)$. Let $H_2=\operatorname{Im}(1-\pi_1-\pi_2)$. The sum $\operatorname{Im}(\pi_1+\pi_2)+H_2$ is almost direct and the intersection $\operatorname{Im}(\pi_1+\pi_2)\cap H$ is contained into ${L_1}_0+{L_2}_0$. Moreover, $\operatorname{Im}(\pi_1+\pi_2)$ contains $L_1$ and $L_2$ since $\pi_1+\pi_2[L_1]=L_1+\kat(\pi_2)$ and $\pi_1+\pi_2[L_2]\leq L_2+\kat(\pi_1)$.\\
Iterating the process, by finite-dimensionality, we reach $\sum_{i=1}^n \pi_i[A]$ with image of finite index in $A$.\\
\textbf{Point 3:}\\
Let $H$ be $\operatorname{Im}(1-\sum_{i=1}^n \pi_i)$ a finite weakly $\Delta$-invariant subgroup and $L_i=\pi_i[A]$ with $\pi_i$ obtained in the point $2$. We prove that $\Kat(\Gamma,\Delta)$ is contained into $\sum_{i=1}^n {L_i}_0+H=A_0$. Let $\delta\in \Delta$, then $\kat(\delta)=\sum_{i=1} \pi_i\delta[0]+(1-\sum_{i=1}^n \pi_i)[\delta[0]]$. The latter is contained in $A_0$ by hypothesis. The intersection between $\sum_{i\not=j} \pi_i[A]+H\cap \pi_j[A]$ is equal to the katakernel of the $\pi_i$ and so it is contained into $\sum_{i=1}^n {L_i}_0$. Therefore, it is sufficient to prove that $\pi_i\delta[0]\leq (L_i)_0$ to verify the assumption and 
$$\pi_i\delta[0]\leq (\delta\pi_i[0]+\kat(\pi_i))\cap L_i=\delta\pi_i[0]\cap L+\kat(\pi_i)$$
This is contained in
$$\delta_L(\pi_i[0])+{L_i}_0\leq \kat(\delta_L)+\kat(\pi_i)+{L_i}_0={L_i}_0.$$
For $\gamma\in \Gamma$, 
$$\gamma[0]\leq \sum_{i=1}^n \pi_i\gamma[0]+H.$$
Consequently, $\pi_i\gamma[0]=\Kat(\pi_i\gamma)\leq \Kat(\Gamma_L)\leq A_0$ and this completes the proof of the Lemma.
\end{proof}
\end{proof}
\section{Theorem A: third case}
We finally prove the last case of our theorem, whose statement is the following.
\begin{theorem}\label{Theorem A3}
    Let $A$ be an abelian definable group in a finite-dimensional theory, and $\Gamma,\Delta$ be two invariant pre-rings of definable endogenies of $A$ such that:
    \begin{itemize}
        \item $\Delta=C^{\#}(\Gamma)$;
        \item $\Gamma$ is essentially unbounded and $\Delta$ is essentially infinite;
        \item $\Gamma$ has no lines;
        \item $\Delta$ has a line;
        \item $A$ is absolutely $(\Gamma,\Delta)$-minimal.
    \end{itemize}
    Then, $\Kat(\Gamma,\Delta)$ is finite.
\end{theorem}
\begin{proof}
    We prove that every $\Delta$-line has a quasi-projection, then the proof follows by Lemma \ref{lemma 10}.\\
    Therefore, let $L$ be a $\Delta$-line. By Lemma \ref{lemma 8}, the elements of $\Gamma_L$ and $\Delta_L$ have finite image or finite kernel. As in the proof of Theorem \ref{Theorem Ab}, there exists $L_0$ maximum finite weakly $\Gamma_L$-invariant subgroup that contains $\Kat(\Gamma_L,\Delta_L)$ and such that it is $(\Gamma_L,\Delta_L)$-invariant. $\ker(\delta_L)$ is $0$ in $L/L_0$. Indeed, $\delta_L^{-1}[L_0]$ is a $\Gamma_L$-invariant subgroup and so it is contained in $L_0$. Moreover, $\ker(\gamma_L)$ for any $\gamma_L\not=0\in \Gamma_L/(\Gamma_L)_0$ is $0$ since $\ker(\Gamma_L)$ is a finite $\Delta_L/(\Delta_L)_0$-invariant subgroup. On the other hand, $\Delta_L/\sim$ acts by monomorphisms and it is infinite, a contradiction. For the same reason, also the finite images of elements in $\Gamma_L$ or $\Delta_L$ in $L/L_0$ are $0$. Therefore, $\Gamma_L/{\sim},\Delta_L/{\sim}$ are two rings of monomorphisms acting on $L/L_0$.\\
    We prove that also in this case, the elements of $\delta_L/{\sim}-[0]$ act by automorphisms on $L/L_0$. This implies, given $\delta\in \Delta_L-[0]_{\sim}$, that $\delta+L_0$ is surjective on $L$. The construction of a quasi-projection follows as in the proof of Theorem \ref{Theorem 2A}. Again working as in Theorem \ref{Theorem 2A}, it is sufficient to verify that there exists a bound on the index of the images of elements in $\Delta_L/{\sim}-[0]$. Suppose, for a contradiction, that there is no bound on the indices of the images of elements in $\Delta_L/{\sim}$. Take $\{\delta_i\}_{i<\omega}\subseteq \Delta_L/{\sim}$ a countable family of monomorphisms such that $|L/L_0:\delta_i(L/L_0)|>i$. Let $B=\bigcap_{i<\omega} \operatorname{Im}(\delta_i)$ be a type-definable subgroup of $L/L_0$ of dimension equal to the dimension of $L$. $B$ is $\Gamma_L/{\sim}$-invariant being the intersection of $\Gamma_L/{\sim}$-invariant subgroups. The restriction of any element of $\Gamma_L/{\sim}$ to $B$ is injective, since any element of $\Gamma_L/{\sim}$ is injective. Moreover, any element $\gamma$ of $\Gamma_L/{\sim}$ is surjective on $L$, by proof of the Theorem \ref{Theorem 2A}. We show that $\gamma$ is surjective on $B$. It is sufficient to prove that, for any $b\in B$, the $b'\in L$ such that $\gamma(b')=b$ in $L/L_0$ belongs to $B$. Let $\delta_i(l)=b$ then $b'=\gamma^{-1}\delta_i(l)$ but $\gamma^{-1}$ and $\delta_i$ commutes so $b'=\delta_i(\gamma^{-1}(l))$ and so $b'\in \operatorname{Im}(\delta_i)$. The definable skew-field of fractions $K$ of $\Gamma_B/{\sim}$ can be identified by a skew-field contained in $\operatorname{Aut}(B)$ through the identification $m/n=mn^{-1}$ with $m,n\in X=Y-Y$ with $Y$ a type-definable subset of maximal dimension in $\Gamma_L/{\sim}$.
    \begin{itemize}
        \item We have already proven that $mn^{-1}$ are partial endomorphisms with inverse $nm^{-1}$. Since $n^{-1}$ is total, these are automorphisms.
        \item The function sending $K\to XX^{-1}$ is well defined since given $mn^{-1}\sim m'n'^{-1}$, they coincide on the domains. Since the domains are total, they are equal as automorphisms.
        \item This identification is a homomorphism for the sum. It follows simply by proving that
        $$x{x'}^{-1}+y{y'}^{-1}\sim z{z'}^{-1} \iff x/x'+y/y'=z/z'.$$
        We recall that the sum in the skew-field of fractions of an Ore domain is given by $x/x'+y/y'=xt+yt'/x't$ with $x't=y't'$. These $t,t'$ exist since $\dim(y'Y+x'Y)=\dim(Y)$ and so the function 
        $$+:y'Y\times x'Y\to y'Y+x'Y$$
        is not injective. Therefore, there exist $t,t'\in X=Y-Y$ such that $y't=x't$. Given 
        $$(xt+yt')(x't)^{-1}(a)=xt(x't)^{-1}(a)+yt'(t')^{-1}y'^{-1}(a),$$
        this is equal to $xx'^{-1}(a)+yy'^{-1}(a)$. 
        \item We verify that this identification is a homomorphism for the product. It is sufficient to prove that if $1/x\cdot y=z/z'$ then $x^{-1}y\sim zz'^{-1}$. We have that $1/x\cdot y=z/z'$ iff $yz'=xz$. $z,z'$ exist again by the Ore condition and $x^{-1}y\sim zz'^{-1}$ iff $y\sim xzz'^{-1}$. Since $xz\sim yz'$, $xzz'^{-1}\sim yz'z'^{-1}\sim y$ and this completes the proof.
    \end{itemize}
    Therefore, the image $K'$ of the identification of $K$ in $\operatorname{Aut}(B)$ is a definable skew-field and $B$ is a $K$-vector space. Since we are working in a finite-dimensional theory and $\dim(K)>0$, the dimension of $B$ is finite and $B$ is a definable subgroup of $L/L_0$. Being of the same dimension of $L/L_0$, it is of finite index in $L/L_0$, and this is a contradiction. This completes the proof of the Theorem.
\end{proof}
We resume all the results obtained in the following Theorem.
\begin{theorem}\label{ThmA}
    Let $A$ be an abelian definable group of finite dimension and $\Delta,\Gamma$ two invariant pre-rings of definable endogenies such that:
    \begin{itemize}
        \item $A$ is absolutely $(\Gamma,\Delta)$-minimal;
        \item $\Gamma$ is essentially unbounded and $\Delta$ essentially infinite or vice-versa;
        \item $\Gamma,\Delta$ commute sharply.
    \end{itemize}
    Then, $\Kat(\Gamma,\Delta)$ is finite.
\end{theorem}
\begin{proof}
  Let $(A,\Gamma,\Delta)$ be a counterexample of the Theorem of minimal dimension.\\
  We may assume that $\Gamma$ is essentially unbounded. In the other case, the proof follows by symmetry. As already observed, we can suppose that $\Gamma=C^{\#}(\Delta)$ and $\Delta=C^{\#}(\Gamma)$. If $\Gamma$ and $\Delta$ have no lines, the contradiction follows by Theorem \ref{Theorem Ab}. If $\Gamma$ has a line, for each $\Gamma$-line $L$, $\Kat(C^{\#}(\Gamma_L),C^{\#}C^{\#}(\Gamma_L))$ is finite, by minimality of the dimension of $A$. Then, the conclusion follows from Theorem \ref{Theorem 2A}. Finally, if $\Gamma$ has no lines but $\Delta$ has a line, the conclusion follows by Theorem \ref{Theorem A3}.
\end{proof}
\section{Proof of theorem B}
To prove Theorem $B$, we work in $A/A_0$ with $A_0$ the bikatakernel of $\Kat(\Delta,\Gamma)$. Theorem $B$ follows from the following result.
\begin{thmB,2}\label{Theorem B,2 version}
    Let $A$ be an abelian definable group of finite dimensions and $\Gamma,\Delta$ two invariant rings of endomorphisms such that:
    \begin{itemize}
        \item $C(\Gamma)=\Delta$ and $C(\Delta)=\Gamma$;
        \item $\Gamma$ is unbounded and $\Delta$ is infinite or vice-versa;
        \item no elements of $\Gamma\cup\Delta$ has finite non-$0$ kernel or finite non-$0$ image;
        \item $A$ is absolutely $(\Gamma,\Delta)$-minimal.
    \end{itemize}
    Then, there exists a definable field $Z$ such that $A$ is a $Z$-vector space of finite dimension contained into $\Delta\cap \Gamma$ and $\Delta$ and $\Gamma$ act $Z$-linearly on $A$.
\end{thmB,2}
The hypothesis of Theorem \ref{Theorem B,2 version} are respected by $A/A_0,\Delta/\sim$ and $\Gamma/\sim$ as in Theorem $B$ since:
\begin{itemize}
    \item By Lemma \ref{easylemma}, $\Delta/\sim$ and $\Gamma/\sim$ act on $A/A_0$ without finite non-$0$ images or finite non-$0$ kernels;
    \item $\Delta/\sim$ and $\Gamma/\sim$ are one infinite and the other unbounded by hypothesis;
    \item $A/A_0$ is absolutely $(\Delta/\sim,\Gamma/\sim)$-minimal;
    \item $C(\Gamma/\sim)=\Delta/\sim$. Indeed, given $\delta\in C(\Gamma/\sim)$, the endogeny $\delta'[a]=a+A_0$ sharply commutes with $\Gamma$. Therefore, it belongs to $\Delta$ and $[\delta']=\delta\in \Delta/\sim$. To verify that it commutes, it is sufficient to observe that 
    $$\delta'\gamma[a]-\gamma\delta'[a]+A_0\leq \delta(\gamma[a]+A_0)-\gamma(\delta(a))+A_0=([\gamma]\delta-\delta[\gamma](a+A_0))+A_0\leq A_0.$$
    This follows by commutation in $\End(A/A_0)$. Since $\kat(\delta')=A_0$, then $\delta'\in \Delta$ by arbitrariety of $\gamma\in \Gamma$.
\end{itemize}
The proof of Theorem \ref{Theorem B,2 version} is divided into two cases:
\begin{itemize}
    \item A \emph{base case} in which all the endomorphisms of $\Gamma\cup\Delta-0$ are  monomorphisms.
    \item A \emph{general case} in which at least one of the two pre-rings has a line.
\end{itemize}
\subsection{Theorem B Second version: base case}
We start from the base case, whose statement is the following.
\begin{theorem}\label{Theorem Bb}
    Given $A$ an abelian definable group and $\Gamma,\Delta$ two invariant rings of definable endomorphisms on $A$ such that
    \begin{itemize}
        \item $\Gamma=C(\Delta)$ and $\Delta=C(\Gamma)$; 
        \item $\Gamma$ is unbounded and $\Delta$ is infinite;
        \item every non-$0$ endomorphisms $\phi$ in $\Gamma\cup\Delta$ is a monomorphism.
    \end{itemize}
    Then, there exists a definable field $Z\subseteq \Delta\cap \Gamma$ such that $A$ is a vector space of finite dimension and $\Gamma,\Delta$ act $Z$-linearly.
\end{theorem}
\begin{proof}
    The elements of $\Gamma$ are also surjective as already observed in the proof of Theorem \ref{Theorem 2A}. It follows that $\Gamma$ can be embedded in the definable skew-field of fractions $K$ by Proposition $3.6$ of \cite{wagner2020dimensional}. Moreover, by Proposition $3.5$ of \cite{wagner2020dimensional}, the center $Z$ of the skew-field $K$ is an infinite and definable field. Therefore, it commutes with every element of $\Gamma$. $K$ acts by definable automorphisms on $A$ and commute with every element of $\Delta$ and $\Gamma$: we define the action simply taking $(m/n)(a+A_0)=mn^{-1}(a)$ for $m,n\in \Gamma$ respectively. The function $\phi:K\to \Gamma$ sending $m/n$ in $mn^{-1}$ is well-defined and every $mn^{-1}$ with $m,n\in \Gamma$ commutes with $\Delta$ as already proven in Theorem $A$. In addition to this, the function $\phi$ is injective. If $mn^{-1}=m'n'^{-1}$, then $n^{-1}=m^{-1}m'n'^{-1}$ and so $n'm'^{-1}=nm^{-1}$ and $mm^{-1}=Id=m'm'^{-1}$.\\
    This implies that $Z$ can be identified with a subset of $\Gamma$ that commutes with every element in $\Gamma$. Consequently, it is contained in $\Delta$, which will be unbounded. Consequently, there exists a definable skew-field $L$ that coincides with $\Delta$.\\
    $Z$ acts on $A$ as a subfield of $K$ so $A$ is a $Z$-vector space of finite dimension by Proposition $3.3$ of \cite{wagner2020dimensional} and $\Gamma,\Delta$ acts $Z$-linearly on $A/A_0$ by hypothesis. This proves the Theorem.
\end{proof}
\subsection{Theorem B Second version: general case}
We prove the second case of Theorem \ref{Theorem B,2 version}, whose statement is the following.
\begin{thmB,2}
    Let $A$ be an abelian definable group of finite dimension and $\Gamma,\Delta$ two invariant rings of definable endomorphisms such that:
    \begin{itemize}
        \item $A$ is absolutely $(\Gamma,\Delta)$-minimal;
        \item $\Gamma=C(\Delta)$ and $\Delta=C(\Gamma)$;
        \item $\Gamma$ is infinite and $\Delta$ is unbounded or vice-versa.
        \item No elements of $\Gamma\cup \Delta$ has finite non-$0$ kernel or image.
    \end{itemize}
    Then, $A$ is a $K$-vector space of finite dimension with $K$ a definable field such that $\Gamma$ and $\Delta$ act $K$-linearly on $A$.
\end{thmB,2}
\begin{proof}
    Let $(A,\Gamma,\Delta)$ be a counterexample with $A$ of minimal dimension.\\
    If $A$ has no $\Gamma$ or $\Delta$-lines, the proof follows by Theorem \ref{Theorem Bb}. Therefore, we may assume, without loss of generality, that $\Gamma$ has a line. We verify that no element in $\Gamma_L\cup \Delta_L$ has finite non-$0$ kernel. Let $\gamma'_L\in \Gamma_L$ with finite image. Then, $\gamma'\gamma\in \Gamma$ has finite non-$0$ image, since $\gamma'_L(L)=\gamma'(\gamma(A))$, a contradiction to the hypothesis. Let $\delta\in \Delta$ such that $\delta_L(L)=\delta\gamma(A)\cap L$ is finite. Then, since $\delta(L)\leq L$, the image $\delta\gamma(A)\cap L=\delta\gamma(A)=B$ is finite. There exist, by Lemma \ref{lemma 8}, a subset of endomorphisms $\{\gamma_i\}_{i\leq n}\subseteq \Gamma$ such that $\sum_{i=1}^n \gamma_i\gamma(A)$ is of finite index in $A$. Therefore, $\delta(A)$ is finite iff $\delta(\sum_{i=1}^n\gamma_i\gamma(A))$ is finite. The latter is contained, by commutation, into $\sum_{i=1}^n \gamma_i(\delta\gamma(A))=\sum_{i=1}^n \gamma_i(B)$ that is finite and non-trivial, a contradiction.\\
    By minimality of $A$, there exists a definable field $Z_L$ contained in $C(\Delta_L)\cap C(C(\Delta_L))$. We can construct the projections as in Theorem $A$, and, by Lemma \ref{lemma 10}, there exists a direct sum of lines $L_i$ of finite index in $A$. We will use the following Lemma in the last part of the proof.
\begin{lemma}\label{Lemma 10}
   Let $A$, $\Gamma$, $\Delta$, and $L$ be as in the Theorem. Then,
    \begin{itemize}
        \item Given a $\Delta$-covariant automorphisms $\sigma:L\to L'$, there exists $\gamma\in \Gamma$ such that $\gamma$ is invertible and extends $\sigma$;
        \item $\Gamma_L=C(\Delta_L)$ and $\Delta_L=C(\Gamma_L)$. Moreover, given $\phi\in \Gamma_L\cap \Delta_L$, there exists $\phi'\in \Gamma\cap \Delta$ such that $\phi'_L=\phi$.
    \end{itemize}
\end{lemma}
\begin{proof}
\textbf{Point 1:}\\
    Let $\sigma$ be a $\Delta$-covariant isomorphism between $L$ and $L'$, two lines. We want to construct $\gamma'\in \Gamma$ such that $\gamma':L\to L'$ coincides with $\sigma$. Take $\pi$ and $\pi'$ the projections of $L$ and $L'$ respectively and $A=L\oplus H$ with $H=\operatorname{Im}(1-\pi)=\ker(\pi)$. Since $\pi[L']$ is a $\Gamma$-image, it can only be $0$ or of dimension equal to $L'$ (by Lemma \ref{lemma 8} and by hypothesis). In the first case, $\pi$ is invertible. Therefore, $\sigma$ is an automorphism $\Delta$-covariant and so it belongs to $\Gamma$. In the latter case, $L'\cap H$ is finite and so trivial. Let $\gamma'=\sigma\pi+(1-\pi)$. Then, $\gamma'$ coincides with $\sigma$ on $L$ and with $1$ on $H$, and it commutes with every element of $\Delta$. By hypothesis, it is an element of $\Gamma$. Finally, it is invertible with inverse given by $\sigma^{-1}\pi'+(1-\pi')$. Indeed, by linearity, it is sufficient to prove that the composition $\gamma'^{-1}\gamma'$ coincides with the identity on $L$ and $H$. On $L$, the element $\gamma'(l)=\sigma(l)$ belongs to $L'$ and so $\gamma'^{-1}\gamma'$ is equal to $\sigma^{-1}\sigma(l)=l$. On $H$, $\gamma'(h)=1$ and $\gamma'^{-1}(l)=1$. This completes the proof of point $1$.\\
    \textbf{Point 2:}\\
    We verify that $\Gamma_L=C(\Delta_L)$ and vice-versa. Given an element $\gamma$ in $C(\Delta_L)$, the homomorphism $\gamma\pi_L$ is in $C(\Delta)=\Gamma$ and the restriction to $L$ is equivalent to $\gamma$.\\
    We verify that, for every $L$ and $L'$, there exists a surjective homomorphism $\gamma'\in \Gamma$ going from one to the other. This implies, since the kernels are all trivial by the previous proof, that it is also an isomorphism. Suppose, by way of contradiction, that all the homomorphisms with finite kernel from $L$ to $L'$ and from $L'$ to $L$ are not surjective (there exists at least one element for each of these families by Lemma \ref{lemma 8}). Additionally, the compositions are not surjective and have finite kernel. These form a subset $X$ of $\Gamma$ given by endomorphisms such that the image of $L$ is a subgroup of finite index in itself. Let $X'=\{x\circ \pi:\ x\in X\}$, a subset of $\Gamma$ whose restriction to $L$ coincides with the restriction to $L$ of $X$. Therefore, we identify $X$ with a subset of $\Gamma_L$. Since, by proof of the Theorem $A$, every element of $\Gamma_L$ is surjective, we have reached a contradiction. By first point, we can assume that $\gamma\in \Gamma$ is an invertible endomorphism that, restricted to $L$, coincides with $\gamma'$. Let $\delta'$ be a $\Gamma_L$-covariant isomorphism, then the function $\gamma\delta\gamma^{-1}$ sends $L'$ in itself. Let $\delta'_{L'}$ be the induced isomorphism. Given $\gamma'\in \Gamma$ that induces an isomorphism between $L$ and $L'$, then the restriction of $\gamma^{-1}\gamma'$ to $L$ commutes with $\Delta_L$ and so it is in $\Gamma_L$. Therefore, it commutes with $\delta$ and 
    $$\gamma'\delta\gamma'^{-1}=\gamma(\gamma^{-1}\gamma')\delta\gamma'^{-1}=\gamma\delta\gamma$$
    Consequently, $\delta'_{L'}$ is independent by the choice of $\gamma\in \Gamma$.\\
    Let $\gamma''_{L'}\in \Gamma_{L'}$ non-zero. This is an automorphism of $L'$ and so, by point $1$, there exists $\gamma''$ in $\Gamma$ invertible that reduces to $\gamma''_{L'}$. $\gamma\gamma''$ is another automorphism in $\Gamma$ that defines an isomorphism between $L$ and $L'$. Therefore,
    $$\gamma''\delta'_{L'}\gamma''^{-1}=(\gamma''\gamma)\delta'(\gamma\gamma'')^{-1}=\delta'_{L'}.$$
    This implies that $\delta'_{L'}$ commutes with $\gamma''_{L'}$. Indeed, $\delta'_{L''}\gamma_1=\gamma_1\delta'_{L'}$ with $\gamma_1\in \sim$ that induces an isomorphism between $L'$ and $L''$: consider $\gamma_2,\gamma_3\in \Gamma/\sim$ inducing isomorphism between $L$ and $L'$ and $L$ and $L''$ respectively. The composition $\gamma_1\gamma_2$ induces an isomorphism between $L$ and $L''$ and on $L'$
    $$\delta'_{L''}\gamma_1=\gamma_1\gamma_2\delta'_{L''}\gamma_2^{-1}\gamma_1^{-1}\gamma_1=\gamma_1\gamma_2\delta'_{L''}\gamma_2^{-1}=\gamma_1\delta'_{L'}$$
    as we want to prove.\\
    Finally, let $\{L_i\}_{i\leq n}$ be the lines such that $A$ is the direct sum of $L_i$, and define $\delta(a)=\sum_{i=1}^n \delta'_{L_i}(l_i)$ with $l_i\in L_i$ and $a=\sum_{i=1}^n l_i$. This endomorphism is definable, well-defined, and extends $\delta'$. It remains to verify that it commutes with every element of $\Gamma$.\\
    Given $\gamma\in \Gamma$, let $\gamma_{i,j}=\pi_i\gamma\pi_j$ with $\pi_i$ the projection associated to $L_i$. By linearity, $\delta$ commutes with $\gamma$ iff it commutes with every $\pi_i\gamma$ and so iff, for every $i,j\leq n$, $\gamma_{i,j}$ commutes with $\delta'_{L_j}$. This follows from the previous observation.\\
    For the last point, observe that, given $\phi\in \Gamma_L\cap \Delta_L$, the associated $\phi'\in \Delta$ commutes with every element of $\Delta$ and so $\phi'\in C(\Delta)=\Gamma$.
\end{proof}
We conclude the proof of Theorem $B$. Let $Z_L$ be the subfield of $C(\Delta_L)\cap C(C(\Delta_L))$ given by the inductive hypothesis. Applying Lemma \ref{Lemma 10}, $Z_L\subseteq \Gamma_L\cap \Delta_L$ and by point $2$ of Lemma \ref{Lemma 10}, each element of $Z_L$ is equal to $\phi_L$ for $\phi\in \Gamma\cap \Delta$. Let $Z=<\phi:\ \phi_L\in Z_L>\subseteq \Gamma\cap \Delta$. The function from $\Delta$ to $\Delta_L$ is a definable isomorphism, and, being $Z_L$ a definable field contained in $\Delta_L$, also $Z$ is a definable field. Moreover, $Z$ acts on $A$ by automorphisms: each $z\in Z$ is an endomorphism that admits an inverse and so it is an automorphism of $A$. Finally, each element of $Z$ commutes with every element in $\Delta\cup\Gamma$. In conclusion $A$ is a $Z$-vector space with $\Gamma$ and $\Delta$ acting $Z$-linearly.
\end{proof}
\section{Quasi-Endomorphisms}
In the following sections, we extend the previous result to quasi-endomorphisms. More precisely, we prove the following theorem.
\begin{theorem}
    Let $A$ be a definable abelian group of finite dimension, $\Delta$ an invariant pre-ring of definable endogenies and $\Gamma$ an invariant near-ring of definable quasi-endomorphisms of $A$ such that:
    \begin{itemize}
        \item $A$ is absolutely $(\Gamma,\Delta)$-minimal;
        \item $\Delta$ is essentially unbounded and $\Gamma$ is essentially infinite or vice-versa;
        \item $\Delta$ and $\Gamma$ flat commute.
    \end{itemize}
    Then, $\Kat(\Gamma,\Delta)$ is finite.
\end{theorem}
\subsection{Basic properties}
We need some basic results on quasi-endomorphisms.
\begin{lemma}
    Let $A$ be an abelian group, $\phi\in \mathcal{F}-\operatorname{End}(A)$ and $B\leq A$ of finite index. Then,
    \begin{itemize}
        \item $\phi[B]$ is of finite index in $\phi[A]$. Therefore, if $\phi[B]$ is finite, then also $\phi[A]$ is finite;
        \item $\phi^{-1}[B]=\{a\in A:\ \phi[a]\leq B+\kat(\phi)\}$ is of finite index in $A$;
        \item $\ker(\phi)=\phi^{-1}[\kat(\phi)]$;
        \item Given $a,b\in \operatorname{Im}(\phi)$, then $\phi^{-1}[a+b]+\ker(\phi)=\phi^{-1}[a]+\phi^{-1}[b]+\ker(\phi)$.
    \end{itemize}
\end{lemma}
\begin{proof}
\begin{itemize}
    \item Suppose $\Dom(\phi)=\bigcup_{i=1}^n a_i+B\cap \Dom(\phi)$ for $a_i\in A$, then $\phi[\Dom(\phi)]=\phi[a_i+B\cap \Dom(\phi)]$. By additivity, it is equal to $\bigcup \phi[a_i]+\phi[B\cap \Dom(\phi)]$ that coincide with $\bigcup_{i=1}^n\bigcup_{k\in \phi[a_i]}k+\phi[B\cap \Dom(\phi)]$. The other implication is obvious.
    \item $\phi^{-1}[B]$ is of finite index in $A$ iff it is of finite index in $\Dom(\phi)$. Since $B$ is of finite index in $A$, $B\cap \operatorname{Im}(\phi)$ is of finite index in $\operatorname{Im}(\phi)$. Therefore, given $a\in \Dom(\phi)$, the intersection $\phi[a]\cap (a_i+B)$ is not empty for some $a_i\in \operatorname{Im}(\phi)$ and so $\phi[a-b_i]\cap B\not= \emptyset$. Consequently, $a-b_i\in \phi^{-1}[B]$ and $\phi^{-1}[B]$ is of finite index.
    \item For the third point, observe that $\ker(\phi)=\{a\in \Dom(\phi):\ 0\in\phi[a]\}=\{a\in \Dom(\phi):\ \kat(\phi)=\phi[a]\}=\phi^{-1}[\kat(\phi)]$.
    \item For the fourth point, $\phi^{-1}[a]+\phi^{-1}[b]\subseteq \phi^{-1}[a+b]$. Indeed, given $a'+b'$ with $\phi[a']=a+\kat(\phi)$ and $\phi[b']=b+\kat(\phi)$, then $\phi[a'+b']=\phi[a']+\phi[b']+\kat(\phi)$ and so $\phi[a'+b']=a+b+\kat(\phi)$. Consequently $a'+b'\in \phi^{-1}[a+b]$. On the other hand, given $c\in \phi^{-1}[a+b]$, then $\phi[c]=a+b+\kat(\phi)$. Taken two counter images $a',b'$ of $a,b$ (that exist by hypothesis) $\phi[c]-\phi[a'+b']=\kat(\phi)$ and so $c-(a'+b')\in \ker(\phi)$ and so $c\in \phi^{-1}[a]+\phi^{-1}[b]+\ker(\phi)$.
\end{itemize}
\end{proof}
We can define two operations on $\mathcal{F}-\End(A)$.
\begin{itemize}
    \item We define $+$ as $\phi+\psi:\Dom(\phi)\cap \Dom(\psi)\rightarrow A/\kat(\psi)+\kat(\phi)$ such that $(\phi+\psi)[a]=\phi[a]+\psi[a]$.
    \item We define $\cdot$ as $\phi\cdot\psi:\psi^{-1}[\Dom(\phi)]\cap \Dom(\psi)\rightarrow A/\phi[\kat(\psi)]$ such that $(\phi\cdot\psi)[a]=\phi[\psi[a]\cap \Dom(\phi)]$.
\end{itemize}
These two operations do not behave well, indeed, neither the left nor the right distributivity holds. Nevertheless, $(\mathcal{F}-\End(A),+,0)$ and $(\mathcal{F}-\End(A),\circ,Id)$ are monoids. Therefore, $\mathcal{F}-\End(A)$ is a \emph{near-ring} with the following (non-standard) definition. 
\begin{defn}
    Let $X$ be a set with two operations $+$ and $\cdot$. $X$ is a \emph{near-ring} if there exists two elements $0,1$ such that $(X,+,0)$ and $(X,\cdot,1)$ are monoids. 
\end{defn}
\begin{lemma}
    $(\mathcal{F}-\End(A),+,0)$ is a commutative monoid, $(\mathcal{F}-\End(A),\cdot,Id)$ is a monoid.
\end{lemma}
\begin{proof}
$(\mathcal{F}-\End(A),+,0)$ is a monoid.
\begin{itemize}
    \item $+$ is associative since the domain of $(\phi+\psi)+\tau$ and $\phi+(\psi+\tau)$ are equal to $\Dom(\phi)\cap \Dom(\psi)\cap \Dom(\tau)$. Obviously, the images coincide.
    \item $+$ is commutative since $A$ is commutative and the domain does not depend on the order.
    \item The endomorphisms $0:A\rightarrow A$ such that $0(a)=0$ for any $a\in A$ is the $0$ of the set.
    \item Each element $\phi$ that has not global domain has no opposite since, for each $\psi\in \mathcal{F}-\End(A)$, the domain of $\phi+\psi$ is strictly contained in $\Dom(\phi)$.
\end{itemize}
$(\mathcal{F}-\End(A),\circ, Id)$ is a monoid:
\begin{itemize}
    \item $\cdot$ is associative. For the images is obvious. For the domains, 
    $$\Dom(\phi(\psi\gamma))=(\psi\gamma)^{-1}[\Dom(\phi)]=\{a\in \Dom(\psi\gamma)|\psi\gamma[a]\in \Dom(\phi)+\kat(\psi\gamma)\}.$$
    This is 
    $$\{a\in \Dom(\gamma):\ \gamma[a]\in \Dom(\psi)+\kat(\gamma),\psi\gamma[a]\in \Dom(\phi)+\kat(\psi\gamma)\}.$$
    On the other hand,
    $$\Dom((\phi\psi)\gamma)=\{a\in \Dom(\gamma):\ \gamma[a]\subseteq \Dom(\phi\psi)+\kat(\gamma)\}.$$
    This is equal to 
    $$\{a\in \Dom(\gamma):\ \gamma[a]\in \Dom(\psi)+\kat(\gamma),\psi[\gamma[a]\cap \Dom(\psi)]\leq \Dom(\phi)+\kat(\psi\gamma)\}.$$
    Therefore, the domains are the same.
    \item $Id$ is obviously the identity. 
\end{itemize}
We observe that the distributivity laws do not hold.
\begin{itemize}
        \item on the left: we take $\psi(\phi+\gamma)$. The domain is 
        $$(\phi+\gamma)^{-1}[\Dom(\psi)]=\{a\in \Dom(\phi)\cap \Dom(\gamma):\ \phi[a]+\gamma[a]\leq \Dom(\psi)+\kat(\phi)+\kat(\gamma)\}.$$
        On the other hand, the domain of $\psi\phi+\psi\gamma$ is 
        $$\{a\in \Dom(\phi)\cap \Dom(\gamma):\ \phi[a]\in \Dom(\psi)+\kat(\phi)\wedge \gamma[a]\in \Dom(\psi)+\kat(\gamma)\}.$$
        This is obviously contained in the first domain, but in general, the other inclusion does not hold. Moreover, the katakernels are different since we only know that 
        $$\kat(\psi(\phi+\gamma))=\psi[(\kat(\phi)+\kat(\gamma))\cap \Dom(\psi)].$$
        The latter contains
        $$\psi[\kat(\phi)\cap \Dom(\psi)]+\psi[\kat(\gamma)\cap \Dom(\psi)]=\kat(\psi\phi+\psi\gamma).$$
        Observe that, when at least one between $\kat(\gamma)$ or $\kat(\phi)$ is contained in the domain of $\psi$, the two katakernels coincide and also the domains. In this case, taken $a\in \Dom(\psi\gamma+\phi\gamma)$, then 
        $$(\psi\phi+\psi\gamma)[a]=\psi\phi[a]+\psi\gamma[a]$$
        and 
        $$\psi(\phi+\gamma)[a]=\psi[(\phi[a]+\psi[a])\cap \Dom(\psi)]$$
        that is equal to $\psi[\phi[a]\cap \Dom(\psi)]+\psi[\gamma[a]]$ that coincides with the first one.\\
        For right distributivity, observe that 
        $$\Dom((\gamma+\phi)\psi)=\{a\in \Dom(\psi):\ \psi[a]\leq (\Dom(\gamma)\cap \Dom(\delta))+\kat(\psi)\}$$
        while 
        $$\Dom(\gamma\psi+\phi\psi)=\{a\in \Dom(\psi):\ \psi[a]\leq (\Dom(\gamma)+\kat(\psi))\cap (\Dom(\phi)+\kat(\psi)).$$
        Obviously, the second contains the first, but the converse is not always true. The katakernels are respectively $\kat(\gamma\psi+\phi\psi)=\gamma[\kat(\psi)\cap \Dom(\gamma)]+\phi[\kat(\psi)\cap \Dom(\phi)]$ and $(\gamma+\phi)\psi[0]=(\gamma+\phi)[\kat(\psi)\cap \Dom(\gamma)\cap \Dom(\phi)]$. The first obviously contains the second, but, in general,  the vice versa is not true.\\
        If we suppose that $\kat(\psi)\leq \Dom(\phi)$ or $\kat(\psi)\leq \Dom(\gamma)$, then the two domains and katakernels coincide.\\
        Suppose $\Dom(\gamma)\supseteq \kat(\psi)$, then the distributivity holds up to a finite error since, given $a$ in the domain of $\gamma\psi+\phi\psi$, then 
        $$(\gamma+\phi)\psi[a]=\{c=c_1+c_2:\ \exists b\ (b,c_1)\in \gamma \wedge (b,c_2)\in \phi\wedge (a,b)\in \psi\}$$
        and 
        $$(\gamma\psi+\phi\psi)[a]=\{c=c_1+c_2:\ \exists b_1,b_2 (b_1,c_1)\in \gamma \wedge (b_2,c_2)\in \phi\wedge (a,b_i)\in \psi\}.$$
        Since $a\in \psi^{-1}[\Dom(\gamma)]$, then $\psi[a]\leq \Dom(\gamma)+\kat(\psi)\leq \Dom(\gamma)$ and so we can define $y\in \gamma[b_2]$ then $c_2+y\in (\gamma+\phi)\psi[a]$ and $c_1-y\in \gamma[b_1-b_2]\leq \gamma[\kat(\psi)]$. This completes the proof.
    \end{itemize}
\end{proof}
To define the structure of a ring, we introduce an equivalence relation on $\mathcal{F}-\End(A)$.
\begin{defn}
    $\phi\sim \psi$ if $(\phi-\psi)[\Dom(\phi)\cap \Dom(\psi)]$ is finite. In the case of endogenies, the latter coincide with the usual equivalence relation.
\end{defn}
Working with quasi-endomorphisms, we are able to define restriction-corestriction also for an almost invariant subgroup.
\begin{lemma}
    Let $\phi$ be a quasi-endomorphism of $A$ and $B$ a subgroup of $A$ almost $\phi$-invariant. Then, the \emph{restriction-corestriction} $\phi_B$ of $\phi$ to $B$ is the quasi-endomorphism on $B$ that sends $x$ to $\phi[x]\cap B$.  
\end{lemma}
\begin{proof}
    Let $\phi'$ be the restriction-corestriction. To prove that $\phi_B$ is in $\mathcal{F}-\End(B)$, it is sufficient to prove that:
    \begin{itemize}
        \item It is well defined. Indeed, if $\phi[b]\subseteq B+\kat(\phi)$, there exists $a\in \phi[b]$ such that $a\in B$ and we can define $\phi[b]\cap B$ as $a+\kat(\phi)\cap B$.
        \item The katakernel of $\phi'$ is finite since it is equal to $\kat(\phi)\cap B$.
        \item The domain of $\phi$ almost contains $B$ since $\Dom(\phi)$ is of finite index in $A$. We verify that $B\cap \phi^{-1}[B]$ is of finite index in $B\cap \Dom(\phi)$. $(B+\kat(\phi))\cap\phi[B]$ is, by hypothesis of almost invariance, of finite index in $\phi[B]$. Therefore, the counter-image is of finite index in $B$.
        \item $\phi$ is a quasi-endomorphism, obviously by hypothesis.
    \end{itemize}
\end{proof}
Our aim, as in the case of endogenies, is to find a finite subgroup $A_0$ in $A$ such that the function, that we call projection, between $\Delta\rightarrow \Delta_{A/A_0}$ and $\Gamma\rightarrow \Gamma_{A/A_0}$ that sends $\phi$ in $\phi_{A/A_0}$ with $\phi_{A/A_0}[a+A_0]=\phi_{A/A_0}[a+A_0\cap \Dom(\phi)]+A_0$ for $a\in \Dom(\phi)A_0/A_0$ is well defined and $\Delta_{A/A_0},\Gamma_{A/A_0}$ are as essentially large as $\Gamma,\Delta$.\\
We determine the properties of $A_0$ with the following lemma.
\begin{lemma}\label{Lemma 19}
    Let $A$ be a definable abelian group, $\Delta$ a near-ring of quasi-endomorphisms on $A$, and $B\leq A$ a finite subgroup of $A$. Then, the projection $\pi$ is well defined and a homomorphism of near-rings iff $B$ is $\Delta$-invariant. Moreover, this projection is invariant under the equivalence relation $\sim$ \hbox{i.e.} $\phi\sim\psi$ iff $\pi(\phi)\sim\pi(\psi)$. This implies obviously that $\Delta$ is as essentially large as $\Delta_{A/B}$.
\end{lemma}
\begin{proof}
   $\pi$ is well defined if $\pi(\phi)$ is a partial endomorphism of $A/A_0$. $\pi(\phi)$ is well defined iff $\phi[(a+A_0)\cap \Dom(\phi))]+A_0=\phi[(b+A_0)\cap \Dom(\phi)]+A_0$ if $a,b\in \Dom(\phi)$ and $a+A_0=b+A_0$.\\
  $\phi[(a+A_0)\cap \Dom(\phi)]+A_0=\phi[a]+\phi[A_0\cap \Dom(\phi)]+A_0$ while $\phi[(b+A_0)\cap \Dom(\phi)]=\phi[b]+\phi[A_0\cap \Dom(\phi)]$. By additivity $\phi[a]-\phi[b]\leq \phi[A_0\cap \Dom(\phi)]$. If $\phi[\Dom(\phi)\cap A_0]\subseteq A_0$, the conclusion follows. This condition is also necessary since, to be a homomorphism, $(\pi\phi)(0)\leq A_0$ and so $\phi[A_0\cap \Dom(\phi)]\leq A_0$.\\
  $\pi$ is an homorphism: given $\pi(\phi)(a+b)=\phi[(a+b+A_0)\cap \Dom(\phi)]$. Since $a,b\in \Dom(\phi)+A_0/A_0$, there exists $a',b'\in \Dom(\phi)$ such that $a+A_0=a'+A_0$ and $b+A_0=b'+A_0$. Therefore, this is equal to $\phi[a']+\phi[b']+A_0$. Nevertheless, $\phi[a']+A_0=\phi[(a'+A_0)\cap \Dom(\phi)]+A_0=\phi[(a+A_0)\cap \Dom(\phi)]$ and this completes the first part of the proof.\\
  To prove the second part, it is sufficient, since $\pi$ is a homomorphism of near-rings, to prove that $\phi\sim 0$ iff $\pi\phi\sim 0$. $\phi\sim 0$ iff $\operatorname{Im}(\phi)$ is finite but $\operatorname{Im}(\pi\phi)=A_0+\operatorname{Im}(\phi)$ and, since $A_0$ is finite, $A_0+\operatorname{Im}(\phi)$ is finite iff $\operatorname{Im}(\phi)$ is finite.
\end{proof}
In particular, we would like to take $A_0=\Kat(\Gamma,\Delta)$. Consequently, we must have that $\Kat(\Gamma,\Delta)$ is finite and $\phi[\Kat(\Gamma,\Delta)\cap \Dom(\phi)]\leq \Kat(\Gamma,\Delta)$ for every $\phi\in \Gamma\cup\Delta$. This, in general, is not clear if the two near-rings only sharp commute. Therefore, we introduce a new concept: the flat commutation.
\begin{defn}
    Given $\delta,\gamma$ two quasi-endomorphism of $A$, then $\delta$ and $\gamma$ \emph{flat commute} if:
    \begin{itemize}
        \item $\delta[\Dom(\gamma)\cap \Dom(\delta)]\leq \Dom(\gamma)$ and $\gamma[\Dom(\delta)\cap \Dom(\gamma)]\leq \Dom(\delta)$;
        \item $(\delta\gamma-\gamma\delta)[\Dom(\delta)\cap \Dom(\gamma)]\leq \kat(\delta)+\kat(\gamma)$. Observe that $\delta\gamma$ and $\gamma\delta$ are defined since $a\in \Dom(\delta)\cap \Dom(\gamma)$ so $\gamma[a]\leq \Dom(\delta)$ and $\delta[a]\leq \Dom(\gamma)$ by first hypothesis.
    \end{itemize}
   Let $X$ be a subset of $\mathcal{F}-\End(A)$. We define $C^{\flat}(X)=\{\phi\in \mathcal{F}-\End(A):\ \forall x\in X\ \phi\text{ and $x$ flat commute}\}$.
\end{defn}
We verify that this new type of commutativity behaves well for sum, product, and katakernel.
\begin{lemma}\label{Lemma 13}
    Given $\Delta$ a near-ring of quasi-endomorphism on $A$, then:
    \begin{itemize}
        \item For any flat commuting quasi-endomorphisms $\gamma,\delta$, $$\gamma[\kat(\delta)],\delta[\kat(\gamma)]\leq \kat(\gamma)+\kat(\delta).$$
        \item The set $C^{\flat}(\delta)$ is a near subring.
    \end{itemize}
\end{lemma}
\begin{proof}
    For the first point, let $\delta[\kat(\gamma)]$. Then,  
    $$\delta[\kat(\gamma)]=\delta[\kat(\gamma)]-0\leq \delta\gamma-\gamma\delta[0]\leq \kat(\gamma)+\kat(\delta).$$
    Since $\mathcal{F}-\End(A)$ is a near-ring, it is sufficient to prove that $C^{\flat}(\Delta)$ is closed for sum, opposite, and product.\\
    Let $\phi,\psi\in C^{\flat}(\delta)$. We start verifying the closure for the sum.
    \begin{itemize}
        \item $(\phi+\psi)[\Dom(\delta)\cap \Dom(\phi)\cap \Dom(\psi)]$ is contained into 
        $$\phi[\Dom(\delta)\cap \Dom(\phi)]+\psi[\Dom(\delta)\cap \Dom(\psi)]\leq \Dom(\delta).$$
        For the vice versa, $\delta[\Dom(\psi)\cap \Dom(\delta)\cap \Dom(\phi)]$ is contained in $\Dom(\phi)\cap \Dom(\psi)$ by hypothesis.
        \item Given $a\in \Dom(\phi+\psi)\cap \Dom(\delta)$, then 
        $$(\phi+\psi)\delta[a]-(\psi+\phi)\delta[a]\leq \kat(\phi)+\kat(\psi)+\kat(\delta).$$
        We can apply both left and right distributivity since $\Dom((\phi+\psi)\delta)\cap \Dom(\phi\delta+\psi\circ \delta)$ contains, by flat commutation, $\Dom(\delta)\cap \Dom(\phi)\cap \Dom(\psi)$. The same holds for right distributivity. Left distributivity holds up to the term $\phi[\kat(\delta)]$ that, by flat commutation, is contained into $\kat(\phi)+\kat(\delta)$.\\
        Therefore,
        $$(\phi+\psi)\delta[a]-\delta(\phi+\psi)[a]\leq \phi\delta[a]+\psi\delta[a]-\delta\phi[a]-\delta\psi[a]+\kat(\delta)+\kat(\phi)$$
        that is contained, by flat commutation, into $\kat(\phi)+\kat(\psi)+\kat(\delta)$. This completes the proof. 
    \end{itemize}
    We prove that $C^{\flat}(\delta)$ is closed for produt.
    \begin{itemize}
        \item We verify that $(\phi\psi)[\Dom(\delta)]\leq \Dom(\delta)$. This is 
        $$\phi[\psi[\Dom(\delta)\cap \Dom(\psi)]\cap \Dom(\phi)]\leq \phi[\Dom(\delta)\cap \Dom(\phi)]\leq \Dom(\delta).$$
        For the other direction, we verify that $\delta[\Dom(\phi\psi)]\leq \Dom(\phi\psi)=\psi^{-1}[\Dom(\phi)]$. Applying $\delta$,  
        $$\delta[\psi^{-1}[\Dom(\phi)]]=\{\delta[a]:\ \psi[a]\in \Dom(\phi)\}.$$
        It belongs to $\Dom(\phi\psi)$ iff $\psi(\delta[a])\leq \Dom(\phi)+\kat(\psi)$. Since $a\in \Dom(\psi)\cap \Dom(\delta)$ and by flat commutation, $\psi(\delta[a])\leq \delta\psi[a]+\kat(\psi)$ but $\psi[a]\leq \kat(\psi)+\Dom(\phi)$ by hypothesis. Therefore, $$\delta[\psi[a]]\leq \delta[(\Dom(\phi)+\kat(\psi))\cap \Dom(\delta)].$$
        Since $\kat(\psi)\leq \Dom(\delta)$, then by flat commutation, $$\delta[(\Dom(\phi)+\kat(\psi))\cap \Dom(\delta)]=\delta[(\Dom(\phi)\cap \Dom(\delta)]+\kat(\psi)\leq \Dom(\phi)+\kat(\psi).$$
        This completes the proof.
        \item For the commutation, we study $\phi\psi\delta[a]-\delta\phi\psi[a]$ for $a\in \Dom(\delta)\cap \Dom(\phi\psi)$. We add and subtract $\phi\delta\psi[a]$ that is defined since $a\in \Dom(\psi)\cap \Dom(\delta)$ and so $\psi[a]\leq \Dom(\delta)$. Moreover, $\psi[a]\leq \Dom(\phi)+\kat(\phi)$ and so $\delta\psi[a]\leq \delta[\Dom(\phi)\cap \Dom(\delta)]+\delta[\kat(\psi)]$ since $\kat(\psi)\subseteq \Dom(\delta)$. Therefore, it is contained in $\Dom(\phi)+\kat(\psi)$ as we want. We reduce $(\phi\psi\delta-\phi\delta\psi)[a]$. $a\in \Dom(\psi)\cap \Dom(\delta)$ and the domain of $\phi(\psi\delta-\delta\psi)$ contains the domain of $(\phi\psi\delta-\psi\delta\phi)$. This implies that 
        $$(\phi\psi\delta-\psi\phi\delta)[a]=\phi[\psi\delta[a]\cap \Dom(\phi)+\delta\psi[a]\cap \Dom(\phi)]\leq \phi[(\psi\delta-\delta\psi[a])\cap \Dom(\phi)]$$
        that is $\phi(\psi\delta-\delta\psi)[a]$. By flat commutation of $\psi$ and $\delta$, this is contained into $\kat(\psi)+\kat(\delta)$. By flat commutation between $\delta$ and $\phi$, $\kat(\delta)\leq \Dom(\phi)$ and so 
        $$\phi[(\kat(\psi)+\kat(\delta))\cap \Dom(\phi)]\leq \phi[\kat(\psi)\cap \Dom(\phi)]+\kat(\phi)+\kat(\delta)\leq \kat(\delta)+\kat(\phi\psi).$$
        For the term $\phi\delta\psi[a]-\delta\phi\psi[a]$, 
        $$\phi\delta\psi[a]-\delta\phi\psi[a]=\phi\delta[\psi[a]\cap \Dom(\phi\delta)]-\delta\phi[\psi[a]\cap \Dom(\delta\phi)].$$
        In addiction, $a\leq \Dom(\delta)$ and by commutation $\psi[a]\leq \Dom(\delta)$. Therefore, $\pi[a]\cap \Dom(\delta\phi)=\pi[a]\cap \Dom(\phi)$.\\
        On the other hand, $\psi[a]\leq \Dom(\phi)+\kat(\psi)$ and so $\psi[a]=l+\kat(\psi)$ with $l\in \Dom(\phi)$. This implies that $\delta[l]\leq \Dom(\phi)$, and then 
        $$\phi\delta\psi[a]=\phi\delta[l+(\kat(\psi)\cap \Dom(\phi\delta))]=\phi\delta[l]+\phi\delta[\kat(\psi)\cap \Dom(\phi\delta)].$$
        Moreover,
        $$\delta[\kat(\psi)\cap \Dom(\phi\delta)]\leq (\kat(\psi)+\kat(\delta))\cap \Dom(\phi)+\kat(\psi)$$
        that is contained into $\phi[\Dom(\phi)\cap \kat(\psi)]=\kat(\phi\psi)$.\\
        It remains to evaluate $\phi\delta[l]-\phi\psi\delta[a]=\phi[\delta[l]-\psi\delta[a]\cap \Dom(\phi)]$. This is equal to 
        $$\phi[(\delta[l]-\psi\delta[a])\cap \Dom(\phi)]\leq \phi[(\delta\psi[a]-\psi\delta[a])\cap \Dom(\phi)]=\phi(\delta\psi-\psi\delta)[a]$$
        that by flat commutation is contained in 
        $$\phi[\kat(\psi)\cap \Dom(\phi)]+\phi[\kat(\delta)]\leq \kat(\phi\psi)+\kat(\delta).$$
    \end{itemize}
This proves that $C^{\flat}(\Delta)$ is a near-subring.
\end{proof}
In the following Lemma, we prove two important properties of flat commutation.
\begin{lemma}\label{Lemma 15}
   Let $\Gamma,\Delta$ be two flat commuting near-rings of quasi-endomorphism. Then,
    \begin{itemize}
        \item $\Kat(\Gamma,\Delta)$ is $(\Gamma,\Delta)$-invariant;
        \item Given $\delta\in \Delta$ and $A_0$ a $\Gamma$-invariant subgroup containing $\Kat(\Delta)$, then $\delta^{-1}[A_0]$ is $\Gamma$-invariant.
    \end{itemize}
\end{lemma}
\begin{proof}
    For the first, observe that, given $a\in \Kat(\Delta)$, there exists $\delta\in \Delta$ such that $a\in \kat(\delta)$. The same holds for $b\in \Kat(\Gamma)$. Given $a+b\in \Kat(\Gamma,\Delta)$, then 
    $$\delta'[a+b\cap \Dom(\delta)]\leq \delta'[(\kat(\delta)+\kat(\gamma))\cap \Dom(\delta)].$$ 
    Since $\kat(\gamma)\leq \Dom(\delta)$, this is equal to 
    $$\delta'[\kat(\delta)\cap \Dom(\delta)]+\kat(\delta')+\kat(\gamma)\leq \Kat(\Gamma,\Delta)$$
    since $\Delta$ is closed for product.\\
    To prove the second point, it is sufficient to verify that $\delta\gamma[\delta^{-1}[A_0]\cap \Dom(\gamma)]\leq A_0$, since $A_0$ contains the bikatakernel. Therefore, $\delta^{-1}[A_0]\leq \Dom(\delta)$, and, applied th flat commutation, $\gamma\delta\delta^{-1}[A_0]\leq \gamma[A_0\cap \Dom(\gamma)]\leq A_0$ by invariance. In particular, $\Kat(\Gamma,\Delta)=A_0$ respects the hypothesis and so $\delta^{-1}[\Kat(\Gamma,\Delta)]$ is $\Gamma$-invariant.\\
\end{proof}
The following result is trivial.
\begin{lemma}
    Let $A$ be a finite-dimensional abelian group, $B$ a definable subset of $A$ and $\phi$ a definable quasi-endomorphism. Then, $\dim(B)=\dim(\phi[B\cap \Dom(\phi)])+\dim(\ker(\phi))$.
\end{lemma}
\begin{proof}
    The proof is equal to the case of endogenies since $\dim(B\cap \Dom(\phi))=\dim(B)$. 
\end{proof}
We finish this section with the definition of global domain.
\begin{defn}
    Given $\Delta$ a near-ring of quasi-endomorphisms, we define 
    $$\Dom(\Delta)=\bigcap_{\delta\in \Delta} \Dom(\delta).$$
    If $\Gamma,\Delta$ are two near-rings of quasi-endomorphisms over $A$, then $\Dom(\Gamma,\Delta)=\Dom(\Delta)\cap \Dom(\Gamma)$.
\end{defn}
The global domain has the following property.
\begin{lemma}\label{Lemma 14}
    Let $\Gamma,\Delta$ be two flat commuting near-rings of quasi-endomorphisms. Then, $\Dom(\Delta)$ is weakly $(\Gamma,\Delta)$-invariant.
\end{lemma}
\begin{proof}
    $\Dom(\Delta)$ is $\Gamma$-invariant, being an intersection of $\Gamma$-invariant subgroups.\\
    We verify that $\Dom(\Delta)$ is weakly $\Delta$-invariant. Assume that $\Dom(\Delta)$ is not of finite index, then there exists a strictly descending chain $\{\Dom(\delta_i)\}_{i\in I}$ such that $\bigcap_{i\in I}\Dom(\delta_i)=\Dom(\Delta)$. Being $\Delta$ a near-ring, $\Dom(\Delta)\leq \bigcap_{i\in I} \Dom(\delta_i\delta)$ with $\Dom(\delta_i)>\Dom(\delta_j)$ for $j>i$ and $I$ an index set. Therefore,
    $$\delta[\Dom(\Delta)]\leq \delta[\bigcap_{i\in I} \Dom(\delta_i\delta)].$$
    This is contained into $\bigcap_{i\in I}\delta[\Dom(\delta_i\delta)]$ and, by definition, $$\delta[\Dom(\delta_i\delta)]\leq \Dom(\delta_i)+\kat(\delta).$$
    In conclusion 
    $$\delta[\Dom(\Delta)]\leq \bigcap_{i\in I} (\Dom(\delta_i)+\kat(\delta))$$
    Since $\{\delta_i\}_{i\in I}$ is a descending chain of subgroups, then, for every $b\in \delta[\Dom(\Delta)]$, there exist $a_i\in \Dom(\delta_i)$ and $c_i\in \kat(\delta)$ such that $b=a_i+c_i$. Consequently, there exists a subset $J$ of same cardinality as $I$ such that $c_j=c$. Therefore, $b-c_j\in \Dom(\delta_j)$ for every $j\in J$ and, so $b-c_j\in \Dom(\Delta)$ \hbox{i.e.}
    $$\delta[\Dom(\Delta)]\leq \Dom(\Delta)+\kat(\delta).$$
    Assume that $\Dom(\Delta)$ is of finite index. Then, there exist $\delta_1,...,\delta_n\in \Delta$ such that, for any $\delta\in \Delta$, 
    $$\Dom(\delta)\geq \bigcap_{i=1}^n\Dom(\delta_i).$$
    Take $\delta'=\sum_{i=1}^n \delta_i$. Then, $\Dom(\delta')=\bigcap_{i=1}^n \Dom(\delta_i)=\Dom(\Delta)$. Let $\delta\in \Delta$, then 
    $$\delta[\Dom(\delta')]\leq \Dom(\delta')+\kat(\delta)\iff\Dom(\delta)\leq \delta^{-1}[\Dom(\delta')]=\Dom(\delta'\delta)\leq \Dom(\delta).$$ 
    By arbitrariety of $\delta$, $\Dom(\Delta)$ is weakly $\Delta$-invariant.
\end{proof}
\section{Theorem A: Base case}
We prove the extension to quasi-endomorphisms of Theorem \ref{Theorem Ab}.
\begin{theorem}\label{Theorem QA}
    Let $A$ be a definable infinite abelian group of finite dimension and $\Gamma,\Delta$ two invariant near-rings of quasi-endomorphisms such that:
    \begin{itemize}
        \item $\Gamma$ is essentially infinite and $\Delta$ is essentially unbounded (or vice-versa);
        \item $\Gamma,\Delta$ flat commute;
        \item every element of $\Gamma\cup\Delta$ has finite image or finite kernel.
    \end{itemize}
    Then, $\Kat(\Gamma,\Delta)$ is finite.
\end{theorem}
 \begin{proof}
     If $\Delta$ is essentially unbounded, it admits a maximal finite weakly $\Delta$-invariant subgroup $A_0$ contained in the domain of $\Delta$, by the proof of Theorem \ref{Theorem Ab}. Since each $\kat(\gamma)$ for $\gamma\in \Gamma$ is weakly $\Delta$-invariant and contained in the domain of every $\delta$, it is contained in $A_0$ and so $\Kat(\Gamma)$ is finite.\\
     To prove that $\Kat(\Delta)$ is finite, we work in $A/A_0$. Here, $\Gamma$ acts by partial endomorphisms by Lemma \ref{Lemma 19}. Indeed, $A_0$ is $\Gamma$-invariant: $\gamma[A_0]\leq \Dom(\Delta)$ and it is weakly $\Delta$-invariant by flat commutation. Consequently, $\gamma[A_0]\leq A_0$.\\
     We take $\Gamma'$ a bounded essentially infinite near-subring of $\Gamma$ (it is possible simply taking $X$ a countable family of pairwise non-equivalent partial endomorphisms and define $\Gamma'$ as the pre-ring generated). On each katakernel $\kat(\delta)+A_0/A_0\leq \Dom(\Gamma)\leq \Dom(\Gamma')$, $\gamma\in \Gamma'$ defines a function simply sending $a+A_0$ to $\gamma[a]+A_0$. By essentially infiniteness, there exists two non equivalent elements $\gamma,\gamma'$ such that $\gamma-\gamma'[\kat(\delta)+A_0/A_0]=0+A_0$ (since the number of function from $\kat(\delta)+A_0/A_0$ to $A_0$ is finite). Therefore, $(\gamma-\gamma')^{-1}[A_0]\geq \kat(\delta)$ and $(\gamma-\gamma')^{-1}[A_0]$ is finite.\\
     This implies that 
     $$\kat(\Delta)= \sum_{\delta\in \Delta} \kat(\delta)\leq \sum_{\gamma\in \Gamma'-[0]} \gamma^{-1}[A_0].$$
     The latter is a bounded sum of finite subgroups and so itself bounded. This is a contradiction by the proof of Theorem \ref{Theorem Ab}.
\end{proof}
\section{Theorem A: First case}
\subsection{Lines}
We introduce lines in this new context. The definition is the same as for endogenies.
\begin{defn}
    We define $L$ a \emph{line} an infinite $\Gamma$($\Delta$)-image that does not contain any infinite $\Gamma$($\Delta$)-image of infinite index.
\end{defn}
The following Lemma is the adaptation to quasi-endomorphisms of the Lemma \ref{lemma 8}.
\begin{lemma}\label{Lemma 12}
    Let $A$ be an abelian definable subgroup of finite dimension and $\Gamma,\Delta$ two flat commuting invariant near-rings of quasi-endomorphisms. Assume that $A$ is absolutely $(\Gamma,\Delta)$-minimal and let $L,L'$ be two $\Gamma$-lines. The following propositions are true:
    \begin{enumerate}
        \item $L$ is almost $\Delta$-invariant;
        \item For every $\gamma\not\sim 0$, $\gamma[\Dom(\gamma)]$ contains a line;
        \item For every $L$ line, there exists a finite sum $\sum_{i=1}^n \gamma_i[L\cap \Dom(\gamma_i)]$ for $\gamma_i\in \Gamma$ of finite index in $A$;
        \item For any couple of lines $L,L'$, there exists $\gamma\in \Gamma$ such that $\gamma[L\cap \Dom(\gamma)]$ almost contains $L'$;
        \item Every two lines have the same dimension. For any $\gamma\in \Gamma$, $\gamma[L]$ is finite or $\ker(\gamma)\cap L$ is finite.
    \end{enumerate}
\end{lemma}
\begin{proof}
    \textbf{Point 1:}\\
    Given $\gamma\in \Gamma$ and $\delta\in \Delta$, then $\gamma[\Dom(\delta)\cap \Dom(\gamma)]$ is of finite index in $L$. Therefore, $\delta[L\cap \Dom(\gamma)]$ is almost contained in $\delta[\gamma[\Dom(\gamma)\cap \Dom(\delta)]]$. By flat commutation, this is in $\gamma[\operatorname{Im}(\delta)\cap \Dom(\gamma)]\leq L$. This verifies the almost invariance.\\
    \textbf{Point 2:}\\
    Obvious by definition.\\
    \textbf{Point 3:}\\
    Given $L$ a line, let $S=\sum_{i=1}^n L_i$ be a finite sum of $\Gamma$ images of $L$ of maximal dimension. We verify that it is almost $(\Delta,\Gamma)$-invariant. The almost $\Delta$-invariance follows from the fact that any $\Gamma$-image is almost $\Delta$-invariant. $L_i\cap \Dom(\delta)$ is of finite index in $L_i$ for any $\delta\in \Delta$. Therefore, the index of $\sum_{i\leq n} L_i\cap \Dom(\delta)$ is finite in $\sum_{i\leq n} L_i$. Consequently,
    $$\delta\big[\sum_{i\leq n} L_i\cap \Dom(\gamma)\big]=\sum_{i\leq n} \delta[L_i\cap \Dom(\gamma)]$$
    is almost contained in $\sum_{i\leq n}L_i$ and this proves the almost invariance.\\
    For $\gamma\in\Gamma$, $\gamma[\sum_{i\leq n}L_i\cap \Dom(\gamma)]$ has $\gamma[\sum_{i\leq n}(L_i\cap \Dom(\gamma))]$ as a subgroup of finite index. The latter is $\sum_{i\leq n} \gamma[L_i]$, which is a sum of $\Gamma$-images of $L$ by closure for the product. Therefore, $\sum_{i\leq n} \gamma[L_i]+\sum_{i\leq n} L_i$ is of same dimension as $\sum_{i\leq n} L_i$. By absolute minimality of $A$, $\sum_{i\leq n} L_i$ is of finite index in $A$.\\
    \textbf{Point 4:}\\
    Let $L'=\gamma'[\Dom(\gamma')],L$ be two lines and $\sum_{i=1}^n \gamma_i[L\cap \Dom(\gamma_i)]$ of finite index in $A$. The image through $\gamma'$ of $\sum_{i=1}^n \gamma_i[L\cap \Dom(\gamma_i)]$ is of finite index in $L'$. Therefore, proceeding as in the previous proof, the sum 
    $$\sum_{i_1}^n \gamma'\gamma_i[L\cap \Dom(\gamma'\gamma_i)]=\sum_{i=1}^n \gamma'\gamma_i[L]$$
    is of finite index in $L'$. By minimality of the dimension of $L'$, each of the $\gamma'\gamma_i[L]$ is of finite index in $L'$ or finite. Not all of them are finite since the sum is of finite index.\\
    \textbf{Point 5:}\\
    By point $4$, given any two lines $L,L'$, there exists $\gamma[L\cap \Dom(\gamma)]$ of finite index in $L'$. Then, $\dim(L)\geq \dim(\gamma[L\cap \Dom(\gamma)])=\dim(L')$. By symmetry, we have the equivalence.
\end{proof}
We introduce $\Delta_L$ and $\Gamma_L$ in this new context.
\begin{lemma}\label{Lemma 16}
    Let $A$ be an abelian definable group of finite dimension and $\Gamma,\Delta$ two invariant pre-rings of quasi-endomorphisms such that: 
    \begin{itemize}
        \item $\Gamma,\Delta$ flat commute;
        \item $A$ is absolutely $(\Gamma,\Delta)$-minimal;
        \item $L$ is a $\Gamma$-line.
    \end{itemize}
    We define $\Delta_L=\langle \phi_{|L}:\ \phi\in \Delta \rangle$ and $\Gamma_L=\langle \phi:L\cap \Dom(\phi)\rightarrow L:\ \phi\in \Gamma,\ \phi[\Dom(\phi)]\subseteq L\rangle$ two near-rings of quasi-endomorphisms on $L$. We have that:
    \begin{itemize}
        \item Any $\gamma\in \Gamma_L$ is equivalent to an element $\gamma'_L$, and every element of $\Delta_L$ is equivalent to an element $\delta'_L$, for $\gamma'\in \Gamma$ with image in $L$ and $\delta'\in \Delta$;
        \item $\Gamma_L$ and $\Delta_L$ flat commute;
        \item $L$ is absolutely $(\Gamma_L,\Delta_L)$-minimal;
        \item $\Gamma_L$ and $\Delta_L$ are as essentially large as $\Delta$ and $\Gamma$.
    \end{itemize}
\end{lemma}
\begin{proof}
    \textbf{Point 1:}\\
    Let $\gamma_i\in \Gamma$ with image in $L$ for every $i$. Then, $\sum_{i=1}^n {\gamma_i}_{L}$ is equal to $(\sum_{i=1}^n\gamma_i)_L$. The domain of the latter is $\Dom\big({\sum_{i=1}^n {\gamma_i}}_{L}\big)=\bigcap_{i=1}^n \Dom(\gamma_i)\cap L$ that is equal to $\bigcap_{i=1}^n (\Dom(\gamma_i)\cap L)$. The latter is the domain of $\sum_{i=1}^n {\gamma_i}_L$. They have same katakernel: $\kat(\gamma_L)=\kat(\gamma)$ for any $\gamma\in \Gamma$ with image in $L$. They have the same image by definition. For the product it is sufficient to observe that $\Dom((\gamma\gamma')_L)$ is equal to $\gamma'^{-1}[\Dom(\gamma)]\cap L$. Since $\gamma'$ has image in $L$, the latter is equal to $\gamma'^{-1}[\Dom(\gamma)\cap L]\cap L$. This is $\gamma'^{-1}[\Dom(\gamma_L)]\cap L$. Therefore, 
    $${\gamma'_L}^{-1}[\Dom(\gamma_L)]=\{a\in \Dom(\gamma')\cap L:\ \gamma'[a]\cap \Dom(\gamma_L)\not=\emptyset\}=\gamma'^{-1}[\Dom(\gamma_L)]\cap L.$$
    Obviously, they have the same image and katakernel.\\
    Given $\delta,\delta'\in \Delta$, then $\delta_L+\delta'_L\sim (\delta+\delta')_L$ and similarly for the product.\\
    The domain of the first is $\delta^{-1}[L]\cap\delta'^{-1}[L]\cap L$, while the domain of the second is $(\delta+\delta')^{-1}[L]\cap L$ and contains the first. Taken $l$ an element of this domain, 
    $$(\delta+\delta')_{L}[l]-(\delta_L-\delta'_L)[l]=(\delta+\delta'[l])\cap L-\delta[l]\cap L-\delta'[l]\cap L.$$
    Since $\delta[l]\leq L+\kat(\delta)$ and $\delta'[a]\leq L+\kat(\delta')$, there exists $b_1\in \delta[a]\cap L$ and $b_2\in \delta'[a]\cap L$ such that $(\delta+\delta')_{L}[l]=(b_1+b_2)+(\kat(\delta')+\kat(\delta))\cap L$. Therefore, the latter is equal to 
    $$b_1+b_2-b_1+\kat(\delta)\cap L-b_2+\kat(\delta')\cap L-\delta'[a]\cap L+(\kat(\delta')+\kat(\delta))\cap L.$$
    This is contained in 
    $$\kat(\delta')\cap L+\kat(\delta)\cap L+(\kat(\delta')+\kat(\delta))\cap L$$
    that is finite.\\
    For product, $\Dom((\delta\delta')_L)=(\delta\delta')^{-1}[L]$ that is equal by definition to 
    $$\{a\in \Dom(\delta\delta'):\ \delta[\delta'[a]\cap \Dom(\delta)]\cap L\not=\emptyset\}.$$
    On the other hand, the domain of $\delta_L\delta'_L$ is 
    $${\delta'_L}^{-1}(\Dom(\delta_L))=\{a\in \Dom(\delta'_L):\ \delta'[a]\cap \Dom(\delta_L)\not=\emptyset\}.$$
    This is 
    $$\{a\in \delta'^{-1}[L]:\ \delta[\delta'[a]\cap \Dom(\delta)]\cap L\not=\emptyset\}.$$ 
    Therefore, the latter is equal to $\Dom((\delta\delta')_L)\cap \delta'^{-1}[L]$. For the katakernels, $\kat((\delta\delta')_L)=\delta[\delta'[0]\cap \Dom(\delta)]\cap L$ while 
    $$\kat(\delta_L\delta'_L)=\delta_L[\delta'_L[0]\cap \Dom(\delta_L)]=\delta[\delta'[0]\cap L\cap \delta^{-1}[L]]\cap L.$$
    This is contained in the first katakernel. On the common domain, the two functions coincide up to katakernel.\\
    \textbf{Point 2:}\\
    $\Gamma_L$ and $\Delta_L$ flat commute. Since $C^{\flat}(\Gamma)$ and $C^{\flat}(\Delta)$ are near subrings, it is sufficient to prove that $\gamma_L$ and $\delta_L$ flat commute for $\gamma\in \Gamma$ with image in $L$ and $\delta\in \Delta$.
    \begin{itemize}
        \item For the domains, we prove first that $\delta_L[\Dom(\gamma_L)\cap \Dom(\delta_L)]\leq \Dom(\gamma_L)$. This is $L\cap \delta[\Dom(\gamma)\cap L\cap \Dom(\delta_L)]$ and it is contained in
        $$L\cap \delta[\Dom(\gamma)\cap \Dom(\delta)]\leq L\cap \Dom(\gamma)=\Dom(\gamma_L).$$
        For the vice-versa, we show that $\gamma_L[\Dom(\delta_L)\cap \Dom(\gamma_L)]\leq \Dom(\delta_L)$. By definition and flat commutation,
        $$\gamma[\Dom(\delta_L)\cap \Dom(\gamma)\cap L]\leq \gamma[\Dom(\delta)\cap \Dom(\gamma)]\leq \Dom(\delta)\cap L.$$
        Since $\Dom(\delta_L)=\delta^{-1}[L]\cap L$, it is sufficient to show that $\delta[\gamma[\Dom(\delta_L)\cap \delta^{-1}[L]]\leq L+\kat(\delta)$. This subgroup is contained, by flat commutativity, in
        $$\kat(\delta)+\gamma[\delta[\delta^{-1}[L]\cap \Dom(\gamma)]].$$
        By definition, it is in 
        $$\kat(\delta)+\gamma[\Dom(\gamma)\cap (L+\kat(\delta))].$$
        Since $\kat(\delta)\leq \Dom(\gamma)$, this subgroup is contained into 
        $$\kat(\delta)+\gamma[L\cap \Dom(\gamma)]+\gamma[\kat(\delta)]\leq L+\kat(\delta)$$ 
        as we want to prove.
        \item We verify the flat commutativity, \hbox{i.e.} we show that, for every $a\in \Dom(\delta_L)\cap \Dom(\gamma_L)$, 
        $$\gamma_L\delta_L[a]-\delta_L\gamma_L[a]\leq \kat(\gamma)+\kat(\delta)\cap L.$$ 
        By definition, $\gamma_L[a]=\gamma[a]$ and $\delta_L[\gamma[a]]$. Since $\gamma[a]\leq L$ and $a\in \Dom(\delta)$, then $\delta[\gamma[a]]\leq L+\kat(\delta)$ and so $\delta_L[\gamma[a]\cap \Dom(\delta_L)]=\delta\gamma[a]\cap L$. Moreover, 
        $$\gamma_L[\delta_L[a]]=\gamma[\delta_L[a]]\leq \gamma\delta[a].$$
        Therefore, this is contained in $L\cap ((\gamma\delta-\delta\gamma)[a])$. Since $\gamma,\delta$ flat commute, the latter is contained into $(\kat(\gamma)+\kat(\delta))\cap L$. On the grounds that $\kat(\gamma)\leq L$, it remains $\kat(\delta)\cap L+\kat(\gamma)$ as we want.
    \end{itemize}
    \textbf{Point 3:}\\
    Suppose there exists an infinite definable almost $(\Gamma_L,\Delta_L)$-invariant subgroup $B'$ not of finite index in $L$. By finite-dimensionality, there exists a finite sum of the form $$S=\sum_{i=1}^n \gamma_i\delta_i[B'\cap \bigcap_{i=1}^n \Dom(\gamma_i)\cap \Dom(\delta_i)=:B]$$ with $\gamma_i,\delta_i$ quasi-endomorphisms in $\Gamma$ and $\Delta$ respectively, of maximal dimension in $A$. This is a definable infinite subgroup, and we prove that it is almost $(\Gamma,\Delta)$-invariant. $\gamma[S\cap \Dom(\gamma)]$ is almost contained in $S$ for every $\gamma\in \Gamma$. Indeed, the subgroup $\gamma[S\cap \Dom(\gamma)]$ is equal to $\gamma[(\sum_{i=1}^n \gamma_i\delta_i[B'])\cap \Dom(\gamma)]$. Since $\sum_{i=1}^n (\gamma_i\delta_i[B]\cap \Dom(\gamma))$ has finite index in $(\sum_{i=1}^n \gamma_i\delta_i[B])\cap \Dom(\gamma)$, to prove the almost containment, it is sufficient to verify that $\sum_{i=1}^n \gamma[\gamma_i\delta_i[B]\cap \Dom(\phi)]=\sum_{i=1}^n (\gamma\gamma_i)\delta_i[B]=S'$ is almost contained in $S$. By maximality of the dimension, $S$ is of finite index in $S+S'$. Therefore, $S'$ is almost contained in $S$ and the same holds for $\gamma[S\cap \Dom(\gamma)]$. Taken $\delta\in \Delta$, $\delta[S\cap \Dom(\delta)]=\delta[(\sum_{i=1}^n \gamma_i\delta_i[B])\cap \Dom(\delta)]$. This is almost contained into $\sum_{i=1}^n \delta[\gamma_i\delta_i[B]]$. Since $\Dom(\delta)$ is of finite index in $A$, this is almost contained into $\sum_{i=1}\delta[\gamma_i[\delta_i[B]\cap \Dom(\delta)]$. For each $i$, $\delta_i[B]\cap \Dom(\delta)\leq \Dom(\delta)\cap \Dom(\gamma_i)$. Applying the flat commutation, 
    $$\gamma_i[\delta[\delta_i[B]\cap \Dom(\delta)]]+\kat(\delta)=\gamma_i[\delta\delta_i[B]].$$
    Therefore, $\delta[S]$ is almost contained in $\sum_{i=1}^n \gamma_i\delta\delta_i[B]$. The latter subgroup is almost contained in $S$ by maximality of the dimension, and so $S$ is also almost $\Delta$-invariant. By absolutely $(\Delta,\Gamma)$-minimality, $S$ has dimension equal to the dimension of $A$. Therefore, given $\gamma\in \Gamma$ such that $\operatorname{Im}(\gamma)\leq L$, $\gamma[S\cap \Dom(\gamma)]$ has finite index in $L$. As usually, this implies that $\sum_{i=1}^n \gamma\gamma_i\delta_i[B]$ has finite index in $L$. Since $B\leq L$, then $\delta_i[B]\cap L={\delta_i}_L[B]$ is of finite index in $\gamma_i[B]$. Therefore, also $\sum_{i=1}^n \gamma\gamma_i[{\delta_i}_L[B]]$ is of finite index in $L$ but ${\delta_i}_L[B]\leq L$ and so $\gamma\gamma_i[{\delta_i}_L[B]]=(\gamma\gamma_i)_L\delta_L[B]$ that is almost contained in $B$ by almost $(\Delta,\Gamma)$-invariance. Therefore, $L$ is almost contained in $\gamma[S]$, which is almost contained in $B$. In conclusion, $\dim(B)=\dim(L)$, a contradiction.\\
    \textbf{Point $4$:}\\
    We take the restriction-corestriction map 
    $$P_L: \Delta/\sim\rightarrow \Delta_L/\sim.$$ 
    We prove that this map is surjective, well-defined, and injective.\\
    It is surjective since, by the previous Lemma, every element of $\Delta_L$ is equivalent to an element of $P_L(\Delta)$.\\
    It is well defined since, if $\delta\sim \delta'$, then the image of $\delta-\delta'$ is finite and so also the one of the restriction-corestriction.\\
    It is a homomorphism since, as we have already proved, $\delta_L+\delta'_L\sim (\delta+\delta')_L$, $(-\delta)_L\sim -\delta'_L$ and $(\delta\delta')_L\sim \delta_L\delta'_L$.\\
    To prove that it is injective, being a homomorphism of rings, it is sufficient to verify that if $\delta_L\sim 0$, also $\delta\sim 0$, \hbox{i.e.} there exists a subgroup $B$ of finite index of $\Dom(\psi)$ such that $\delta[B]$ is finite. By previous Lemma, there exists a finite sum $S=\sum_{i=1}^n \gamma_i[L'']$ with $\gamma_i\in \Gamma$ and $L''=L\cap \bigcap_{i=1}^n \Dom(\gamma_i)\cap \Dom(\delta_L)$ such that $S$ has finite index in $A$. To conclude, we verify that $\delta[\sum_{i=1}^n \gamma_i[L'']]$ is finite. This is equal to $\sum_{i=1}^n \delta\gamma_i[L'']$. By flat commutativity, every $\delta\gamma_i[L'']$ is equal to $\gamma_i\delta[L'']+\kat(\delta)$. Since $L''$ is contained in $L$, the subgroup $\delta[L'']=\delta_L[L'']+\kat(\psi)$ is finite. Therefore, $\delta[\Dom(\delta)]$ is finite as we want to prove. This implies that $\Delta$ is as essentially large as $\Delta_L$.\\
    It remains to prove that $\Gamma_L$ is as essentially large as $\Gamma$.\\
    Let $\{\gamma_1,...,\gamma_n\}\in \Gamma$ be such that $K=\bigcap_{i=1}^n \ker(\gamma_i)$ is of minimal dimension possible. Suppose, for a contradiction, that none of the $\gamma_i\Gamma$ is essentially as large as $\Gamma$. Recall that 
    $$\gamma:\Gamma\sim\rightarrow \gamma\Gamma/\sim$$
    is a homorphism of near-rings, as proved in Lemma \ref{Lemma 12}, and $|\Gamma/\sim|>|\gamma\Gamma/\sim|$. Then 
    $$|X_i=\{\gamma\in \Gamma:\ \gamma_i\gamma\sim 0\wedge \gamma\not\sim 0\}|=|\Gamma/\sim|>|\Gamma_{L}/\sim|.$$
    Then, at each step, we can find $|\Gamma/\sim|$ elements $\gamma$ in $X_i$ such that $\gamma\not\sim 0$. Consequently, there exists $\gamma\in \Gamma$ such that $\gamma_i\gamma\sim 0$ for any $i\leq n$ and $\gamma\not\sim 0$. By definition, each $\gamma_i\gamma[\Dom(\gamma_i\gamma)]$ is finite. Therefore, by dimensionality, the kernel of $\gamma_i\gamma$ has finite index in $A$. Consequently, also $\gamma[\ker(\gamma_i\gamma)]$ is of finite index in $\operatorname{Im}(\gamma)$. Since this holds for all $i$, also $\gamma[K]$ is of finite index in $\operatorname{Im}(\gamma)$. Since $\gamma[\Dom(\gamma_i\gamma)]$ is infinite, being not equivalent to $0$, there exists a line $L$ contained in $\gamma[\Dom(\gamma_i\gamma)]$ and $\gamma'\in \Gamma$ such that $\gamma'[L\cap \Dom(\gamma')]\subseteq L$ is of finite index in $L$. Since the dimension of $K\cap \ker(\gamma)$ is the same as $\dim(K)$ (by minimality of the dimension) and $K$ has finite index in $L$ (by previous construction), we conclude that $D=K\cap \ker(\gamma)\cap L$ is of finite index in $L$. This implies that $\gamma[L]$ has finite image since $\gamma[D]\subseteq \kat(\gamma)$ ($D\subseteq \ker(\gamma)$) and $D$ has finite index in $L$ by hypothesis. This is impossible since $L$ is infinite. In conclusion, there exists $i\leq n$ such that $\gamma_i\Gamma$ is as essentially large as $\Gamma$.\\
    Let $\gamma_0$ be such that $\gamma_0[\Dom(\gamma_0)]=L$ and $L'=\gamma'[\Dom(\gamma_i)]$ with $\gamma'=\gamma_i$. By the previous Lemma, there exists $\alpha\in \Gamma$ such that $\gamma_0\alpha[L']$ is of finite index in $L$. Therefore, $L'\cap \ker(\gamma_0\alpha)$ is finite. Suppose that for a quasi-endomorphism $\psi$, $\gamma_0\alpha\gamma'\psi\sim 0$. Then, the image is finite and so $\ker(\phi_0\alpha)$ almost contains $\gamma'\psi[\Dom(\gamma'\psi)]$. Since the intersection of these two subgroups is finite, $\gamma'\psi\sim 0$. Consequently, we may conclude that left multiplication $\gamma'\Gamma/\sim \rightarrow \gamma_0\alpha\gamma'\Gamma/\sim \subseteq \gamma_0\Gamma/\sim$ preserves the equivalence, \hbox{i.e.} it is injective and so $\gamma_0\Gamma$ is essentially as large as $\Gamma$.\\
    We conclude that $\Gamma_L$ is essentially large. By Lemma \ref{Lemma 12}, there exist $\gamma_1,...,\gamma_n$ such that $\sum_{i=1}^n \gamma_i[L"]$ with $L''=L\cap \bigcap_{i=1}^n \Dom(\gamma_i)$ is of finite index in $A$ and none of the $\gamma_i$ is equivalent to $0$. Let $J$ be a large set of indices such that $\{\psi_j:\ j\in J\}\subseteq \gamma_0\Gamma$ is a set of inequivalent pairwise endogenies (by essentially largeness of $\gamma_0\Gamma$). $\psi_j\gamma_i[\Dom(\psi_j\gamma_i)]\subseteq L$ and $(\psi_j\gamma_i){_L}\in \Gamma_L$. Moreover, since $\sum_{i\leq n}\gamma_i[L'']$ is of finite index in $A$. Then, for any couple of pairwise endogenies $\psi_j,\psi_{j'}$, there exists an $i$ such that $(\psi_j-\psi_{j'}\gamma_i[L"]$ with $L''=L\cap \Dom((\psi_j-\psi_{j'})\gamma_i)$ is infinite (if not the sum will be finite). By Ramsey's theorem, for a fixed $i$, there exists a large subset such that $\{(\psi_j\gamma_i)_L:\ j\in J'\}$ is a set of pairwise inequivalent $\mathcal{F}-\End(L)$. This completes the proof.
\end{proof}
\subsection{Theorem A: proof of base case revisited}
We divide the proof into two cases: the first in which $\Delta$ is an essentially unbounded pre-ring of endogenies. In particular, we prove the following theorem.
\begin{theorem}
    Let $A$ be an abelian infinite group, $\Delta$ a pre-ring of definable endogenies and $\Gamma$ a pre-ring of definable quasi-endomorphisms of $A$ such that:
    \begin{itemize}
        \item $\Delta$ is essentially unbounded;
        \item $\Gamma$ is essentially infinite;
        \item $A$ is absolutely $(\Gamma,\Delta)$-minimal;
        \item $\Delta$ and $\Gamma$ flat commute.
    \end{itemize}
    Then, $\kat(\Gamma,\Delta)$ is finite and $\Dom(\Gamma)$ is of finite index in $A$.
\end{theorem}
Since $\Dom(\Gamma)$ is $(\Gamma,\Delta)$-invariant, we can assume, up to work here, that both the pre-rings are of endogenies. We start from the base case.
\begin{theorem}\label{Extension Base case}
     Let $A$ be an abelian infinite group of finite dimension, $\Delta$ a pre-ring of definable endogenies and $\Gamma$ a near-ring of definable quasi-endomorphisms of $A$ such that:
    \begin{itemize}
        \item $\Delta$ is essentially unbounded;
        \item $\Gamma$ is essentially infinite;
        \item No element of $\Gamma\cup\Delta$ has infinite image and infinite kernel;
        \item $\Delta$ and $\Gamma$ flat commute.
    \end{itemize}
    Then, $\Kat(\Gamma,\Delta)$ is finite and $\Dom(\Gamma)$ is of finite index in $A$.
\end{theorem}
\begin{proof}
     The first part follows from Theorem \ref{Theorem QA}. Let $K$ be the maximum finite $\Delta$-invariant finite subgroup contained in $A$, which is also $\Gamma$-invariant by flat commutation.\\
     For the second part, suppose, for a contradiction, that there exists $\{\gamma_i\in \Gamma\}_{i\in \mathbb{N}}$ such that $\Dom(\gamma_i)>\Dom(\gamma_{i+1})$ for every $i<\omega$. Moreover, we can suppose that $\gamma_i,\gamma_j$ are not equivalent. Consequently, the near-ring $\Gamma'$ generated by $\{\gamma_i\in \Gamma\}_{i\in \mathbb{N}}$ is bounded and essentially infinite. We work in $\Dom(\Gamma')+K/K=D$ and define $\Gamma',\Delta'$ the restrictions-corestrictions of $\Gamma,\Delta$ to $D$, respectively. $\Gamma'$ and $\Delta'$ are well-defined two near-rings of quasi-endomorphisms by Lemma \ref{Lemma 14}. In particular, $\Delta'$ is a ring of endomorphisms.\\
     These two near-rings of quasi-endomorphisms inherit the properties of the two rings $\Gamma,\Delta$:
     \begin{itemize}
         \item $\Delta',\Gamma'$ commute since given $[\delta]\in \Delta'$, $[\gamma]\in \Gamma'$ and $a+K$ with $a\in \Dom(\Gamma)$ then $[\delta][\gamma]-[\gamma][\delta](a+K)\leq (\delta\gamma[a]-\gamma\delta[a])\cap \Dom(\Gamma)+K\leq K$ by flat commutation. 
         \item $\Delta',\Gamma'$ are essentially large as $\Gamma,\Delta$ respectively.\\
         Given $\delta,\delta'\in \Delta$, if the image of $\delta-\delta'[\Dom(\Gamma)+K]$ is finite, being $\Dom(\Gamma)$ of bounded index in $A$, also the image of $\delta-\delta'$ is bounded and so finite being definable. A similar proof holds for $\gamma,\gamma'\in \Gamma$.
     \end{itemize}
     $\Delta'$ acts by monomorphisms on $D$: $\ker([\delta])=\delta^{-1}[K]\cap (\Dom(\Gamma')+K)$. Since $\delta$ is an endomorphism, $K\leq \delta^{-1}[K]$ and so this coincide with $K+\Dom(\Gamma')\cap \delta^{-1}[K]$. $\Dom(\Gamma')\cap \delta^{-1}[K]$ is weakly $\Gamma'$-invariant, being intersection of a $\Gamma'$-invariant subgroup ($\delta^{-1}[K]$ by Lemma \ref{Lemma 15}) and a weakly $\Gamma'$-invariant subgroup ($\Dom(\Gamma')$ is weakly $\Gamma'$-invariant by Lemma \ref{Lemma 14}). Therefore, it is contained in $K$ by previous proof.\\
     Finally, every finite image of elements in $\Delta'$ is $0$ since $\delta(\Dom(\Gamma')+K)+K$ is finite iff $\delta[\Dom(\Gamma')+K]$ is finite. By $\Delta'$-invariance of $\Dom(\Gamma')$, 
     $$\delta[\Dom(\Gamma')+K]\leq \Dom(\Gamma')+K$$
     and
     $$\gamma\delta[\Dom(\Gamma')+K]\leq \gamma\delta[\Dom(\Gamma')]+ \delta\gamma[\Dom(\Gamma')]+K\leq \delta[\Dom(\Gamma')]+K.$$ 
     Therefore, $\delta[\Dom(\Gamma')]+K\bigcap \Dom(\Gamma')$ is weakly $\Gamma'$-invariant and so it is contained in $K$.\\
     Therefore, on $A/K$, every element of $\Delta'$ is equivalent to another iff on the common domains they coincide.\\
     $\Delta'$ is essentially unbounded since $\Delta$ is essentially unbounded. Therefore, it contains a type-definable subset of dimension strictly greater than $0$. Moreover, the dimension of each type-definable subset is bounded by the dimension of $A$ as in the proof of Theorem \ref{ThmA}. Therefore, the field of fractions $L$ of $\Delta'$ is type-definable.\\
     Working as in the proof of theorem \ref{ThmA}, there exists a definable subset $X$ of $\Delta'$ such that, for any $\delta\in \Delta'$, there are $m,n\in X$ such that the image of $[\delta]-mn^{-1}$ is finite on $\Dom(\Gamma'+K/K)$.\\
     We verify that this image is $0+K$. 
     We prove that $mn^{-1}(\Dom(\Gamma')+K/K)$ is $\Gamma'$-invariant. Let $b+K\in \operatorname{Im}(n)$ and $\gamma\in \Gamma'$. Then, there exists $a+K\in \Dom(\Gamma')+K$ such that $n(a+K)=b+K$. Up to take a suitable $a$, we may assume $a\in \Dom(\Gamma')$. Then, $\gamma [m](a+K)\leq \gamma[(m[a]+K)\cap \Dom(\gamma)]$ and $m[a]\leq \Dom(\gamma)$ by flat commutation. The latter is contained in $m\gamma[a]+K$ by $\Gamma'$-invariance of $K$. By flat commutation, since $a\in \Dom(\Gamma)$, we have that $\gamma[a]\leq \Dom(\Gamma')+\kat(\gamma)\leq \Dom(\Gamma')$. Then, $\gamma[m[a]]+K\leq m[\gamma[a]]\leq \operatorname{Im}(mn^{-1})$. In conclusion, when the image of $[m][n]^{-1}-[m'][n']^{-1}$ is finite, it is $0$ in $\Dom(\Gamma')/K$. Therefore, $[m][n]^{-1}$ and $[m'][n']^{-1}$ coincide on the common domain. As in the proof of Theorem \ref{ThmA}, the skew-field of fractions of $\Delta'$ is definable and acts on $\Dom(\Gamma')$. Therefore, $\Dom(\Gamma')$ is of finite index, contradicting the fact that there exists a subset $(\gamma_i)_{i\in \mathbb{N}}$ of $\Gamma$ such that $\Dom(\gamma_i)>\Dom(\gamma_{i+1})$. This completes the proof of the base case.
\end{proof}
\subsection{Theorem A: first case}
We prove our Theorem in case the pre-ring of endogenies $\Delta$ is essentially unbounded.
\begin{theorem}\label{ExtA general}
     Let $A$ be an abelian infinite group, $\Delta$ be a pre-ring of endogenies and $\Gamma$ a near-ring of quasi-endomorphisms of $A$ such that:
    \begin{itemize}
        \item $\Delta$ is essentially unbounded;
        \item $\Gamma$ is essentially infinite;
        \item $A$ is absolutely $(\Gamma,\Delta)$-minimal;
        \item $\Delta$ and $\Gamma$ flat commute.
    \end{itemize}
    Then, $\Kat(\Gamma,\Delta)$ is finite and $\Dom(\Gamma)$ is of finite index in $A$.
\end{theorem}
\begin{proof}
   Let $(A,\Gamma,\Delta)$ be a counterexample with $A$ of minimal dimension.\\
    If $\Gamma,\Delta$ have no lines, the contradiction follows by Theorem \ref{Extension Base case}. \\
    Suppose $\Delta$ has a line. We verify that $\Dom(\Gamma)$ is of finite index in $A$.\\
    Let $\{L_i\}_{i<n}$ be lines such that $\sum_{i=1}^n L_i$ is of finite index in $A$. For each $i$, let $\gamma_i$ be the restriction-corestriction of $\gamma$ to $L_i$ with domain 
    $$\Dom(\gamma_i)=\{l\in L_i\cap \Dom(\gamma_i):\ \gamma[l]\leq L_i+\kat(\gamma)\}=L_i\cap \Dom(\gamma).$$ 
    By Lemma \ref{Lemma 13}, $\Delta_{L_i},\Gamma_{L_i}$ respect the hypothesis of the Theorem and, by induction hypothesis, $\Dom(\Gamma_{L_i})$ is of finite index in $L_i$. Given $T_i$ a transversal of $\Dom(\Gamma_i)$ in $L_i$ and $T$ a transversal of $\sum_{i=1}^n L_i$ in $A$, the sum $\sum_{i=1}^n T_i+T$ contains a transversal of $\sum_{i=1}^n \Dom(\Gamma_i)$ in $A$. Since $\Dom(\gamma)\supseteq \sum_{i=1}^n \Dom(\gamma_i)$, the subgroup $\Dom(\Gamma)$ is of finite index in $A$.\\
    If $\Gamma$ has a line, let $L_i$ be lines such that $\sum_{i\leq n} L_i$ is of finite index in $A$ and take $\gamma_i\in \Gamma$ such that, for any $i\leq n$, $\gamma_i[A]\leq L_i$ and ${\gamma_i}_{L_i}[L_i]$ of finite index in $L_i$ (these exist by Lemma \ref{Lemma 12}). Then, $\Dom(\gamma)\geq \Dom(\gamma_i\gamma)$ for any $i$. By definition $\Dom(\gamma_i\gamma)\geq \Dom(\Gamma_{L_i})$ so, as before, $\Dom(\Gamma)$ is of finite index.\\
    Finally, to prove the finiteness of the katakernel, it is sufficient to pass to $\Dom(\Gamma)$. Indeed, $\Kat(\Gamma,\Delta)$ is finite iff $\Kat(\Gamma,\Delta)\cap \Dom(\Gamma)$ is finite. Let $\Gamma_D$ and $\Delta_D$ be the restrictions-corestrictions of $\Gamma$ and $\Delta$ to $D=\Dom(\Gamma)$. We prove that $\Kat(\Gamma_D,\Delta_D)$ is equal to $\Kat(\Gamma,\Delta)\cap D$. Let $a\in \Kat(\Gamma,\Delta)\cap D=(\kat(\gamma)+\kat(\delta))\cap D$. Since $D$ contains $\kat(\delta)$, this is equal to 
    $$\kat(\gamma)\cap D+\kat(\delta)=\kat(\gamma_D)+\kat(\delta_D)\leq \Kat(\Delta_D,\Gamma_D).$$ By arbitrariety of $a\in \Kat(\Gamma,\Delta)\cap D$, we conclude that $\Kat(\Gamma,\Delta)\cap D\leq \Kat(\Delta_D,\Gamma_D)$. The other inclusion is obvious.\\
    The rings $\Delta_D,\Gamma_D$ respect the hypothesis of the Theorem \ref{ThmA}: 
    \begin{itemize}
        \item For the essentially unboundeness and infiniteness and for the fact that they are endogenies, it follows as in the proof of Theorem \ref{Extension Base case}.
        \item For the sharp commutation observe that, given $\delta\in \Delta,\gamma\in \Gamma$ and $a\in D$, then 
        $$\delta_D\gamma_D-\gamma_D\delta_D[a]\leq (\delta\gamma-\gamma\delta[a])\cap D\leq \kat(\gamma)\cap D+\kat(\delta)=\kat(\gamma_D)+\kat(\delta_D)$$
        since $D\supseteq \kat(\delta)$.
        \item Since $D$ is a definable subgroup of $A$, a definable subgroup of $D$ is a definable subgroup of $A$. Suppose, for a contradiction, that there exists a definable infinite almost $(\Delta_D,\Gamma_D)$-invariant subgroup $B$ of $D$. Then, for any $\delta\in \Delta$, the subgroup $\delta_D[B]=\delta[B]\cap D$ is of finite index in $\delta[B]$. Since $B$ is of finite index in $\delta_D[B]$, we conclude that $B$ is almost $\Delta$-invariant. The same for $\Gamma$. This contradicts the absolutely minimality of $A$.
    \end{itemize}
    Therefore, by Theorem \ref{ThmA}, the bikatakernel is finite and the Theorem is proved.
\end{proof}
Observe that, up to work in $D=\bigcap_{\gamma\in \Gamma} \Dom(\gamma)$, we can apply Theorem $B$ to linearize the action. 
\begin{theorem}
    Let $A$ be an abelian infinite group, $\Delta$ be a pre-ring of endogenies and $\Gamma$ a near-ring of quasi-endomorphisms of $A$ such that:
    \begin{itemize}
        \item $\Delta$ is essentially unbounded;
        \item $\Gamma$ is essentially infinite;
        \item $A$ is absolutely $(\Gamma,\Delta)$-minimal;
        \item $\Delta$ and $\Gamma$ flat commute.
    \end{itemize}
\end{theorem}
Then, there exists a definable subgroup $D$ of finite index in $A$ and a finite subgroup $D_0$ such that $D/D_0$ is a $K$-vector space for a definable field $K$. Moreover, $\Delta$ and $\Gamma$ acts $K$-linearly on $D/D_0$.
\section{Theorem A: second case}
For the second case, we suppose $\Gamma$ essentially unbounded.\\
In this case, we only prove that the bikatakernel is finite. The statement of the theorem is the following.
\begin{theorem}
    Let $A$ be an abelian definable infinite group, $\Gamma$ an invariant near-ring of definable quasi-endomorphisms, and $\Delta$ an invariant pre-ring of definable endogenies. Assume that:
    \begin{itemize}
        \item $\Gamma$ is essentially unbounded and $\Delta$ is essentially infinite;
        \item $\Gamma$ and $\Delta$ flat commute;
        \item $A$ is absolutely $(\Gamma,\Delta)$-minimal.
    \end{itemize}
    Then, $\Kat(\Gamma,\Delta)$ is finite.
\end{theorem}
\begin{proof}
Let $(A,\Gamma,\Delta)$ be a counterexample of minimal dimension. By Theorem \ref{Theorem QA}, we may assume that either $\Gamma$ or $\Delta$ has a line.\\
    Suppose that $\Gamma$ has no lines. Being $\Gamma$ essentially unbounded and not having a line, $\Kat(\Delta)$ is finite. We work in $A/\Kat(\Delta)$, where $\Delta$ acts by endomorphisms. If $\Delta$ has essentially infinitely many elements with finite kernel, then the katakernel of $\Delta$ is finite, by the same proof of Theorem \ref{Theorem QA}. Therefore, we may assume that there exist only finitely many elements in $\Delta/\sim$ with finite kernel. The set $\Delta'/\sim=\{\delta:\ |\ker(\delta)|<\omega\}/\sim$ is a semigroup. Indeed, $\delta\delta'$ has finite kernel if both of the endogenies have finite kernel and $\Delta'/\sim$ contains the identity. Moreover, this semigroup has cancellation on the right. Assume $\delta_1\delta\sim \delta_2\delta$. Since $\operatorname{Im}(\delta')$ has finite index, then $\delta-\delta_1$ is finite iff $\delta\delta_1[\operatorname{Im}(\delta')]$ is finite. $\delta_1\delta\sim \delta_2\delta$ implies $\delta_1\sim \delta_2$. Therefore, $\Delta'/\sim$ is a finite semigroup with cancellation and so a group. In particular, there exists $n\in \mathbb{N}$ such that $\delta^n\sim Id$ for any $\delta\in \Delta'$. Observe that, for any $\Delta$-image not of finite index $B=\delta[A]$, we can define:
\begin{itemize}
    \item $\Delta_B$ the ring generated by $\delta_B:B\rightarrow B$ such that $\delta_B[b]=\delta[b]$ and $\delta[A]\leq B$;
    \item $\Gamma_B$ the ring generated by the restriction-corestrictions to $B$ of the elements of $\Gamma$.
\end{itemize}
 By Lemma \ref{Lemma 12}, these flat commute, are as essentially large as $\Gamma$ and $\Delta$ and, $B$ is absolutely $(\Gamma_B,\Delta_B)$-minimal. Therefore, by induction hypothesis, the bikatakernel is finite. Then, $\delta[\kat(\gamma)]\leq \Kat(\Delta_B,\Gamma_B)$. Indeed, for any $\gamma\in \Gamma$, 
$$\delta[\kat(\Gamma)]\leq \Kat(\Gamma)\cap B+\kat(\delta)\leq \Kat(\Delta_B,\Gamma_B).$$ 
Therefore, $\Kat(\Gamma)$ is almost contained in the kernel of any $\delta$ with image not of finite index. In particular, it is almost contained in a finite intersection $K$ of minimal dimension of infinite $\Delta$-kernels. This definable subgroup is of dimension strictly less than the dimension of $A$ since $\Delta$, by hypothesis, has lines. \\
We verify that $K$ is almost $\Gamma$-invariant. Let $K=\bigcap_{i=1}^n \ker(\delta_i)$, then $\gamma[K\cap \Dom(\gamma)]$ is almost contained in $K$ iff $\gamma[K]$ is almost contained in any $\ker(\delta_i)$ for any $i$. By dimensionality, this coincide to prove that $\delta_i\gamma[K\cap \Dom(\gamma)]$ is finite for every $\gamma_i$. By commutation, this is equivalent to $\gamma\delta_i[K]\leq \gamma[\kat(\delta_i)]\leq \kat(\gamma)+\kat(\delta_i)$. To prove the almost $\Delta$-invariance, we distinguish two cases: if $\delta\in \Delta$ has infinite kernel or $\delta$ has finite kernel. Assume the first. Then, $\ker(\delta)\cap K$ is of the same dimension as $K$ by minimality. Therefore, $\ker(\delta)$ almost contains $K$ and so $\delta[K]$ is finite and clearly almost contained in $K$. Assume the second \hbox{i.e.} $\delta\in \Delta$ has finite kernel. By previous proof, $\delta^n\sim 1$ for a certain $n<\omega$. We need the following easy Lemma.
\begin{lemma}
    Let $G$ be a definable finite-dimensional abelian group and $A, B,C\leq G$ definable subgroups. Assume that $C$ is finite. Then, $\dim((A+C)\cap (B+C))=\dim(A\cap B)$.
\end{lemma}
\begin{proof}
    By fibration, $\dim((A+C)\cap (B+C))=\dim((A+C)\times (B+C))-\dim(A+C+B+C)$. $\dim((A+C)\times (B+C))=\dim(A)+\dim(B)$ again by fibration, and $\dim(A+B+C)=\dim(A+B)$. Therefore, $\dim((A+C)\cap (B+C))=\dim(A)+\dim(B)-\dim(A+B)=\dim(A\cap B)$.
\end{proof}
Therefore, $\delta(\bigcap_{i=1}^n \big(\ker(\delta_i)+\ker(\delta)\big))\sim \delta[K]$ and so it is sufficient to prove that the latter is almost contained in $K$. $\delta[\bigcap_{i=1}^n \ker(\delta_i)+\ker(\delta)]$ is almost contained in $\bigcap_{i=1}^n \ker(\delta_i)+\ker(\delta)$. In particular, we can see $\delta$ as a monomorphism from $A/\ker(\delta)$ to $A$. In this case, since $\delta$ is injective
$$\delta\bigg(\bigcap_{i=1}^n \ker(\delta_i)+\ker(\delta)/\ker(\delta)\bigg)=\bigcap_{i=1}^n \delta\big(\ker(\delta_i)+\ker(\delta)/\ker(\delta)\big).$$
We prove that $\delta(\ker(\delta_i))=\ker(\delta_i\delta^{-1})$ where $\delta^{-1}$ is the quasi-endomorphisms with domain $\operatorname{Im}(\delta)$ that sends $a$ in $b+\ker(\delta)$.
Indeed, if $a\in \delta(\ker(\delta_i))$, then $\delta^{-1}[a]\leq \ker(\delta_i)+\ker(\delta)$ and so $\delta_i[\delta^{-1}[a]]\leq \delta_i(\ker(\delta))$ that is exactly the katakernel of $\delta_i\delta^{-1}$. Therefore, $\delta(\ker(\delta_i))\leq \ker(\delta_i\delta^{-1})$. By assumption on $\delta$, there exists $n<\omega$ such that $\delta^n\sim Id$. We verify that $\delta^{n-1}\sim \delta^{-1}$. Given $a\in \operatorname{Im}(\delta)$, then $\delta^{-1}[a]=b+\ker(\delta)$ with $\delta[b]=a+\kat(\delta)$. Since $\delta\delta^{n-1}[a]=a$, we have that $b-\delta^{n-1}[a]\in \ker(\delta)$ \hbox{i.e.} $(\delta^{-1}-\delta^{n-1})[a]\leq \ker(\delta)$ for every $a\in A$. Since $\ker(\delta)$ is finite, we obtain that $\delta^{-1}\sim\delta^{n-1}$. This implies that, on $\ker(\delta_i\delta^{-1})\leq A$, the image of $\delta_i\delta^{n-1}$ is finite. Consequently, $\ker(\delta_i\delta^{n-1})$ is of finite index in $\delta(\ker(\delta_i))$.\\
We verify that $\delta(\ker(\delta_i))$ almost contains $\ker(\delta_i\delta^{-1})$. Let $B=\{b\in \Dom(\delta):\ \delta(b)\in \ker(\delta_i\delta^{-1})\}$. Then, $\delta(B)=\ker(\delta_i\delta^{-1})$ since, given $a\in \ker(\delta_i\delta^{-1})$, there exists $b\in \Dom(\delta)$ such that $\delta(b)=a$, by definition. If we verify that $\ker(\delta_i)$ almost contains $B$, then $\delta(\ker(\delta_i))$ almost contains $\delta(B)=\ker(\delta_i\delta^{-1})$. It is sufficient to verify that $\delta_i(B)$ is finite. Given $b\in B$, then 
$$\delta_i(b)\leq \delta_i(b+\ker(\delta))=\delta_i(\delta^{-1}\delta[b])=\delta_i\delta^{-1}[a]$$
with $a\in \ker(\delta_i\delta^{-1})$, and so $\delta_i\delta^{-1}[a]\leq \delta_i(\ker(\delta))$. Since this image is finite, $\delta(\ker(\delta_i))$ almost contains $\ker(\delta\delta^{-1})$ and so $\delta(\ker(\delta_i))\sim \ker(\delta_i\delta^{-1})\sim \ker(\delta_i\delta^{n-1})$.\\
This implies that $\dim(K)=\dim(\delta(K))$ since $\delta$ has finite kernel and $\dim(\delta(K))=\dim(\bigcap_{i=1}^n \ker(\delta_i\delta^{n-1})$. By the minimality of the dimension 
$$\dim\bigg(\bigcap_{i=1}^n \ker(\delta_i\delta^{n-1})\cap K\bigg)=\dim(K)=\dim\bigg(\bigcap_{i=1}^n \ker(\delta_i\delta^{n-1})\bigg)$$
and so $K$ almost contains $\ker(\delta_i\delta^{n-1})$ for any $i$.\\
In conclusion, $K$ is an almost $(\Gamma,\Delta)$-invariant definable subgroup of $A$. By absolute minimality, it can only be finite or of finite index. The second case is a contradiction to the existence of $\Delta$-lines. Therefore, $K$ is finite and $\Kat(\Gamma)$ is finite since $\Kat(\Gamma)$ is almost contained in $K$. The proof in this case is completed.\\
We may assume that $\Gamma$ has at least a line $L=\gamma[A]$. By induction hypothesis, $L_0=\Kat(\Gamma_L,C^{\flat}_{\Endog}(\Gamma_L))$ is finite where $C^{\flat}_{\End}(\Gamma_L)$ is the pre-ring of endogenies that flat commutes with $\Gamma_L$. Moreover, any finite $\Gamma_L$ invariant subgroup contained in $\Dom(\Gamma)$ is contained in $L_0$. Let $\gamma\in \Gamma$ be such that $\gamma[A]\leq L$ and $\gamma[L]$ is of finite index in $L$. We define 
$$\pi:a\mapsto \gamma_L^{-1}\gamma[a]+\gamma^{-1}[L_0]$$
with domain $\gamma^{-1}(\Dom(\gamma_L))$ and katakernel $\gamma^{-1}[L_0]$.
This map is clearly a quasi-endomorphism. We prove that it flat commutes with every $\delta\in \Delta$. It is sufficient to verify that $\delta[\Dom(\pi)]\leq \Dom(\pi)$ and $\pi\delta-\delta\pi[A]\leq \kat(\pi)+\kat(\delta)$. For the first, $a\in \Dom(\pi)$ iff $a\in \Dom(\gamma)$ and $\gamma[a]\leq \Dom(\gamma_L)+\kat(\gamma)$. Given $a\in \Dom(\pi)\cap \Dom(\delta)$, then, $\delta[a]\leq \Dom(\pi)$ iff $\gamma\delta[a]\leq \Dom(\gamma_L)+\kat(\gamma)$.\\
$$\gamma\delta[a]\leq(\delta\gamma[a])\cap L+\kat(\gamma)\leq \delta_L\big(\Dom(\gamma_L)\cap \Dom(\delta)+\kat(\gamma)\big)+\kat(\gamma).$$
By flat commutation, the latter is contained in 
$$\Dom(\gamma_L)+\kat(\gamma)+\kat(\delta_L)\leq \Dom(\gamma_L)+\kat(\gamma)$$
since the domain of $\gamma_L$ contains $\kat(\delta)$.\\
The second is equivalent to proving that
$$\gamma_L[\delta_L\pi-\pi\delta][A]\leq L_0.$$
By definition, $\delta_L\pi[A]\leq \Dom(\gamma_L)+\kat(\gamma)$ and same for $\pi\delta[A]$. Fixed $a\in \Dom(\delta_L\pi-\pi\delta)$, the first term, by distributivity, is equal to 
$$\gamma_L\delta_L[\gamma_L^{-1}\gamma[a]+\gamma_L^{-1}[L_0]].$$
This is in $\Dom(\gamma_L)$ and so we can apply flat commutation, obtaining that
$$\delta_L[\gamma_L\gamma_L^{-1}[\gamma[a]]]+L_0\leq \delta_L[\gamma[a]+L_0]]=\delta_L[\gamma[a]]+L_0.$$
The second term is of the form 
$$\gamma_L[\gamma_L^{-1}\gamma\delta[a]+\gamma^{-1}_L[L_0]]\leq \gamma\delta[a]+L_0.$$
But $\gamma\delta[a]\leq \delta_L\gamma[a]+\kat(\gamma)$ and the proof is completed.\\
Then, we proceed as in Theorem \ref{Theorem 2A} and obtain that the bikatakernel is finite.
\end{proof}
I have not yet proved that we can linearize the action in this case. The main problem is the construction of a definable field. In particular, we do not have a bound on the type-definable subsets of the ring of fractions generated by quasi-endomorphisms, since we do not have, in general, an element that is contained in the domain of any quasi-endomorphism.
\section{Characteristic 0 case}
In this section, we analyze the action of an essentially unbounded near-ring $\Gamma$ of definable quasi-endomorphisms acting on an abelian definable $\Gamma$-minimal group $A$ such that $A[n]=\{a\in A:\ na=0\}$ is finite for any $n<\omega$. In this case, there naturally exists an essentially infinite ring of endomorphisms that commutes with $\Gamma$: $\mathbb{Z}$. Consequently, there is no need to assume the existence of a second near-ring of quasi-endomorphisms. We verify that either $A$ is virtually connected, and the proof follows by the two Theorems of \cite{WagnerDeloro}, or $A$ has an infinite definable subgroup on which $\Gamma$ acts "almost trivially".
\begin{defn}
    Let $A$ be a definable abelian group and $\Gamma$ a near-ring of definable quasi-endomorphisms of $A$. $\Gamma $ acts \emph{almost trivial on $B$} with $B$ a definable almost $\Gamma$-invariant subgroup of $A$ if $\Gamma_B/\sim$ is not as essentially large as $\Gamma$.
\end{defn}
\begin{lemma}\label{cha0}
    Let $\Gamma$ be an invariant near-ring of definable quasi-endomorphisms of $A$. Assume that:
    \begin{itemize}
        \item $\Gamma$ is essentially unbounded;
        \item $A$ is weakly $\Gamma$-minimal;
        \item $A[n]$ is finite for any $n\in \mathbb{N}$.
    \end{itemize}
    Then, either $A$ has an almost $\Gamma$-invariant subgroup on which $\Gamma$ acts almost trivially or $A$ is virtually connected, and we can apply Theorem $B$ of \cite{WagnerDeloro} to $A^0$.
\end{lemma}
\begin{proof}
The ring $\mathbb{Z}$ acts on $A$ by endomorphisms, and it is essentially infinite (since the $n$-torsion is finite for any $n$).\\
Assume that $\Gamma$ has no line. Then, $T(A)=\bigcup_{n<\mathbb{N}}A[n]$ is finite by the proof of Theorem \ref{Theorem Ab}. Moreover, since $\kat(\gamma)\leq T(A)$ for any $\gamma\in \Gamma$, the torsion $T(A)$ is $\Gamma$-invariant. Taking $A/T(A)$, we may assume $T(A)=0$. Therefore, any quasi-endomorphism of $A$ is a partial endomorphism of $A/T(A)$. Moreover, any element $\gamma\in \Gamma$ with finite image has $0$-image, and any finite kernel is $0$. This implies that $\Gamma/\sim$ acts by partial endomorphisms with trivial kernel. Since any $nA$ is $\Gamma$-invariant, $\Gamma/\sim$ acts on $A_1=\bigcap_{n<\omega} nA$. By dimensionality, $A_1$ is a type-definable subgroup of bounded index in $A$. Moreover, it is divisible. Let $a\in A_1$ and $n\in \omega$. Then, for every $m\in \mathbb{N}$, $a\in A_{nm}$ and so there exists $y_m\in A_m$ such that $n(y_m)=a$. This implies that $n(y_m-y_{m'})=0$ for any $m,m'\in \omega$. Since the $n$-torsion is $0$, $y_m=y_{m'}$ for any $m,m'\in \mathbb{N}$. Therefore, $y=y_m$ is an element of $A_1$ such that $ny=a$. By arbitrariety of $n$ and $a$, we conclude that $A_1$ is divisible. In particular, $A_1$ has no subgroups of finite index. Therefore, $\Gamma/\sim$ acts on $A_1$ by definable automorphisms. Therefore, the ring generated by this set $X$ of partial endomorphisms is an Ore domain. Moreover, the dimension of any type-definable subset is bounded by the dimension of $A$. By Proposition 3.3 of \cite{wagner2020dimensional}, the skew-field of fractions $K$ generated by $X$ is definable. By the identification used also in Theorem \ref{Theorem 2A}, the skew-field acts on $A_1$ by automorphisms. Therefore, $A_1$ is definable, being a vector space of finite dimension over a definable skew-field. This implies that $A_1$ is of finite index in $A/T(A)$. Consequently, $A$ is virtually connected.\\
Assume $\Gamma$ has a line $L$. By Lemma \ref{Lemma 12}, $\Gamma_L$ has no lines. Moreover, either $\Gamma_L$ is essentially unbounded or $\Gamma_L$ is essentially infinite. In the second case, taken a finite sum of $\gamma$-translates of $L$ of maximal dimension, this is an almost $\Gamma$-invariant on which $\Gamma$ acts almost trivially. Assume the second. By previous proof, $T(L)$ is finite and $L$ is virtually connected. Denote by $B$ a sum of $\Gamma$-translates of $L$ of maximal dimension. $B$ is a virtually connected almost $\Gamma$-invariant subgroup, by maximality of the dimension. Therefore, $B^0$ is $\Gamma$-invariant. By minimality of $A$, the dimension of $B$ is equal to $\dim(A)$ and so $A$ is virtually connected. 
\end{proof}
\section{Zilber's Field Theorem}
We prove a version of Zilber's Field Theorem in the finite-dimensional context. As already observed, it is possible to linearize the action of an abelian group on a minimal abelian group in the finite Morley rank case \cite[Theorem 3.7]{poizat} and in the $o$-minimal ones \cite[Theorem 2.6]{peterzil2000simple}. We would like to extend the previous results to the general finite-dimensional case (a version for connected finite-dimensional groups follows from \cite{WagnerDeloro}). This will follow from Theorem $B$ and from the fact that a minimal $G$-invariant subgroup is absolutely $G$-minimal. The latter is a consequence of Schilisting's Theorem \cite[Theorem 4.2.4]{wagner2000simple}.
\begin{theorem}
    Let $G$ be a group and $\{G_i\}_{i\in I}$ a uniformly commensurable family of subgroups in $G$. Then, there exists a subgroup $H$ in $G$ commensurable with every $G_i$ and invariant under all the automorphisms that fix set-wise the family $\{G_i\}_{i\in I}$. Moreover, if $G$ is definable and $\{G_i\}_{i\in I}$ is an uniformly commensurable family of definable subgroups in $G$, also $H$ is definable.
\end{theorem}
We introduce the notion of almost centraliser of an action.
\begin{defn}
    Let $G$ be a group acting on a group $A$. The \emph{almost centraliser of the action in $G$}, denoted by $\widetilde{C}_G(A)$, is the subgroup $\{g\in G:\ C_A(g)\sim A\}$. The \emph{almost centraliser of the action in $A$}, denoted by $\widetilde{C}_A(G)$, is the subgroup $\{a\in A:\ C_G(a)\sim G\}$.
\end{defn}
It follows from Lemma \ref{boundedind} that, in a finite-dimensional theory, both the two almost centralisers are definable if $G$ and $A$ are definable in a finite-dimensional theory. Moreover, $\widetilde{C}_G(A)$ is normal in $G$, while $\widetilde{C}_A(G)$ is normal and $G$-invariant if $A$ is abelian
\begin{lemma}\label{lemma 1}
     Let $G$ be a group acting on a group $A$. Then,
     \begin{itemize}
         \item $\widetilde{C}_G(A)$ is normal in $G$;
         \item $\widetilde{C}_A(G)$ is $G$-invariant.
     \end{itemize}
\end{lemma}
\begin{proof}
    Let $g\in G,g'\in \widetilde{C}_G(A)$. Then 
    $$(g'gg'^{-1}-1)(A)=g'(g-1)(gA)=g'Im(g-1)$$
    is finite.\\
    Let $a\in \widetilde{C}_A(G)$ and $g\in G$. Let $a\in \widetilde{C}_A(G)$, then $Ga$ is finite. This implies that $Gga\subseteq Ga$ is finite.
\end{proof}
We verify that $\widetilde{C}_G(A)$ is of finite index in $G$ iff $\widetilde{C}_A(G)$ is of finite index in $A$.
\begin{lemma}\label{SimAlmCenAct}
    Let $G$ be a group acting on the group $A$, both definable in a finite-dimensional theory. Then, $\widetilde{C}_A(G)$ is of finite index in $A$ iff $\widetilde{C}_G(A)$ is of finite index in $G$. In this case, the action of $G$ is \emph{almost trivial}.
\end{lemma}
\begin{proof} 
    Both $\widetilde{C}_G(A)$ and $\widetilde{C}_A(G)$ are definable by Lemma \ref{boundedind}.\\
    We prove that, if $\widetilde{C}_G(A)$ is of finite index in $G$, then $\widetilde{C}_A(G)$ is of finite index in $A$. The other implication follows by symmetry.\\
    Let $X$ denote the definable subset
    $$=\{(a,g)\in A\times G: g(a)=a\}.$$
    Let $\pi_2$ be the projection to the second factor. Then, $\pi_2^{-1}(\widetilde{C}_G(A))$ is a definable subset of $X$. By fibration, 
    $$\dim(\pi_2^{-1}(\widetilde{C}_G(A)))=\dim(\widetilde{C}_G(A))+\dim(A)$$
    since, given $g\in \widetilde{C}_G(A)$, the dimension of $\{a\in A:\ g(a)=a\}$ is equal to $\dim(A)$. Therefore, $\dim(X)=\dim(A)+\dim(G)$. Since $\widetilde{C}_A(G)$ is definable, also its complement $A'$ in $A$ is definable. Given $\pi_1$ the projection from $X$ to $A$ then, by union,
    $$\dim(X)=\max\{\dim(\pi_1^{-1}(A')),\dim(\pi_1^{-1}(\widetilde{C}_A(G))\}.$$
    By fibration, the first set has dimension less than or equal to $\dim(A)+\dim(G)-1$. Indeed, all the fibers $\pi_1^{-1}(a)$ for $a\in A'$ have dimension strictly less than $\dim(G)$. Therefore,
    $$\dim(A)+\dim(G)=\dim(\pi^{-1}(\widetilde{C}_A(G)))\leq \dim(\widetilde{C}_G(A))+\dim(G).$$
    Consequently, $\dim(A)=\dim(\widetilde{C}_A(G))$.
\end{proof}
The now prove Zilber's Field Theorem for finite dimensional groups.
\begin{theorem}\label{Zilber}
    Let $G$ be an infinite abelian group acting by automorphisms on an infinite abelian group $A$, both definable in a finite-dimensional theory. Assume that:
    \begin{itemize}
        \item $A$ is $G$-minimal;
        \item The action is not almost trivial.
    \end{itemize}
    Then, there exists a finite subgroup $A_0$ and a definable field $K$ such that $A/A_0\simeq K^+$ and $G/G_0$ embeds in $K^{\times}$.
\end{theorem}
\begin{proof}
It is sufficient to prove that $A$ is absolutely $G$-minimal. Then, the proof follows by Theorem $B$ (since $G$ commutes with itself and it is essentially unbounded by hypothesis). Therefore, $A/A_0$ is a $K$-vector space for $A_0$ a finite subgroup. Since a vector subspace is $G$-invariant and, by minimality, of finite index in $A$, $A/A_0\simeq K^+$ and $G/G_0$ embeds in $K^{\times}$.\\
Assume that $A$ is not absolutely $G$-minimal. Let $B$ be an almost $G$-invariant subgroup such that $\dim(A)>\dim(B)$. The uniformly definable family of subgroups $\{g\cdot B\}_{g\in G}$ is commensurable. By compactness, it is uniformly commensurable, and, by Schlisting's Theorem, there exists $H\sim B$ such that $H$ is $G$-invariant. This contradicts $G$-minimality.
\end{proof}
We extend the previous theorem to the action of an almost abelian group on another almost abelian group. This is not just a mere technical extension, since we usually encounter almost abelian groups while working with finite-dimensional groups.
\begin{defn}
    A group $G$ is \emph{almost abelian} if $\widetilde{Z}(G):=\widetilde{C}_G(G)=G$.
\end{defn}
By Lemma \ref{boundedind}, finite-dimensional almost abelian groups are BFC.
\begin{defn}
    A group $G$ is a \emph{BFC-group} if there exists $n\in \mathbb{N}$ such that $|g^G|\leq n$ for any $g\in G$.
\end{defn}
BFC-groups are abelian-by-finite as proven in \cite{neumann1954groups}.
\begin{lemma}\label{BFC-groups}
    Let $G$ be a BFC-group. Then, $G'$ is finite.
\end{lemma}
Therefore, almost abelian finite-dimensional groups have finite derived subgroup. 
 In this case, the minimality is substituted by the stronger notion of \emph{absolute minimality}. 
\begin{defn}
    Let $G$ be a definable group acting on the definable abelian group $A$. $A$ is \emph{absolutely $G$-minimal} if $A$ is $H$-minimal for any definable subgroup $H$ of finite index in $G$. 
\end{defn} 
The proof of the following theorem follows, partially, from \cite[Theorem 2.4]{wagner2023finite}.
\begin{theorem}
    Let $G$ be a definable almost abelian group acting on a definable infinite almost abelian group $A$. Assume that:
    \begin{itemize}
        \item $A$ is absolutely $G$-minimal;
        \item the action is not almost trivial.
    \end{itemize}
    Then, $G/\widetilde{C}_G(A)$ is abelian and there exists a definable field $K$ such that $G/\widetilde{C}_G(A)$ definably embeds in $K^{\times}$ and $A/\widetilde{C}_A(G)$ is isomorphic to $K^{+}$.
\end{theorem}
\begin{proof}
    $\widetilde{C}_A(G):=A_0$ and $\widetilde{C}_G(A):=G_0$ are definable normal subgroups respectively of $A,G$ by Lemma \ref{lemma 1}. Moreover, they are not of finite index since the action is not almost trivial and by Lemma \ref{SimAlmCenAct}. By absolute $G$-minimality, $\widetilde{C}_A(G)$ is finite. We define the action of $G/G_0$ on $A/A_0$ simply as $[g][a]=[g(a)]$.\\
    We prove that this action is well-defined. Since $A_0$ is $G$-invariant by Lemma \ref{lemma 1}, $g(a)\in A_0$ for every $g\in G$ and $a\in A_0$.
    We may work in $A/A_0$. This group is abelian since, by Lemma \ref{BFC-groups}, the derived subgroup $A'$ of $A$ is a finite characteristic subgroup. In particular, it is $G$-invariant and, for any $a\in A'$, $Ga$ is finite. This clearly implies that $a\in A_0$. We verify that $\widetilde{C}_{A/A_0}(G)=0+A_0$. Let $a+A_0\in A/A_0$ be such that $Ga+A_0$ is finite in $A/A_0$. Since $A_0$ is finite, $Ga$ is itself finite \hbox{i.e.} $a\in A_0$. To prove the well-defineness, it is sufficient to show that $G_0$ acts as the identity on $A/A_0$ \hbox{i.e.} for any $g_0\in G_0$, $(g_0-1)A+A_0\leq A_0$. $(g_0-1)A+A_0=A_1$ is  a finite $G_1:=C_G(g_0)$-invariant subgroup. By almost abelianity, $G_1$ is a definable subgroup of finite index in $G$ and so $A_1\leq A_0$. The previous proof also verifies that there are no finite subgroups of $A/A_0$ invariant for a definable subgroup of finite index in $G/G_0$.\\
    Let $M=\langle G/G_0\rangle$ be the invariant domain of endomorphisms of $A/A_0$ generated by $G/G_0$. We prove that every endomorphism is either a monomorphism or $0$. Let $m\in M$. Then, $m$ is generated by finitely many elements $g_1,...,g_n\in G/G_0$. Therefore, $\mathrm{im}(m)$ and $\ker(m)$ are $\bigcap_{i=1}^n C_G(g_i):=G_1$-invariant. Since the latter is a definable subgroup of finite index, $\mathrm{im}(m)$ and $\ker(m)$ are either finite or of finite index in $A/A_0$, by absolute minimality. Being both invariant for $G_1$-invariant, then, if they are finite, they are equal to $0$. This implies that either $m$ is a monomorphism or it is equal to $0$. Consequently, $M$ is an integral domain. Indeed, if the kernel of $m$ and $m'$ are $0$, then $\ker(mm')$ is equal to $\ker(m')+m'^{-1}(\ker(m))$ that is $0$. The dimension of any type-definable subset of $M$ is smaller than $2\dim(A)$. Indeed, fixed $c\in A/A_0-\{0+A_0\}$, every $x\in X$ is uniquely determined by the couple $(c,x(c))\in (A/A_0)^2$. By \cite[Proposition 3.6]{wagner2020dimensional}, the skew-field of fractions $K$ is a definable infinite skew-field. Moreover, we can define a function $\phi$ between $K$ and $\mathcal{F}-\operatorname{End}(A)/\sim$. This function associate to $m/n$ for $m,n\in M,n\not=0$, the equivalence class $\phi(m/n)=[mn^{-1}]$ where $mn^{-1}$ is the quasi-endomorphism 
    $$mn^{-1}: \mathrm{im}(n)\rightarrow A/A_0$$
    that sends $n(a)$ to $m(a)$. Proceeding as in Theorem $A$, we may conclude that this skew-field of fractions acts definably on $A/A_0$, that is a $K$-vector space of finite dimension. Since $G/G_0$ is almost abelian, $(G/G_0)'$ is finite and let $n_1=|(G/G_0)'|$. Therefore, $C_{G/G_0}((G/G_0)')$ is of finite index $n_2$. Since $G/G_0$ is a group embeddable in the multiplicative group of a skew-field, there are only finitely many elements of order $n$ for any $n<\omega$. We verify that $G^{n_1n_2}$ is a central subgroup in $G$. It is sufficient to prove that $[g,h^{n_1n_2}]=0$ for any $g,h\in G$. Clearly, $[g,h^{n_1n_2}]=[g,h^{n_2}]^{n_1}=0$ since $h^{n_2}\in C_G((G/G_0)')$. Therefore, the definable normal subgroup $G^{n_1n_2}/G_0$ is of finite index in $G/G_0$ and it is central. The field of fractions $K_1$ of the ring of automorphisms generated by $G^{n_1n_2}/G_0$ acts on $A/A_0$. Consequently, $A/A_0$ is a $K_1$-vector space of finite dimension. Suppose, for a contradiction, that the linear dimension on $A/A_0$ as a $K_1$-vector space is not $1$. Then, given $a+A_0\in A/A_0-\{0\}$, $K_1a$ is an infinite definable $G^{n_1n_2}$-invariant subgroup. By absolute minimality, it must be of same dimension as $A$. In conclusion, $A/A_0$ is isomorphic to $K_1^{+}$. This implies that $\dim(K_1)=\dim(K)$. Since $K$ must be a vector space over $K_1$, $K_1=K$. Therefore, $(G/G_0)'$ is abelian. Moreover, $A/A_0$ has no proper $G/G_0$-invariant subgroups. Suppose, for a contradiction, that there exists a proper $G$-invariant subgroup $A_1/A_0$ of finite index in $A/A_0$. Then, it is a subvector space. Indeed, the action of $K$ on $A_1/A_0$ defined as $mn^{-1}(a)$ with $a\in A'/A_0$ is the same. Working as in the previous proof, we obtain the conclusion. Therefore, $A'$ is an infinite vector space of at least dimension $1$. Since $A/A_0$ has dimension $1$, $A'$ must coincide with $A/A_0$. Then, the conclusion follows from Theorem \ref{Zilber}.
\end{proof}
\section{Conclusion}
A first important question is whether Theorem $A$ and Theorem $B$ are still true with the hypothesis of minimality. In general, this seems an unnatural assumption. Indeed, in the cases in which it holds (as \cite{WagnerDeloro} and \ref{Zilber}), it is a consequence of the connectivity.
\begin{question}
    Are Theorems $A$ and $B$ still true when we only assume minimality?
\end{question}
Another possible extension of Theorem $A$ and Theorem $B$ is when both $\Gamma$ and $\Delta$ are near-rings of definable quasi-endomorphisms. The problem for this case is that, in general, is not clear how to construct a definable field when we do not have a bound on the indices, as already described at the end of section $10$.
\begin{question}
    Are Theorems $A$ and $B$ still true when we assume that both $\Gamma$ and $\Delta$ are near-rings of quasi-endomorphisms?
\end{question}
Finally, we may ask if Theorem $A$ and/or $B$ holds for two pre-rings of quasi-endomorphisms with a notion of commutation weaker than the flat one. For example, we can define sharp commutation for quasi-endomorphisms without assuming anything on the domains. This notion still extends the sharp commutation for endogenies. Nevertheless, it leads to many obstacles, as the non-$(\Gamma,\Delta)$-invariance of the bikatakernel. Find a result in this direction requires a completely different approach.
\begin{question}
    Are Theorems $A$ and $B$ still true if we assume a notion of commutation weaker than flat commutativity?
\end{question}
The main directions of research for these results are the possible applications. Already Theorem \ref{Zilber} leads the way to the study of finite-dimensional groups, and in particular supersimple and superstable of finite Lascar rank, in a direction that resembles the finite Morley rank approach, with potentially interesting application. For example, it can be applied to classify superstable soluble groups of Lascar rank $2$. The main issue for supersimple groups is that is not knows if a supersimple definably minimal group of finite Lascar rank should be finite-by-abelian-by-finite. Nevertheless, the strength of Theorem $B$ is not limited to groups. Another application is to Lie rings. A \emph{Lie ring} is simply an abelian group with a bilinear anti-symmetric map 
$$[\_,\_]:\mathfrak{g}\times \mathfrak{g}\to \mathfrak{g}$$
 that respects the Jacobi identity. In \cite{invitti2025lie}, Theorem $B$ is used to linearize the action of a Lie ring on an absolutely minimal module. This leads to a classification of NIP finite-dimensional Lie rings, which resembles the result in \cite{deloro2023simple}.

\end{document}